\documentclass[11pt]{amsart}
\usepackage{amsmath,amsfonts,amsbsy,amsgen,amscd,mathrsfs,amssymb,amsthm}
\usepackage{enumerate,mathtools}
\usepackage{tabularx}
\usepackage{microtype}
\usepackage{array}
\newcolumntype{P}[1]{>{\centering\arraybackslash}p{#1}}
\usepackage{makecell}

\usepackage{stmaryrd}

\usepackage{fullpage}
\usepackage{url}
\usepackage{mathrsfs}
\usepackage{float}
\usepackage{bbm}
\usepackage{quiver}

\usepackage[colorlinks=true,allcolors=blue]{hyperref}

\usepackage[
backend=biber,
style=alphabetic,
maxalphanames=5,
maxnames=5
]{biblatex}
\makeatletter
\renewcommand*{\eqref}[1]{%
  \hyperref[{#1}]{\textup{\tagform@{\ref*{#1}}}}%
}
\makeatother

\usepackage{tikz}
\usetikzlibrary{shadings}

\numberwithin{equation}{section}

\newtheorem{theorem}{Theorem}[section]
\newtheorem{proposition}[theorem]{Proposition}
\newtheorem{question}[theorem]{Question}
\newtheorem{conjecture}[theorem]{Conjecture}
\newtheorem{corollary}[theorem]{Corollary}
\newtheorem{lemma}[theorem]{Lemma}
\theoremstyle{definition}
\newtheorem{definition}[theorem]{Definition}

\newtheorem{example}[theorem]{Example}
\newtheorem{remark}[theorem]{Remark}

\newtheorem*{claim*}{Claim}

\newcommand{\A}{\mathbb{A}}
\newcommand{\C}{\mathbb{C}}
\newcommand{\E}{\mathbb{E}}
\newcommand{\F}{\mathbb{F}}
\newcommand{\G}{\mathbb{G}}

\newcommand{\Q}{\mathbb{Q}}
\newcommand{\R}{\mathbb{R}}

\newcommand{\Z}{\mathbb{Z}}

\newcommand{\calA}{\mathcal{A}}

\newcommand{\calM}{\mathcal{M}}

\DeclareMathOperator{\aff}{aff}

\DeclareMathOperator{\Aut}{Aut}

\DeclareMathOperator{\Char}{char}

\DeclareMathOperator{\diag}{diag}
\DeclareMathOperator{\disc}{disc}

\DeclareMathOperator{\End}{End}

\DeclareMathOperator{\Frac}{Frac}
\DeclareMathOperator{\Frob}{Frob}
\DeclareMathOperator{\Gal}{Gal}

\DeclareMathOperator{\Hom}{Hom}

\DeclareMathOperator{\id}{id}

\DeclareMathOperator{\inv}{inv}
\DeclareMathOperator{\Jac}{Jac}

\DeclareMathOperator{\Li}{Li}

\DeclareMathOperator{\Mat}{Mat}
\DeclareMathOperator{\mult}{mult}

\DeclareMathOperator{\ord}{ord}

\DeclareMathOperator{\rk}{rk}

\DeclareMathOperator{\Spec}{Spec}

\DeclareMathOperator{\tr}{tr}
\DeclareMathOperator{\Tr}{Tr}

\DeclareMathOperator{\Vol}{Vol}

\newcommand{\et}{{\operatorname{et}}}

\newcommand{\vphi}{\varphi}

\newcommand{\GL}{\operatorname{GL}}

\newcommand{\GSp}{\operatorname{GSp}}

\newcommand{\Sp}{\operatorname{Sp}}

\newcommand{\ST}{\operatorname{ST}}

\newcommand{\USp}{\operatorname{USp}}
\usepackage[shortlabels]{enumitem}

\newcommand{\injects}{\hookrightarrow}

\newcommand{\tensor}{\otimes} 

\usepackage[OT2,T1]{fontenc}
\DeclareSymbolFont{cyrletters}{OT2}{wncyr}{m}{n}
\DeclareMathSymbol{\Sha}{\mathalpha}{cyrletters}{"58}

\newcommand{\bb}{\mathbb}
\newcommand{\ovl}{\overline}
\usepackage{dsfont}
\usepackage[all]{xy}

\newcommand{\val}{\operatorname{val}}

\newcommand{\cO}{\mathcal{O}}

\newcommand{\fg}{{\mathfrak g}}

\newcommand{\cM}{\mathcal {M}}

\newcommand{\Gm}{{\mathbb{G}}_m}

\newcommand{\stab}{\mathrm{stable}}

\def\aff{{\mathbb A}}
\def\ff{{\mathbb F}}
\def\rat{{\mathbb Q}}

\def\gp{{\mathbb G}}

\def\inv{^{-1}}

\def\tensor{\otimes}
\def\iso{\cong}

\newcommand{\half}[1]{\frac{#1}{2}}
\newcommand{\abs}[1]{{\left|#1\right|}}
\newcommand{\st}[1]{\left\{#1\right\}}
\def\integ{{\mathbb Z}}
\def\ra{\rightarrow}

\DeclareMathOperator{\vol}{vol}
\DeclareMathOperator{\mat}{Mat}
\DeclareMathOperator{\gl}{GL}
\DeclareMathOperator{\gsp}{GSp}

\DeclareMathOperator{\gal}{Gal}

\def\der{\textrm{der}}
	
\def\geom{\mathrm{geom}}

\newcommand{\fc}{{\mathfrak c}}

\numberwithin{figure}{section}

\title{On the pointwise convergence of the number of abelian varieties over $\bb{F}_p$ with fixed trace}

\author{Zhao Yu Ma and Jit Wu Yap}
\address{Princeton University, Princeton, NJ, USA}
\email{zm5336@princeton.edu}
\address{Massachusetts Institute of Technology, Cambridge, MA, USA}
\email{jitwuyap@mit.edu}

\makeatletter
\let\@wraptoccontribs\wraptoccontribs
\makeatother
\contrib[with an appendix by]{Jeff Achter and Julia Gordon}
\begin{document}
\begin{abstract}
Extending Katz--Sarnak heuristics, Ballini--Lombardo--Verzobio \cite{BLV} conjectures a limiting distribution as $p \to \infty$ for $\# \calA_g(\bb{F}_p,t)$, the number of $g$-dimensional PPAVs over $\bb{F}_p$ with trace $t$, as a product of natural local factors $v_{l}(t)$ for non-archimedean places $l$ and the Sato-Tate measure $\ST_g$ corresponding to $\infty$. We prove that their conjecture is true for all $g$, i.e. 
$$\lim_{p \to \infty} \sum_{t \in \bb{Z}}\left|\frac{\# \calA_{g}(\bb{F}_p,t)}{\# \calA_g(\bb{F}_p)} - \frac{1}{\sqrt{p}} \ST_g(t/\sqrt p) \prod_{l \text{ prime }} v_{l}(t) \right| = 0.$$
As a consequence, we obtain analogous results on the distribution of curves of genus $2$ and $3$, answering questions of Bergström--Howe--García--Ritzenthaler \cite{refinements} and \cite{BLV}. 
\end{abstract}

\maketitle


\hypersetup{linktocpage=true}
\tableofcontents
\section{Introduction}
Let $C$ be a genus $g \geq 2$ curve over the prime field $\bb{F}_p$. As $C$ varies over the family of genus $g$ curves, Katz--Sarnak \cite{KS99} proved a limiting law for the distribution of $\#C(\bb{F}_p)$ as $p \to \infty$. More precisely, let $\mathcal M_g(\F_p)$ be the set of genus $g$ curves over $\F_p$, where we count curves in $\mathcal M_g(\F_p)$ with weight $\frac{1}{\# \Aut_{\bb{F}_p}(C)}$. Let $t_p(C) = \# C(\bb{F}_p) - (p+1)$ be the trace of a curve and let its normalization be $x_p(C) = t_p(C)/\sqrt{p}$. By the Hasse--Weil bound, we have  $|x_p(C)| \leq 2g$. Let $\mu_{\calM_g(\bb{F}_p)}$ be the probability distribution of $x_p(C)$ for a randomly chosen $C \in \calM_g(\bb{F}_p)$. 

\begin{theorem}[{Katz-Sarnak for $\mathcal M_g$ \cite[Theorem 10.7.12, 10.8.2]{KS99}}]\label{thm: katz-sarnak}
Let $g\ge 2$ be an integer and let $\USp_{2g} = \Sp_{2g}(\C) \cap U(2g)$. Let the trace map $\tr\colon \USp_{2g} \rightarrow [-2g,2g]$ send a matrix in the unitary symplectic group to its trace. Define the Sato-Tate measure $\mu_{\ST_g}=\tr_* \mu_{\USp_{2g}}$ to be the pushforward of the Haar measure on $\USp_{2g}$ along the trace map. Then, $$\mu_{\mathcal M_g(\F_p)} \xlongrightarrow{d} \mu_{\ST_g} $$ converges in distribution as $p\rightarrow \infty$. 
\end{theorem}

The integer $t_p(C)$ is the trace of the Frobenius action on the first etale cohomology group $H^1_{\et}(C_{\ovl{\bb{F}}_p}, \ovl{\bb{Q}}_l)$ for any $l \not = p$. Due to the Weil pairing, the Frobenius acts by a matrix $M \in \GSp_{2g}(\ovl{\bb{Q}}_l)$ with multiplier $p$. The underlying heuristic is then that $M$ should be randomly distributed according to the Haar measure. The closely related Sato-Tate problem, where we instead fix $C$ and vary $p$, has very similar heuristics and one can refer to \cite{Sutherland_2019} for more details.
\par 
There is an analogous theorem for $\mathcal A_g(\bb{F}_p)$, the set of $g$-dimensional principally polarized abelian varieties (PPAVs) over $\bb{F}_p$ with trace $t$, where we again count with weight $\frac{1}{\# \Aut_{\bb{F}_p}(A)}$. Define the normalized trace $x_p(A) = t_p(A)/\sqrt{p}$ where $t_p(A)$ is the trace of the Frobenius acting on $A$ and let $\mu_{\calA_g(\bb{F}_p)}$ be the distribution of $x_p(A)$ for a randomly chosen $A \in \calA_g(\bb{F}_p)$.

\begin{theorem}[{Katz-Sarnak for $\mathcal A_g$ \cite[Theorem 11.3.10]{KS99}}]\label{thm: katz-sarnak A_g}
Let $g\ge 1$ be an integer. Then, $$\mu_{\mathcal A_g(\F_p)} \xlongrightarrow{d} \mu_{\ST_g} $$ converges in distribution as $p\rightarrow \infty$. 
\end{theorem}

Recently, there has been interest \cite{ refinements, BLV} in going beyond Katz--Sarnak to predict the pointwise behavior of the distribution $\mu_{\calM_g}(t/\sqrt{p})$ and $\mu_{\calA_g}(t/\sqrt{p})$ where $t$ is an integer. In other words, we want to know the weighted cardinality of $\mathcal M_g(\F_p,t)$ and $\mathcal A_g(\F_p,t)$, which we define to be the set of genus $g$ curves and $g$-dimensional PPAVs with trace exactly equal to $t$. We take this as the central motivating question of our paper. 

In \cite{refinements}, Bergström-Howe-Lorenzo-Ritzenthaler asks whether the following is true:
\begin{question}[{Naïve pointwise Katz-Sarnak for $\mathcal M_g$ \cite[Conjecture 5.1]{refinements}}]\label{question: naive}
Let $g\ge 2$ be an integer. Is it true that 
$$\sup_{t\in \Z}\left|\frac{\#\mathcal M_g(\F_p,t)}{\# \mathcal M_g(\F_p)}\sqrt p- \ST_g(t/\sqrt p)\right| \rightarrow 0$$
as $p\rightarrow \infty$?
\end{question}

This is a much stronger statement than Theorem \ref{thm: katz-sarnak} and appears to be beyond the usual point counting methods over $\bb{F}_p$ via the Lefschetz fixed point formula. In \cite{BLV}, Ballini-Lombardo-Verzobio proved that the above conjecture was false for all $g\ge 2$ and gave a corrected version. The reason for this is because Question \ref{question: naive} fails to take into account the distribution of matrices on non-archimedean places.
\par 
For each prime $l$, the analogous random matrix heuristic predict that the distribution of the trace of $C$ in $\Z_l$ should be given by the trace of a random matrix in $\GSp_{2g}(\Z_l)$. This was proven in \cite[Theorem 2.1]{BLV}. As a result, if we want to predict the ``$\bb{Z}$''-distribution of the trace, we should combine the local factors over all places of $\bb{Q}$. For $l\neq p$, we have
$$v_l(t)=v_l^{g,p}(t)\coloneqq \lim_{k\rightarrow \infty}\frac{\# \{\gamma\in \GSp_{2g}(\Z/l^k\Z)\mid \tr(\gamma)=t, \text{mult}(\gamma)=p\}}{\#\GSp_{2g}(\Z/l^k\Z)/(l^{k}\phi(l^k))}$$
where we omit the dependence on $g,p$ when it is clear from context. For $l=p$ we define $v_p(t)=1$, see Remark \ref{rmk: vp choice} for more details. It is important to note that the local factors we give here are different from the ones proposed in \cite{BLV}, which counts projections from $\Z_l$ instead, and we explain this discrepancy in Remark \ref{rmk: AAGG problem}. After accounting for these local factors, \cite{BLV} conjectures that there is an $L^1$ convergence of $\mu_{\mathcal M_g(\F_p)}$ and $\mu_{\mathcal A_g(\bb{F}_p)}$ to the expected distribution. We state the conjectures in the case of prime fields as follows.
\begin{conjecture}[{$L^1$ Katz-Sarnak for $\mathcal M_g$ \cite[Conjecture 3.4] {BLV}}]\label{conj: L1 Katz-Sarnak}
Let $g\ge 2$ be an integer. Then,
$$\sum_{t\in \Z}\left|\frac{\#\mathcal M_g(\F_p,t)}{\# \mathcal M_g(\F_p)}-\frac{1}{\sqrt p} \ST_g(t/\sqrt p)\prod_{l}v_l(t)\right| \rightarrow 0$$
as $p\rightarrow \infty$.
\end{conjecture}

\begin{conjecture}[{$L^1$ Katz-Sarnak for $\mathcal A_g$ \cite[Conjecture 3.7] {BLV}}]\label{conj: L1 Katz-Sarnak Ag}
Let $g\ge 2$ be an integer. Then,
$$\sum_{t\in \Z}\left|\frac{\#\mathcal A_g(\F_p,t)}{\# \mathcal A_g(\F_p)}-\frac{1}{\sqrt p} \ST_g(t/\sqrt p)\prod_{l}v_l(t)\right| \rightarrow 0$$
as $p\rightarrow \infty$.
\end{conjecture}
\begin{remark} We make an elementary but important observation that while $L^1$ convergence is weaker than pointwise convergence, it is not too far off as $L^1$ convergence implies pointwise convergence for \textit{almost all} $t\in [-2g\sqrt p,2g\sqrt p]$. More precisely, for any $\epsilon,\delta>0$, we have for all sufficiently large $p$ that all $t\in [-2g\sqrt p,2g\sqrt p]$ with at most $\delta \sqrt p$ possible exceptions satisfy $$\left|\frac{\#\mathcal A_g(\F_p,t)}{\# \mathcal A_g(\F_p)}\sqrt p-\ST_g(t/\sqrt p)\prod_{l}v_l(t)\right|<\epsilon.$$
This is a simple consequence of Chebyshev's inequality.
\end{remark}

Our main theorem is that Conjecture \ref{conj: L1 Katz-Sarnak Ag} is true.

\begin{theorem}\label{thm: A_g}
Conjecture \ref{conj: L1 Katz-Sarnak Ag} is true for all $g \geq 2$. In fact, for any $\epsilon>0$, we have
$$\sum_{t\in \Z}\left|\frac{\#\mathcal A_g(\F_p,t)}{\# \mathcal A_g(\F_p)}-\frac{1}{\sqrt p} \ST_g(t/\sqrt p)\prod_{l}v_l(t)\right| = O_g\left(\log(p)^{-1+\epsilon}\right).$$
\end{theorem}
When $g = 1$, the conjecture is true due to Gekeler's formula \cite{Gekeler2003FrobeniusDO}, derived from the work of \cite{Deuring_1941} and the analytic class number formula. In this simple case we actually have equality between the two terms for all $t,p$. Other historical work for counting elliptic curves over $\F_p$ are \cite{Birch1968HowTN}, \cite{Lenstra_1987}, \cite{Galbraith_McKee_2000}. 

Despite the work done in the elliptic curve case, there has been no significant progress on the case of genus $g\ge 2$. In our paper, we deduce Conjecture \ref{conj: L1 Katz-Sarnak} for $g = 2$ and a variant for $g = 3$ using Theorem \ref{thm: A_g} and the Torelli map.
\begin{theorem}[$L^1$ Katz-Sarnak for $\mathcal M_2$ and $\mathcal M_3^{sym}$]\label{thm: M_2, M_3 sym}
Let $\mathcal M_3^{sym}(\F_p,t)=\frac{\mathcal M_3(\F_p,t)+\mathcal M_3(\F_p,-t)}{2}$ be the symmetrized distribution of $\mathcal M_3(\F_p,t)$. For any $\epsilon>0$, we have
$$\sum_{t\in \Z}\left|\frac{\#\mathcal M_2(\F_p,t)}{\# \mathcal M_2(\F_p)}- \frac{1}{\sqrt p}\ST_2(t/\sqrt p)\prod_{l}v_l(t)\right| =O\left(\log(p)^{-1+\epsilon}\right),$$
$$\sum_{t\in \Z}\left|\frac{\#\mathcal M_3^{sym}(\F_p,t)}{\# \mathcal M_3(\F_p)}- \frac{1}{\sqrt p}\ST_3(t/\sqrt p)\prod_{l}v_l(t)\right| =O\left(\log(p)^{-1+\epsilon}\right).$$
\end{theorem}

Now we explain how to deduce Theorem \ref{thm: M_2, M_3 sym} from Theorem \ref{thm: A_g}. Recall that the Torelli map $\mathcal M_g \injects \mathcal A_g$ that sends a curve to its Jacobian is injective, and that we have $\dim(\mathcal M_g)=3(g-1)$ and $\dim(\mathcal A_g)=g(g+1)/2$. Furthermore, the Torelli theorem tells us that 
$$\Aut(\Jac(C)) \cong \begin{cases}
\Aut(C) & \text{if } C \text{ is hyperelliptic}\\
\Aut(C)\times \Z/2\Z &\text{otherwise}.\\
\end{cases}$$
We first discuss the $g=2$ case. By the Weil bound, the groupoid cardinalities $\# \mathcal M_2(\F_p)$ and $\# \mathcal A_2(\F_p)$ both have size $p^3 + O(p^{2.5})$. Furthermore, a genus $2$ curve is generically hyperelliptic, which means that $\# \Aut(\Jac(C)) = \# \Aut(C)$. Hence, the image of the Torelli map is almost surjective, so the $L^1$ convergence for $\mathcal M_2$ follows from $\mathcal A_2$. For $g=3$, while the groupoid cardinalities $\# \mathcal M_3(\F_p)$ and $\# \mathcal A_3(\F_p)$ both have size $p^6+O(p^{5.5})$, a curve of genus $3$ is generically non-hyperelliptic, so $\# \Aut(\Jac(C)) = 2 \# \Aut(C)$ and the image of the Torelli map contains only approximately half of the points in $\mathcal A_3(\F_p)$. This obstruction can be described more precisely -- if $C$ is non-hyperelliptic, the quadratic twist of its Jacobian $\Jac(C)_{\chi_2}$ cannot be a Jacobian, and furthermore $t_p(\Jac(C)_{\chi_2})=-t_p(\Jac(C))$, see \cite[Appendix]{Serre_Torelli} for more details. This explains why we can only prove $L^1$ convergence for $\mathcal M_3^{sym}$. 
\par 
One may wonder whether we can use Theorem \ref{thm: A_g} to prove Conjecture \ref{conj: L1 Katz-Sarnak} for general $g$ or analogous conjectures for other families such as the moduli space of hyperelliptic curves $\mathcal H_g$. It appears that the crux to resolving this problem is to characterize the image of the Torelli map, i.e. which PPAVs are Jacobians of curves (or hyperelliptic curves), which still remains an open problem.
\par 
We also conjecture that pointwise versions of Conjecture \ref{conj: L1 Katz-Sarnak} and Theorem \ref{thm: A_g} are true.  Although we fall short of proving it, we believe that a refinement of our method could lead to a possible proof, at least in the case of $\mathcal A_g$, $\mathcal M_2$ and $\mathcal M_3^{sym}$. 
\begin{conjecture}[Pointwise Global Katz-Sarnak for $\mathcal M_g$ and $\mathcal A_g$]\label{conj: pointwise}
For $g\ge 2$, we have
$$\sup_{t\in \Z}\left|\frac{\#\mathcal M_g(\F_p,t)}{\# \mathcal M_g(\F_p)}\sqrt p- \ST_g(t/\sqrt p)\prod_{l}v_l(t)\right| \rightarrow 0$$
as $p\rightarrow \infty$. For $g\ge 1$, we have
$$\sup_{t\in \Z}\left|\frac{\#\mathcal A_g(\F_p,t)}{\# \mathcal A_g(\F_p)}\sqrt p- \ST_g(t/\sqrt p)\prod_{l}v_l(t)\right| \rightarrow 0$$
as $p\rightarrow \infty$.
\end{conjecture}

\subsection{Proof overview}
The methods that are usually used to prove convergence theorems like Theorem \ref{thm: katz-sarnak} and \ref{thm: katz-sarnak A_g} are geometric in nature, but these methods are unlikely to be strong enough to prove something global like Theorem \ref{thm: A_g}. For example, one such method relies on computing the étale cohomology of the moduli spaces which allow us to compute moments of the trace, but this method cannot distinguish curves of trace $t$ and $t+1$. Another possible approach is to use monodromy and Deligne's equidistribution theorem, but this only gives us convergence over each local place. Thus, we take a completely different approach. Roughly, our proof consists of two parts -- first, we derive a formula for the number of PPAVs with a fixed characteristic polynomial, then, we sum this formula up over characteristic polynomials with fixed trace.

We discuss the first part of our proof. In Section \ref{sec: parametrization}, we parametrize the moduli space $\mathcal A_g$ in terms of the characteristic polynomials $f_A$ using Honda-Tate theory, and we define the region of characteristic polynomials. In Section \ref{sec: ppav count formula}, we then prove the following formula for the number of PPAVs with a corresponding generic characteristic polynomial $f_A$ which may be of independent interest. 
\begin{theorem}\label{thm: PPAV counts for char poly}
Suppose that $A$ is a simple, ordinary abelian variety whose Galois group is generic, i.e. isomorphic to $(\Z/2\Z)^g\rtimes S_g$. Then, there are a nonzero number of PPAVs with characteristic polynomial $f_A$, and the groupoid cardinality is given by
$$\# \mathcal A_g(\F_p,f_A)=p^{\dim(\mathcal A_g)/2}v_{\infty}(f_A)\prod_l v_l (f_A)$$
where the local factors $v_\infty(f_A)$ and $v_l(f_A)$ are defined in Equation \eqref{eqn: vinf definition} and \eqref{eqn: vl t and f definition}.
\end{theorem}

Our work in Section \ref{sec: parametrization} also tells us that the conditions in Theorem \ref{thm: PPAV counts for char poly} are generic among all isogeny classes. As far as we know, this is first such explicit formula for characteristic polynomials, albeit only in the generic case. Other similar formulas were given in \cite{Achter_2023} and \cite{Howe_2022}, however, these count PPAVs within different classes which are finer than the data of a characteristic polynomial. We also expect there to be no such simple formula in the non-generic case, although we believe it is possible to give a concrete upper bound.

The starting point for the proof of Theorem \ref{thm: PPAV counts for char poly} is Kottwitz's orbital integral formula for the number of PPAVs in a Polarized Abelian Variety (PAV) isogeny class as stated in Theorem \ref{thm: orbital integral}. Then, we built on the work of \cite{Achter_2023} which rewrites this as an Euler product of rational orbital integrals as well as a global factor including the Tamagawa number in Theorem \ref{thm: AAGG}. This was a generalization of the earlier work done by \cite{Achter_Gordon_2017} for the case of elliptic curves. \cite{Achter_2023} attempts to express these rational orbital integrals explicitly in terms of a ratio of matrices projected from $\GSp_{2g}(\Z_l)$, but we found an error in this explicit expression for the ratios. Fortunately, we only require stable orbital integrals for this paper, which can be easily expressed in the ratio of explicit matrix counts, which is explained in the appendix of this paper. 

Then, to move from counts in a PAV-isogeny class to an AV-isogeny class determined by characteristic polynomial, we characterize the set of PAV-isogeny classes in an AV-isogeny with Kottwitz's invariant and show that this is well behaved in the generic case, which allows us to sum up the counts over the PAV-isogeny classes to obtain Theorem \ref{thm: PPAV counts for char poly}. Note that when compared to Theorem \ref{thm: AAGG}, there is no Tamagawa number in the formula, which seems to suggest that the Tamagawa number is an adjustment factor for PAV-isogeny classes, at least in the generic case. 

In Section \ref{sec: local factors} and \ref{sec: point counting}, we prove some technical lemmas on the local factors which are required for the rest of the proof. While the local factors $v_{\infty}(f)$ and $v_{l}(t)$ are easy to understand, $v_l(f)$ is much more difficult. This is because the singularities of the associated subscheme of $\GSp_{2g}$ cut out by characteristic polynomial $f$ are hard to understand, which leads to difficulty bounding point counts over $\Z/l^k\Z$. In Section \ref{sec: bounds, stab, disc} we discuss conjectures and partial results on $v_l(f)$, and Section \ref{sec: point counting} is dedicated to proving these partial results. Fortunately, these partial results are enough to prove our final theorem, even though they lead to more technical difficulty in the analytical arguments. We highlight that this difficulty is intrinsic to the case of $g\ge 2$ and does not exist in the elliptic curve case, essentially because $\GSp_{2g}$ is much more complicated than $\GL_{2}$.

Now we explain the second part of our proof. To obtain Theorem \ref{thm: A_g} from Theorem \ref{thm: PPAV counts for char poly}, we need to sum the counts $\#\mathcal A_g(\F_p,f)$ over all characteristic polynomials $f$ with a fixed trace $t$. Similar ideas of summing Euler products in the elliptic curve case were explored in \cite{David_Koukoulopoulos_Smith_2016} and \cite{Ma_2024}, however, the case of $g\ge 2$ is more difficult which requires the introduction of new ideas and methods. Most notably, we tackle an arithmetic statistic question on bounding the size of a family of exceptional number fields.

In Section \ref{sec: vl(fA)}  we see that for all but finitely many primes $l$, $v_l(f_A)$ is a ratio of local Dedekind zeta functions ${\zeta_{K,l}(1)}/{\zeta_{K^+,l}(1)}$ where $K \coloneqq \Q[x]/f_A(x)$ and $K^+$ is the totally real subfield of $K$, and this implies the convergence of $\prod_l v_l(f_A)$. In order to deal with the sum of infinite Euler products, we need to approximate $\prod_l v_l(f_A)$ with a truncated product $\prod_{l\le l_0} v_l(f_A)$. In the elliptic curve case, $K$ is a quadratic field, so these reduce to Dirichlet L-functions, which allows for a completely analytic proof by looking at Dirichlet characters such as in \cite{David_Koukoulopoulos_Smith_2016}. However, this does not work in the $g\ge 2$ case as we are now concerned with Dedekind zeta functions and Artin L-functions, and thus it seems likely that a purely analytic argument would be impossible. Nevertheless, by looking at prime splitting, we see that giving such a truncation of the infinite Euler product is essentially equivalent to the statement of an effective Chebotarev theorem for the Dedekind zeta function of the Galois closure $\widetilde K$. Fortunately for us, there is exactly a recent effective Chebotarev theorem in families of number fields developed by \cite{PTW20} and improved in \cite{jessethorner} for general Galois group $G$, and our problem is a natural application of this result. We state this in Theorem \ref{thm: eff chebotarev}.

However, a key problem that arises here is the existence of exceptional fields for which the effective Chebotarev theorem could fail, and we need to bound the number of these. To do this, we use a quantitative version of Hilbert's irreducibility theorem given in \cite{effectivehilbertsirreducibilitytheorem} which was proven using Bombieri-Pila type of methods. It is important for our application that we use a version that is only logarithmically dependent on the height of the polynomial. In more detail, we need to bound the quantity $\mathfrak m_{\mathfrak F_G(Q)}$ defined in Theorem \ref{thm: eff chebotarev}, which essentially counts the maximum number of number fields in our family that one particular number field in our family intersects. To do this, for each number field $L$ we construct a polynomial over $\Z[x,a_1,\ldots, a_g]$ where $a_i$ are the variables of the characteristic polynomial $f$, such that the irreducibility of this polynomial over $x$ given some values of $a_1,\ldots, a_g$ implies that $\Q[x]/f(x)$ does not intersect $L$, which allows us to apply Hilbert's irreducibility. We discuss this in Section \ref{sec: chebotarev}.

Finally, we finish the proof of Theorem \ref{thm: A_g} in Section \ref{sec: summing euler products} using an analytical argument which on a high level proceeds as follows. We are left with a sum of truncated products $\sum_{f} v_{\infty}(f)\prod_{l\le l_0}v_l(f)$. We approximate $v_l(f)$ by an $l^k$-periodic function $v_{l,k}(f)$, so the product is $\prod_{l\le l_0}l^k$ periodic. It would be easy to sum these products if this period is $o(p)$ by cutting the region of characteristic polynomials into boxes of $o(p)$ length. However, this is not the case. Nevertheless, we prove an averaging lemma in Section \ref{sec: summing euler products} which allows us to do the above, essentially by using the Taylor expansion of $\exp(\sum_{l\le l_0}\log(v_{l,k}(f)))$. Hence, by cutting the region of characteristic polynomials into boxes of size $o(p)$, in each box $\prod_{l\le l_0}v_l(f)$ averages out to $\prod_{l\le l_0}v_l(t)$, and we can approximate the sum of $v_{\infty}(f)$ over these boxes with an integral, giving us the Sato-Tate distribution. Of course, the actual details are technical and complicated, notably, we need to introduce an auxiliary factor to deal with our suboptimal bounds on $v_l(f)$.

\subsection{Acknowledgments}
The first author would like to thank Andrew Sutherland for introducing him to \cite{refinements} and surrounding problems. We are also grateful to Jonas Bergström, Wee Teck Gan, Nicholas Katz, Davide Lombardo, Christophe Ritzenthaler, Peter Sarnak, Will Sawin, Jacob Tsimerman and Matteo Verzobio for their helpful comments.
\section{Parametrization by characteristic polynomials}\label{sec: parametrization}
The rest of the paper is devoted to proving Theorem \ref{thm: A_g}. We fix the genus $g\ge 2$, since the $g=1$ case corresponds to elliptic curves which has already been proven in \cite{Gekeler2003FrobeniusDO}. All little-$o$ or big-$O$ notation will omit the dependence on $g$, but the dependence on all other variables like $p$ and $l$ will be included. A common theme in this paper is that we will focus on counting PPAVs in the most generic case (e.g. simple, ordinary, etc.) as this is sufficient to prove the $L^1$ bound we want.
\subsection{Honda-Tate theory} This section roughly follows the exposition in \cite[Section 3]{Howe1995PrincipallyPO}. Consider an abelian variety $A$ of dimension $g$ over $\F_p$, and its characteristic polynomial $f_A(x)$ which has degree $2g$. Weil's theorem tells us that every root of $f_A(x)$ is a $p$-Weil number, i.e. an algebraic integer which has absolute value $p^{1/2}$ for every embedding into $\C$. The Honda-Tate Theorem says that the map sending an abelian variety $A$ to the complex roots of $f_A$ induces a bijection between isogeny classes of simple abelian varieties and Galois conjugacy classes of Weil $p$-numbers. Since every abelian variety is isogenous to a product of simple abelian varieties, this gives a partition of abelian varieties into isogeny classes, in terms of the data of $p$-Weil numbers.

However, it is difficult to use $p$-Weil numbers as a parametrization for counting purposes. Instead, we will use a variant of the Honda-Tate theorem for ordinary abelian varieties where isogeny classes are parametrized by characteristic polynomials. See \cite[Theorem 3.3]{Howe1995PrincipallyPO} and \cite[Theorem 1]{HondaTate}.
\begin{definition}\label{definition: ordinary}
A $g$-dimensional abelian variety $A$ over $\F_p$ is \textit{ordinary} if the set of $A[p](\ovl{\bb{F}}_p)$ has exactly $p^g$ elements. A \textit{Weil $p$-polynomial} is an integer monic polynomial where all complex roots have magnitude $p^{1/2}$. An \textit{ordinary Weil $p$-polynomial} is a Weil $p$-polynomial of even degree where $p$ doesn't divide the middle coefficient.
\end{definition}
\begin{theorem}[Honda-Tate for ordinary abelian varieties]\label{thm: honda tate}
The map $A\mapsto f_A(x)$ induces a bijection between the isogeny classes of ordinary abelian varieties of dimension $g$ over $\F_p$ to the set of ordinary Weil $p$-polynomials of degree $2g$. Furthermore, isogeny classes of simple ordinary abelian varieties correspond to irreducible ordinary Weil $p$-polynomials. 
\end{theorem}
\begin{remark}
The bijection in Theorem \ref{thm: honda tate} does not extend to a bijection from all abelian varieties to all Weil $p$-polynomials. Indeed, there are simple non-ordinary abelian varieties $A$ which have characteristic polynomial that is a nontrivial power of an irreducible polynomial, say $f_A(x)=g(x)^r$. In this case, $g(x)$ is also a Weil polynomial, there is no corresponding abelian variety with characteristic polynomial $g(x)$, as it would have to be isogenous to $A$ by the original Honda-Tate theorem.
\end{remark}
\subsection{The region $\mathcal R_g$} Next, we find an explicit condition on the coefficients of an ordinary Weil $p$-polynomial. Throughout the paper, we will often switch between multiplier $p$ (characteristic polynomials) and multiplier $1$ (normalized characteristic polynomials) versions of symplectic polynomials, and we will attempt to consistently denote the first kind by $f(x)$ with coefficients $a_i$ and the second by $g(x)$ with coefficients $b_i$. 
\begin{definition}\label{def: R_g}
Let the region $\mathcal R_g$ consist of the tuples $\mathbf b=(b_1,\ldots, b_g)\in \R^g$ such that the polynomial $g_{\mathbf b}(x)\coloneqq x^{2g}+b_1x^{2g-1}+\cdots +b_gx^g+b_{g-1}x^{g-1}+\cdots +1$ has roots that come in pairs $\{e^{i\theta},e^{-i\theta}\}$. Define the map $[p]:\R^g\rightarrow \R^g$ by sending $(b_1,\ldots, b_g)\mapsto (b_1p^{1/2},\ldots, b_gp^{g/2})$, i.e it scales the $i$-th term by $p^{i/2}$. Then, it is clear that the image $[p](\mathcal R_g)$ contains the tuples $\mathbf a = (a_1,\ldots, a_g)\in \R^g$ such that $f_{\mathbf a}(x)\coloneqq x^{2g}+a_1x^{2g-1}+\cdots +a_gx^g+pa_{g-1}x^{g-1}+\cdots+p^{g-1}a_1x+p^g$ has roots that come in pairs $\{\sqrt pe^{i\theta},\sqrt pe^{-i\theta}\}$, because we have the relation $f_{[p](\mathbf b)}(x)=p^gg_{\mathbf b}(x/\sqrt p)$.
\end{definition}
The reason we introduced the region $\mathcal R_g$ is because it parametrizes ordinary Weil $p$-polynomials.
\begin{proposition}\label{prop: ordinary weil p-poly}
The ordinary Weil $p$-polynomials of degree $2g$ are exactly those of the form $f_{\mathbf a}(x)$ where $\mathbf a\in [p](\mathcal R_g)$ and $p\nmid a_g$. 
\end{proposition}
\begin{proof}
An ordinary Weil $p$-polynomial $f(x)$ of degree $2g$ has integral coefficients, so its roots must come in pairs of $\{\sqrt p e^{i\theta}, \sqrt p e^{-i\theta} \}$ or $\{-\sqrt p, \sqrt p\}$. Thus, when factored over $\R$, the irreducible factors of $f(x)$ are of the form $x^2+tx+p$ or $x^2-p$. However, it cannot have the factor $x^2-p$, because otherwise we would have $f(x) = (x^2-p)^r\prod_{i=1}^{g-r} (x^2+t_ix+p) \equiv x^{2r} \prod_{i=1}^{g-r} (x^2+t_ix) = x^{r+g}\prod_{i=1}^{g-r} (x+t_i)\pmod{p}$ for $r>0$ which clearly has middle coefficient a multiple of $p$. It follows that an ordinary Weil $p$-polynomial must have roots that come in conjugate pairs $\{\sqrt p e^{i\theta}, \sqrt p e^{-i\theta}\}$, so it is of the form $f_{\mathbf a}(x)$ with $(a_1,\ldots, a_g)\in [p](\mathcal R_g)$ and the condition that $p\nmid a_g$ follows directly. The converse is clear from the definitions. 
\end{proof}
We give a more abstract description of $\mathcal R_g$ which will be useful later.  
\begin{lemma}\label{lem: R_g Usp}
Any matrix $M\in \USp_{2g}$ has a characteristic polynomial $g_M(x)$ with symmetric coefficients which can be written as $g_{\mathbf b}(x)$ for some $\mathbf b$, so we let $\pi \colon \USp_{2g}\rightarrow \R^g$ be the map that sends $M$ to $\mathbf b$. Then, $\mathcal R_g=\pi(\USp_{2g})$ is the image, and in particular, $\mathcal R_g$ is a compact subset of $\R^g$.
\end{lemma}
\begin{proof}
Take $M\in \USp_{2g}\coloneqq \Sp_{2g}(\C)\cap U(2g)$. Since $M$ is symplectic, we have $M^{-1} = \Omega^{-1}M^T\Omega$ where $\Omega = \left(\begin{smallmatrix}
0 & I_g \\ -I_g & 0 \end{smallmatrix}\right)$. As eigenvalues remain the same under transposition and conjugation, $M^{-1}$ has the same eigenvalues as $M$. Furthermore, since $M$ is also unitary, this means that eigenvalues either come in complex conjugate pairs $\{e^{i\theta},e^{-i\theta}\}$ or are $1$ or $-1$. Furthermore, both $-1$ and $1$ have even multiplier because $\det(M)=1$, and thus all roots can be grouped into complex conjugate pairs $\{e^{i\theta},e^{-i\theta}\}$. It follows directly that the characteristic polynomial $g_M(x)$ has symmetric coefficients and that $\pi(M)\in \mathcal R_g$. Conversely, it is clear that for any $(b_1,\ldots,b_g)\in \mathcal R_g$, the corresponding polynomial is the characteristic polynomial of the diagonal matrix $\diag(e^{i\theta_1},\ldots ,e^{i\theta_g},e^{-i\theta_1},\ldots ,e^{-i\theta_g})\in \USp_{2g}$.
\end{proof}
\begin{example} We describe the region $\mathcal R_2$ explicitly using inequalities. For $(b_1,b_2)\in \mathcal R_2$, pair the conjugate roots and factorize the corresponding polynomial $x^4+b_1x^3+b_2x^2+b_1x+1 = (x^2+tx+1)(x^2+sx+1)$ where $|t|,|s|\le 2$, so we have $b_1=t+s$ and $b_2=2+st$. Thus, $|b_1|\le 4$, and given a fixed $b_1$ in this range, $b_2$ is maximized by taking $s=t$ and minimized when $s=2,t=b_1-2$ when $b_1\ge 0$ and $s=-2,t=b_1+2$ when $b_1\le 0$. Thus, $\mathcal R_2 = \left\{(b_1,b_2) \mid |b_1|\le 4, \  2|b_1|-2 \le b_2\le 2+(b_1/2)^2\right\}$.
\end{example}
One would expect intuitively that for general $g$ that $\mathcal R_g$ can be completely described explicitly using equations, just like in the $g=2$ case. We prove a slightly weaker statement which suffices for our purposes, but before that we give a useful definition.
\begin{definition}\label{def: tot real poly}
Suppose a monic polynomial $g(x)$ of degree $2g$ has roots that come in pairs $\{e^{i\theta_j},e^{-i\theta_j}\}$, or equivalently that it is of the form $g_{\mathbf b}(x)$. Then, define the totally real polynomial $g^+(x)$ of degree $g$ to be the monic polynomial with roots $e^{i\theta_j}+e^{-i\theta_j}$, so that $g(x)=x^gg^+(x+\frac 1 x)$. Likewise, for multiplier $p$ polynomials, if $f(x)$ has roots $\{\sqrt pe^{i\theta_j},\sqrt pe^{-i\theta_j}\}$, define $f^+(x)$ to have roots $\sqrt p(e^{i\theta_j}+e^{-i\theta_j})$ so $f(x)=x^gf^+(x+\frac p x)$.
\end{definition}
\begin{lemma}\label{lem: explicit R_g}
There exist a compact subset $K\subseteq \R^{g-1}$ and two continuous functions $s,t\colon K\rightarrow \R$ such that $$\mathcal R_g =\{(b_1,\ldots, b_g)\in \R^g \mid s(b_1,\ldots,b_{g-1})\le b_g \le t(b_1,\ldots,b_{g-1}), (b_1,\ldots, b_{g-1})\in K\}.$$
\end{lemma}
\begin{proof}
We let $K$ be the projection from $\mathcal R_g$ to the first $g-1$ coordinates which is compact because Lemma \ref{lem: R_g Usp} tells us that $\mathcal R_g$ is compact. Then, note that $\mathbf b\in \mathcal R_g$ if and only if $g^+_{\mathbf b}(x)=x^g+c_1x^{g-1}+\cdots+c_g$ has all $g$ roots in $[-2,2]$. By expanding the equation $g_{\mathbf b}(x)= x^gg^+_{\mathbf b}(x+\frac 1 x)$, it is clear that for all $0\le i\le g$ we have $b_i=c_i+\binom{g-i+2}{1}c_{i-2}+\binom{g-i+4}{2}c_{i-4}+\cdots$. This system of equations is defined by a lower triangular matrix with $1$s on the diagonal, and by inverting this matrix we get another lower triangular matrix with $1$s on the diagonal. Thus, we can write each $c_i$ as $b_i$ plus a linear combination of $b_{i-2},b_{i-4},\ldots$. Here, we included the variables $b_0,c_0$ in our system of equations, and both are equal to $1$.

Now, suppose we fixed $b_1,\ldots,b_{g-1}$, which defines $c_1,\ldots,c_{g-1}$ by the previous paragraph. Plot the graphs defined by $y=x^g+c_1x^{g-1}+\cdots +c_{g-1}x$ and the line $y=-c_g$, and note that the intercepts are exactly the real roots of $g^+_{\mathbf b}(x)$. From this we can see that all $g$ roots of $g^+_{\mathbf b}(x)$ lie in $[-2,2]$ if and only if $c_g$ lies in some closed (or empty) interval. By our change of variables, this happens if and on $b_g\in [s(b_1,\ldots,b_{g-1}),t(b_1,\ldots,b_{g-1})]$ for some functions $s,t$. To prove $s$ and $t$ are continuous, note that varying $b_1,\ldots, b_{g-1}$ continuously leads to changing $c_1,\ldots, c_{g-1}$ and the graph $y=x^g+c_1x^{g-1}+\cdots +c_{g-1}x$ continuously, which changes the interval of $c_g$ and hence $b_g$ continuously.
\end{proof}
The reason why we want $\mathcal R_g$ to be expressed in this way is that this allows us to approximate the region $\mathcal R_g$ well enough with hypercubes.
\begin{proposition}\label{prop: volume R_g}
The volume $\Vol(\mathcal R_g)$ is well-defined and finite. Furthermore, chop up $\R^g$ into a regular grid of hypercuboids with side lengths $\epsilon_1,\ldots, \epsilon_g\le \epsilon$ and let $\mathcal R^-_g$ (resp. $\mathcal R^+_g$) be the union of the hypercuboids which are completely contained within (resp. which intersect) $\mathcal R_g$. Then, $|\Vol(\mathcal R_g)-\Vol(\mathcal R_g^{\pm})|=O(\epsilon)$.
\end{proposition}
\begin{proof}
This essentially follows from Lemma \ref{lem: explicit R_g} and the theory of Riemann integration. We have $\Vol(\mathcal R_g) = \int_{K} \left( t(b_1,\ldots, b_{g-1})-s(b_1,\ldots, b_{g-1})\right) db_1\cdots db_{g-1}$. To estimate this integral, consider the subdivision of $\mathcal R_g$ given in the statement, and let $\mathcal T$ be the set of hypercuboids intersecting $\mathcal R_g$. Let $\pi\colon \R^g\rightarrow \R^{g-1}$ be the projection onto the first $g-1$ coordinates, so the previous subdivision induces a subdivision of $\R^{g-1}$ via $\pi$, and let $\mathcal S$ be the set of hypercuboids intersecting $\mathcal K$. Then, $\Vol(\mathcal R_g^+)=(\#\mathcal T)\epsilon_1\cdots \epsilon_g = \sum_{S\in\mathcal S}\#\{T\in \mathcal T\mid \pi(T)=S\}\epsilon_1\cdots \epsilon_g$. Since $s,t$ are continuous over a compact set, they are uniformly continuous and in particular $c$-Lipschitz, so for each hypercuboid $S\in \mathcal S$, $\left|\max(s(S))-\min(s(S))\right|<c\epsilon$ and analogously for $t$. Hence, comparing the sum and integral over $S$, we have $(\#\{T\in \mathcal T\mid \pi(T)=S\}-2) \epsilon_1\cdots \epsilon_g\le (\min(t(S))-\max(s(S)))\epsilon_1\cdots \epsilon_{g-1}\le \int_S t(b_1,\ldots, b_g)-s(b_1,\ldots,b_g)db_1\cdots db_g +2c\epsilon\epsilon_1\cdots\epsilon_{g-1}$. Furthermore, note the cardinality of $\mathcal S$ is $\le C/\epsilon_1\cdots\epsilon_g$ because $K$ is compact, so summing over $S\in \mathcal S$ we have $\Vol(\mathcal R_g^+)\le \Vol(\mathcal R_g)+(2c+2)C\epsilon$. The same proof works for $\mathcal R_g^-$.
\end{proof}

As a corollary, we obtain the following bound on the number of ordinary Weil $p$-polynomials of degree $2g$. Note that this has been previously proven by Dipippo and Howe \cite{DH98}. 

\begin{corollary}\label{cor: num ord weil poly}
There are $\Theta(p^{g(g+1)/4})$ many ordinary Weil $p$-polynomials of degree $2g$.
\end{corollary}
\begin{proof}
We chop $\mathcal R_g$ into hypercubes of side length $p^{-1/2}$ and volume $p^{-g/2}$. In each cube, there are approximately $\frac{p-1}{p}p^{g(g-1)/4}$ tuples $(a_1/p^{1/2},\ldots, a_g/p^{g/2})$ with $p\nmid a_g$. By Proposition \ref{prop: volume R_g}, there are approximately $\frac{p-1}{p}\Vol(\mathcal R_g)p^{g(g+1)/4}(1+O(p^{-1/2}))$ such tuples in $\mathcal R_g$, and by Proposition \ref{prop: ordinary weil p-poly} these corresponds to ordinary Weil $p$-polynomials.
\end{proof}
It is not hard to prove the same statement for all Weil $p$-polynomials, by showing that there are $o(p^{g(g+1)/4})$ many non-ordinary Weil $p$-polynomials. 
A similar upper bound was derived in \cite[Corollary 2.3]{LipnowskiTsimerman2018}, but a more careful analysis of the region $\mathcal R_g$ as seen in Lemma \ref{lem: explicit R_g} and Proposition \ref{prop: volume R_g} was needed for our lower bound.
\subsection{Irreducibility} We will prove that several properties are satisfied by a generic ordinary Weil $p$-polynomial, which would eventually allow us to simplify things by assuming these properties. First, we will show that a generic ordinary Weil $p$-polynomial is irreducible.
\begin{lemma}\label{lem: generic irreducibility}
There are at most an $O(p^{-(g-1)/2})$ proportion of ordinary Weil $p$-polynomials that are reducible over $\Z[x]$.
\end{lemma}
\begin{proof}
Suppose an ordinary Weil $p$-polynomial can be factored over $\Z$ as $f(x)=r(x)s(x)$. From the proof of Proposition \ref{prop: ordinary weil p-poly}, $f(x)$ and thus $r(x)$ must have roots that come in pairs $\{\sqrt p e^{i\theta},\sqrt p e^{-i\theta}\}$. Say it is of degree $2d$, then following the same proof we can write $r(x)=f_{\mathbf c}(x)$ for some $c\in [p](\mathcal R_r)$. By the proof of Corollary \ref{cor: num ord weil poly}, there are $O(p^{d(d+1)/4})$ possibilities for $r(x)$. Likewise, $s(x)$ has degree $2(g-d)$ and there are $O(p^{(g-d)(g-d+1)/4})$ many such $s(x)$. We obtain the desired result by multiplying these and summing over all $d$, noting that the worst case happens when $d$ is $1$ or $g-1$.
\end{proof}
Recall that by Theorem \ref{thm: honda tate} irreducible ordinary Weil $p$-polynomials correspond to simple and ordinary abelian varieties. Let $[A,\lambda]$ be a PPAV where $A$ is a simple ordinary abelian variety, and let $f_A$ be the characteristic polynomial which is an irreducible ordinary Weil $p$-polynomial. It is easier to describe the endomorphism ring of $A$ in this irreducible case. Let $\End^0(A) = \End(A)\otimes \Q$ and $K=\Q[\Frob_p]\subseteq \End^0(A)$.
\begin{lemma}
We have $\End^0(A)=K\cong \Q[x]/f_A(x)$.
\end{lemma}
\begin{proof}
$\End^0(A)=K$ follows from \cite[Theorem 2(c)]{Tate_End}, and $K\cong \Q[x]/f_A(x)$ follows as $f_A(x)$ is irreducible, so the minimum polynomial of $\Frob_p$ is $f_A(x)$.
\end{proof}
Note that $K$ is a CM-field, where complex conjugation sends $x\mapsto p/x$. This action is the same as the Rosati involution on $\End^0(A)=K$ which is given by $\alpha \mapsto \lambda^{-1}\circ \alpha^\vee \circ \lambda$. Let $K^+$ be the totally real subfield of $K$ fixed by complex conjugation. Then, we can write $K^+\cong \Q[x]/f^+_A(x)$ where $f^+_A(x)$ is defined in Definition \ref{def: tot real poly}. Recall that $f_A(x)=x^gf_A^+(x+\frac p x)$ so the irreducibility of $f_A(x)$ implies that of $f_A^+(x)$.
\subsection{Generic Galois group} We investigate the Galois group $G$ of a generic irreducible ordinary Weil $p$-polynomial $f_{\mathbf a}(x)=x^{2g}+a_1x^{2g-1}+\cdots +p^g$, defined to be the Galois group of its splitting field $\widetilde K$. View $G$ as a subset of $S_{2g}$ via its action on the $2g$ embeddings $\iota \colon K\cong \Q[x]/f(x) \injects \widetilde K$ sending $x$ to each of the $2g$ roots of $f(x)$. The roots of $f(x)$ occur in conjugate pairs, so pair these embeddings accordingly as $\{\iota_1,\iota_2\},\ldots, \{\iota_{2g-1},\iota_{2g}\}$. Clearly, $G$ has to preserve these pairs, so we have $G\subseteq (\Z/2\Z)^g\rtimes S_g$, where the $(\Z/2\Z)^g$ action swaps the elements in each of the $g$ pairs and $S_g$ permutes the $g$ pairs. Intuitively, the generic Galois group should be as large as possible, and we will prove that this is indeed the case. Note that $(\Z/2\Z)^g\rtimes S_g$ is the Weyl group of the symplectic group.
\begin{proposition}\label{prop: generic galois}
There are at most an $O(p^{-1/4}\log p)$ proportion of irreducible ordinary Weil $p$-polynomials which do not have Galois group $(\Z/2\Z)^g\rtimes S_g$.
\end{proposition}
The paper \cite{Davis_Duke_Sun_1998} tackled a very similar problem, where they considered the set of reciprocal polynomials $g(x)=x^ng(1/x)$ with coefficients having absolute value less than $N$, and they proved that the proportion of exceptions is at most $O(N^{-1/2}\log N)$, essentially using a version of the large sieve. In our case, we have instead the multiplier $p$ case $f(x)=x^{2g}f(p/x)$ with the coefficients satisfying $(a_1,\ldots, a_g)\in [p](\mathcal R_g)$. One key difference is that in the first case $N$ is independent of the defining equation of the polynomials $g(x)=x^ng(1/x)$, while in our case both the region and the defining equation depends on $p$. Another difference is that in our case the region has side lengths with different orders of magnitude. Thus our case does not follow directly from \cite{Davis_Duke_Sun_1998}. Nevertheless, we are able to prove it by modifying their proof. Note that we change some of their variables in the following proof. 
\begin{proof}
Say that a polynomial $f(x)\in \F_q[x]$ has splitting type $l$ for the prime $q$ if its factorization into irreducibles mod $q$ consists of distinct factors which are all linear except one factor of degree $l$. For $q\neq p$, let $\Omega_l(q)$ be the subset of $\mathbf{a}\in \F_q^g$ such that the corresponding $f_{\mathbf a}(x)=x^{2g}+a_1x^{2g-1}+\cdots +p^g$ has splitting type $l$ for $q$. Note that the proof of Lemma 3 of loc. cit. works for all $q\neq p$ because we correspondingly have $f(x)=x^gf^+(x+\frac p x)$ and $p$ is a unit in $\F_q$. Thus, we have $|\Omega_l(q)|=0$ for odd $l$ and $|\Omega_{2k}(q)|=\frac{q^g}{2^{g-k+1}k(g-k)!}+O(q^{g-1})$ for $0\le k\le g$. For $q=p$, simply let $\Omega_l(p)$ be the empty set for all $l$.

Let $P_l(x)=\sum_{q\le x} |\Omega_l(q)|q^{-g}$. For $\mathbf a=(a_1,\ldots, a_g) \in \Z^g$, let $P_l(\mathbf a,x)$ be the number of primes $q\le x$ where $\mathbf a \pmod q$ is in $\Omega_l(q)$. Lemma 5 of loc. cit. gives $\sum_{|\mathbf a|\le N}(P_l(\mathbf a,x)-P_l(x))^2 = O(N^gP_l(x))$ for $N\ge x^2$. While Lemma 5 sums over the box $[-N,N]^g$, it is clear that it will also hold with any translate of $[-N,N]^g$ as there is nothing special about the origin. Hence, taking $N=p^{1/2}$ and $x=p^{1/4}$, we can cover the region $[p](\mathcal R_g)$ with $O(p^{g(g-1)/4})$ translates of the box to obtain $\sum_{\mathbf a\in [p](\mathcal R_g)}(P_l(\mathbf a,p^{1/4})-P_l(p^{1/4}))^2 = O(p^{g(g+1)/4}P_l(p^{1/4}))$.

Now let $f_{\mathbf a}(x)$ be an irreducible ordinary Weil $p$-polynomial. If it has Galois group $G$ which is a proper subgroup of $(\Z/2\Z)^g\rtimes S_g$, then by Lemma 2 of loc. cit., there exists some $l\in\{2,4,2g-2,2g\}$ where $G$ doesn't contain any $l$-cycle. This implies that $f$ cannot have splitting type $l$ for any prime $q$, and thus $P_l(\mathbf a,p^{1/4})=0$ (see the paragraph after Lemma 4 of loc. cit.). By our first paragraph and the prime number theorem we have $P_l(p^{1/4})=\Theta\left(\frac{p^{1/4}}{\log p}\right)$. Thus, by our second paragraph there can be at most $O(p^{(g^2+g-1)/4}\log p)$ such $f_{\mathbf a}(x)$ for each $l\in\{2,4,2g-2,2g\}$. Summing up across these four $l$s and dividing by $\Theta(p^{g(g+1)/4})$ as in Corollary \ref{cor: num ord weil poly} yields the result.
\end{proof}
We remark that \cite{anderson2024galoisgroupsreciprocalpolynomials} gives better quantitative bounds than \cite{Davis_Duke_Sun_1998}, although their proof is more complicated and hence more difficult to adapt to our situation.
\subsection{Existence of PPAVs} Before we start counting PPAVs, we first show that a generic isogeny class does contain PPAVs. The question of whether an isogeny class contains PPAVs is answered in \cite{Howe1995PrincipallyPO}, but the full condition is complicated. Here, we state the relevant part for our use.
\begin{proposition}[{\cite[Corollary 11.4]{Howe1995PrincipallyPO}}]\label{prop: PPAV exist condition}
An isogeny class of simple ordinary abelian varieties of dimension $g$ over $\F_p$ contains a PPAV if $g$ is odd or $K/K^+$ is ramified at some finite prime. 
\end{proposition}
The latter condition can be interpreted in the following way.
\begin{lemma}\label{lem: K/K+ ramified condition}
$K/K^+$ is ramified at some finite prime if the Galois group $G$ of $f_A(x)$ is not contained in the alternating group $A_{2g}$. 
\end{lemma}
\begin{proof}
We prove the contrapositive. If $K/K^+$ is unramified, then the relative discriminant $\Delta_{K/K^+}$ is the unit ideal. We can relate the absolute and relative discriminants by $\Delta_{K}=N_{K^+/\Q}(\Delta_{K/K^+})\Delta_{K^+}^2$, so the absolute discriminant $\Delta_{K}$ must be a square. Furthermore, $\disc(f)=\Delta_{K}[\mathcal O_K\colon \Z[\alpha]]^2$ where $\alpha$ is a root of $f_A(x)$, so $\disc(f)$ must be a square as well. By \cite[Theorem 4.7]{Conrad_Galois}, this implies that $G$ is contained in $A_{2g}$.
\end{proof}
It follows that an isogeny class of simple ordinary abelian varieties has a PPAV when $G\cong (\Z/2\Z)^g\rtimes S_g$. By Proposition \ref{prop: generic galois}, this is true generically -- there are at most $O(p^{-1/4}\log p)$ proportion of simple ordinary abelian varieties that do not contain a PPAV.
\section{PPAV counts for a generic characteristic polynomial} \label{sec: ppav count formula}
There has been much historical work done on counting $\F_p$-points over moduli spaces, with the most classical case of modular curves due to Deligne-Rapoport. On a much larger scale, the Langlands-Rapoport conjecture predicts a general formula for the $\F_p$-points over a general Shimura variety. Although this has not yet been proven in general, it is known for many families of Shimura varieties. This includes the family of Siegel Shimura varieties $\mathcal A_g$ which happens to be one of the simplest examples of Shimura varieties, which Kottwitz detailed in \cite[Part III]{Kottwitz_perspective}. We take this as our starting point for this section.
\subsection{Kottwitz's orbital integral formula} Let $R$ be a ring. First, we define the notion of an $R$-isogeny of both abelian varieties and polarized abelian varieties. 
\begin{definition}
Let $A$ and $A'$ be abelian varieties over $\F_p$. An \textit{$R$-isogeny of abelian varieties} from $A$ to $A'$ is an element $\phi\in \Hom_{\F_p}(A,A')\otimes_\Z R$ that is invertible, i.e. there exists an inverse $\psi\in \Hom_{\F_p}(A',A)\otimes_\Z R$ so that $\phi\circ \psi$ and $\psi\circ \phi$ are both the identity in $\End_{\F_p}(A',A')\otimes_{\Z} R$ and $\End_{\F_p}(A,A)\otimes_{\Z} R$ respectively. Let $[A,\lambda],[A',\lambda']$ be polarized abelian varieties. An \textit{$R$-isogeny of polarized abelian varieties} is an $R$-isogeny of abelian varieties $\phi\colon A\rightarrow A'$ where $\phi^*\lambda \coloneqq \phi^{\vee} \circ \lambda' \circ \phi =c\lambda\in \Hom_{\F_p}(A,A')\otimes_\Z R$ for some scalar $c\in R^\times$.
\end{definition}
We call a $\Q$-isogeny of abelian varieties an AV-isogeny, and a $\Q$-isogeny of polarized abelian varieties a PAV-isogeny, omitting AV or PAV when the context is clear. Also note that two abelian varieties are isogenous in the usual sense if there is a $\Q$-isogeny between them. Two polarized abelian varieties may not be PAV-isogenous, but may have underlying abelian varieties that are AV-isogenous.

Now, fix a polarized abelian variety $[A,\lambda]$. We wish to give a formula for the number of PPAVs $[A',\lambda']$ that are PAV-isogenous to it. Note that here $[A',\lambda']$ is principally polarized while $[A,\lambda]$ may not be. For $l\neq p$, consider the action of $\Frob_p$ on $H^1(A_{\bar\F_p},\Q_l)$ which is canonically isomorphic to the dual of $V_l=T_l\otimes_{\Z_l}\Q_l$ and hence is a $\Q_l$ module of rank $2g$. By choosing an isomorphism from $H^1(A_{\bar\F_p},\Q_l)$ to $\Q_l^{2g}$ which transports the Weil pairing induced by $\lambda$ to the standard symplectic form, the action of $\Frob_p$ gives a $\gamma_{[A,\lambda],l}\in \GSp_{2g}(\Q_l)$ up to conjugacy. Simultaneously considering all primes $l\neq p$ gives us $\gamma_{[A,\lambda]}^p\in \GSp_{2g}(\A_f^p)$. 

For $l=p$, we consider instead $H^1_{cris}(A,\Q_p)$ whose dual is canonically isomorphic to $D_0(A)=D(A)\otimes_{\Z_p}\Q_p$. Here, $D(A)$ is the Dieudonné module over the ring $\Z_p[F,V]$, it is free over $\Z_p$ with rank $2g$ and is equipped with a symplectic form. Then, the Frobenius action $F$ defines an element $\gamma_{[A,\lambda],p}$ of $\GSp_{2g}(\Q_p)$ with multiplier $p$ up to conjugacy. As a side note, the description in the $l=p$ case is more complicated over the setting of finite fields $\F_q$ compared to our case of prime fields $\F_p$, for example, one would need to consider twisted conjugacy instead of regular conjugacy. 

By considering lattices in the first cohomology groups of PPAVs, Kottwitz proved the following orbital integral formula, which we have simplified slightly in the setting of prime fields $\F_p$.
\begin{theorem}[{\cite[205]{Kottwitz_perspective}}]\label{thm: orbital integral}
Let $I_{[A,\lambda]}$ be the set of all PPAVs that are isogenous to $[A,\lambda]$. Then, its groupoid cardinality is
$$\# I_{[A,\lambda]} = \int_{T_{[A,\lambda]}(\Q)\backslash G(\A_f)} \mathbbm{1}_{G(\hat\Z_f^p)}(g^{-1}\gamma^p_{[A,\lambda]} g)\cdot \mathbbm{1}_{G(\Z_p)aG(\Z_p)}(g^{-1}\gamma_{[A,\lambda],p}g) \ dg$$
where $G=\GSp_{2g}$ and $a=\diag(p,\ldots ,p,1,\ldots ,1)$.
\end{theorem}
Here, we define the functor $T_{[A,\lambda]}(R)\coloneqq \{\alpha\in (\End_{\F_p}(A)\otimes R)^\times \mid \alpha\alpha^\dagger\in R^\times\}$ to be the group of $R$-isogenies from $[A,\lambda]$ to itself. Here, $\alpha^\dagger$ is the Rosati involution, which we recall is given by complex conjugation on the CM-field $K$. From this, it is easy to see that $T_{[A,\lambda]}$ is an algebraic torus over $\Q$ as
$$T_{[A,\lambda]}=\ker(\G_m\times R_{K/\Q}(\G_m)\rightarrow R_{K^+/\Q}(\G_m))$$
where $R_{K/L}$ is the restriction of scalars functor and the map sends $(x,y)\mapsto x^{-1}N_{K/K^+}(y)$. 

In the case when $[A,\lambda]$ is not principally polarized, it is possible that the orbital integral could evaluate to zero if there are no PPAVs in its PAV-isogeny class. On the other hand, when $[A,\lambda]$ is principally polarized, the orbital integral must be positive. Another way to see this is as follows. Note that $\Frob_p$ acts on $H^1(A_{\bar\F_p},\Z_l)$, and because $\lambda$ is a principal polarization we can choose the isomorphism from $H^1(A_{\bar\F_p},\Q_l)$ to $\Q_l^{2g}$ to send $H^1(A_{\bar\F_p},\Z_l)$ to $\Z_l^{2g}$. Hence, we can pick a representative $\gamma_{[A,\lambda],l}\in \GSp_{2g}(\Z_l)$, and by simultaneously considering all primes $l\neq p$ can choose a representative $\gamma_{[A,\lambda]}^p\in \GSp_{2g}(\hat\Z_f^p)$. Similarly, at the prime $p$, using Dieudonné modules we can choose a representative $\gamma_{[A,\lambda],p}\in \GSp_{2g}(\Z_p)a\GSp_{2g}(\Z_p)$, so the orbital integral formula is positive.
\subsection{Reinterpretation as an Euler product}\label{sec: reinterpretation as euler product}
It is difficult to use Kottwitz's orbital integral formula directly for counting. In \cite{Achter_2023}, Achter-Altuğ-Garcia-Gordon reinterpreted this as a much more concrete Euler product, which we state in the theorem below. However, there is an error in the paper which leads to one of the definitions of $v_l([A,\lambda])$ being incorrect. We take care to work around this problem, and this is explained more in Remark \ref{rmk: AAGG problem} and the Appendix. 
\begin{theorem}[{\cite[Theorem A]{Achter_2023}}]\label{thm: AAGG} Suppose that $[A,\lambda]$ is a PPAV over $\F_p$ with commutative endomorphism ring (for example, this holds if $A$ is a simple and ordinary). Then,
$$\# I_{[A,\lambda]} = p^{\dim(\mathcal A_g)/2}\tau_Tv_{\infty}([A,\lambda])\prod_l v_l ([A,\lambda]).$$
\end{theorem}

We will use this in the case when $A$ is simple and ordinary. We explain the factors appearing in Theorem \ref{thm: AAGG}. Then, $\tau_T$ is the Tamagawa number associated to the algebraic torus $T_{[A,\lambda]}$. By \cite{Ono}, we can compute this via the formula
\begin{equation}\label{eqn: tamagawa number definition}
\tau_T = \frac{|H^1(\Q,X^*(T_{[A,\lambda]})|}{|\Sha^1(T_{[A,\lambda]})|}
\end{equation}
where $\Sha^1(T_{[A,\lambda]})=\ker(H^1(\Q,T)\rightarrow \prod_v H^i(\Q_v,T))$ where the product is over all places of $\Q$, and $X^*(T_{[A,\lambda]})$ is the character lattice of $T_{[A,\lambda]}$.

The local factor at the archimedean place $\infty$ is 
\begin{equation}\label{eqn: vinf definition}
v_{\infty}([A,\lambda])=v_{\infty}(f_A)=\frac{\sqrt{|D(f_A)|}}{(2\pi)^g}
\end{equation}
where the Weyl discriminant $D(f_A)\coloneqq \prod_{1\le i < j \le g} (2\cos(\theta_i)-2\cos(\theta_j))^2 \prod_{1\le i\le g} 4 \sin^2(\theta_j)$ where $\sqrt p e^{i\theta_1},\sqrt p e^{-i\theta_1},\ldots \sqrt p e^{i\theta_g},\sqrt p e^{-i\theta_g}$ are the roots of $f_A(x)$. Another expression for the Weyl discriminant is $D(f_A)=(-1)^gp^{-g(3g-1)/2}\disc(f_A)/\disc(f^+_A)$, see Section 6A of loc. cit. \cite{Achter_2023}.

The problem in the paper arises in the non-archimedean local places $v_l([A,\lambda])$, as we explain here.
\begin{remark}\label{rmk: AAGG problem}
Combining Definition 4.1 and Corollary 3.4 in loc. cit., the \textit{incorrect} local factor at the non-archimedean place $l$ was written explicitly as
$$v_l([A,\lambda]) = \lim_{k\rightarrow \infty}\frac{\# \pi_k\left(\Mat_{2g}(\Z_l)\cap \prescript{\GSp_{2g}(\Q_l)}{}{\gamma_{[A,\lambda],l}}\right)}{\#\GSp_{2g}(\Z_l/l^k)/(l^{kg}\psi(l^k))}$$
where $\prescript{\GSp_{2g}(\Q_l)}{}{\gamma_{[A,\lambda],l}}$ is the conjugacy class of $\gamma_{[A,\lambda],l}$ in $\GSp_{2g}(\Q_l)$, and $\pi_k\colon M_{2g}(\Z_l)\rightarrow M_{2g}(\Z/l^k\Z)$ is the projection to $\Z/l^k\Z$. Note that when $l\neq p$, we can replace $\Mat_{2g}(\Z_l)$ with $\GSp_{2g}(\Z_l)$ because the conjugacy class lies in $\GSp_{2g}$ as the multiplier is coprime to the characteristic.

We give a counterexample in which case this local factor is incorrect. Take $g=1$, so $\GSp_{2g}=\GL_2$. Note that in this particular case, two semisimple elements are conjugate in $\GL_2(\Q_l)$ if and only if they have the same characteristic polynomial. Picking the characteristic polynomial $x^2+5x+4$ and $l=3$, the \textit{incorrect} local factor above is 
$$\lim_{k\rightarrow \infty}\frac{\# \pi_k\left(\{\gamma \in GL_{2}(\Z_3)\mid \tr(\gamma)=5, \det(\gamma)=4\}\right)}{3^{2k-2}(3^2-1)}$$
On the other hand, following \cite[Equation (3.7)]{Gekeler2003FrobeniusDO}, the \textit{correct} local factor is given by $$\lim_{k\rightarrow \infty}\frac{\#\left\{\gamma \in GL_{2}(\Z/3^k\Z)\mid \tr(\gamma)=5, \det(\gamma)=4\right\}}{3^{2k-2}(3^2-1)}.$$
The difference is that the first counts projections from $\Z_3$ to $\Z/3^k\Z$, while the second directly counts matrices in $\Z/3^k\Z$. It turns out that in this case (more specifically when $l$ divides the discriminant), not every matrix in $\Z/3^k\Z$ can be lifted to $\Z_3$ such that the characteristic polynomial remains the same. Indeed, running some simple code we see that the first local factor is $7/6$ while the second is $3/2$. Furthermore, in the $\GL_2$ case, the PAV-isogeny class is the same as the AV-isogeny class (see discussion in Section \ref{sec: kottwitz triples}), so the formula in Theorem \ref{thm: AAGG} should yield the same count as \cite[Corollary 5.4]{Gekeler2003FrobeniusDO}, but we find that this is not the case.

We pinpoint the problem to the proof of Lemma 3.2 in loc. cit., more specifically the part which uses Hensel's lemma. The counterexample above can be modified to a counterexample of Lemma 3.2. Fortunately, this mistake is not fatal to our application, as what we need are stable orbital integrals ($\GSp_{2g}(\bar \Q_l)$ conjugacy classes) instead of rational orbital integrals ($\GSp_{2g}(\Q_l)$ conjugacy classes). Gordon and Achter have kindly written an appendix elaborating on this remark, as well as explaining how stable orbital integrals can be written explicitly in terms of Gekeler-style ratios following \cite{Gekeler2003FrobeniusDO} which we require later.
\end{remark}
Even though the explicit formulation of $v_l([A,\lambda])$ is incorrect, the arguments relating the orbital integral to the geometric measure is still correct, see Section 5 of loc. cit.. Thus, the product formula in Theorem \ref{thm: AAGG} is correct if we use the following definitions of the local factors in Section 5C of loc. cit.:
\begin{equation}\label{eqn: vl ppav definition}
\begin{split}
v_l([A,\lambda])&=\frac{l^{\dim(\Sp_{2g})}}{\#\Sp_{2g}(\F_l)}\mathcal O^{\geom}_{\gamma_{[A,\lambda],l}}(\mathbbm{1}_{\GSp_{2g}(\Z_l)}) \text{ for } l\neq p,\\
v_p([A,\lambda])&=p^{-g(g+1)/2}\frac{p^{\dim(\Sp_{2g})}}{\# \Sp_{2g}(\F_p)}\mathcal O^{\geom}_{\gamma_{[A,\lambda],p}}(\mathbbm{1}_{G(\Z_p)aG(\Z_p)}),
\end{split}
\end{equation}
where the rational orbital integral $\mathcal O^{\geom}_{\gamma}$ is defined at the end of Section 2E of loc. cit.. This definition is rather technical, and we will only use it in the proof of Lemma \ref{lem: rational orbits to stable orbits}. Note that either definition is completely defined by the $\GSp_{2g}(\Q_l)$ conjugacy class of $\gamma_{[A,\lambda],l}$, so we define $v_l(\gamma)$ (including for $l=p$) by replacing $\gamma_{[A,\lambda],l}$ with $\gamma\in \GSp_{2g}(\Q_l)$ in the definitions above so that $v_l([A,\lambda])=v_l(\gamma_{[A,\lambda],l})$.

Here, we take the opportunity to define and clarify all the local factors that we use in this paper. We just defined the two archimedean local factors in Equation \eqref{eqn: vinf definition}. The archimedean local factor associated to trace is $\ST_g$ as defined in the introduction. 

We have three types of non-archimedean local factors, $v_l(t)$, $v_l(f)$ and $v_l([A,\lambda])$, associated to the $p$-adic density of matrices with a particular trace, characteristic polynomial and PAV-isogeny class respectively, and $v_l([A,\lambda])$ was defined in Equation \eqref{eqn: vl ppav definition}. For $v_l(f)$, let $M^*\subseteq \Mat_{2g}\times \A^1$ be the closed scheme cut out by the symplectic equation $A\Omega A^T = m \Omega$ where $\Omega = \begin{pmatrix} 0 & I_g\\ -I_g & 0
\end{pmatrix}$, and $M$ be the open subscheme of $M^*$ for which the multiplier $m\neq 0$. Over $\Z_l$, we have $M(\Z_l)=\GSp_{2g}(\Q_l)\cap \Mat_{2g}(\Z_l)$. For each prime $l$ (including $l=p$), we define
\begin{equation}\label{eqn: vl t and f definition}
\begin{split}
v_l(f)&=\lim_{k\rightarrow \infty}\frac{\#\{\gamma \in M(\Z/l^k\Z) \mid f_\gamma\equiv f \}}{\#\GSp_{2g}(\Z/l^k\Z)/(l^{gk}\phi(l^k))}   .
\end{split}
\end{equation}
Once again, when $l\neq p$ we can replace $M$ with $\GSp_{2g}$. We defined $v_l(t)$ in the introduction, here we remark on the comment of $v_p(t)=1$.
\begin{remark}\label{rmk: vp choice}
Based on Equation \eqref{eqn: vl t and f definition}, morally we \textit{should} define
$$v_l(t)=\lim_{k\rightarrow \infty}\frac{\#\{\gamma \in M(\Z/l^k\Z) \mid \tr(\gamma)=t, \mult(\gamma)=p \}}{\#\GSp_{2g}(\Z/l^k\Z)/(l^{k}\phi(l^k))}.$$
This agrees with our definition for $l\neq p$, but for $l=p$ we do not do this and instead defined $v_p(t)=1$. This is for two reasons: firstly, because we are in the prime field case $\F_p$, any factor $v_p(t)$ that satisfies $v_p(t)\rightarrow 1$ as $p\rightarrow \infty$ would give the convergence in Theorem \ref{thm: M_2, M_3 sym} and \ref{thm: A_g}, so it does not matter which precise local factor $v_p(t)$ we choose. The second reason is that because the locally closed subscheme $M$ is much more complicated than $\GSp_{2g}$, we are unable to prove bounds on the ratio given above, although we expect it to be $1+O(p^{-2})$ like in Proposition \ref{prop: vl(t) = 1 + O(l^-2)}.
\end{remark}
The following lemma relates $v_l([A,\lambda])$ to $v_l(f)$, and this essentially uses the fact that stable orbital integrals are the disjoint union of rational orbital integrals.
\begin{lemma}\label{lem: rational orbits to stable orbits}
Suppose $\gamma\in \GSp_{2g}(\Q_l)$ is regular semisimple, then we have $v_l(f_{\gamma})=\sum_{\gamma'}v_l(\gamma')$ (including $l=p$) where the sum is over representatives $\gamma'$ of all $\GSp_{2g}(\Q_l)$-conjugacy classes with characteristic polynomial $f_{\gamma'}=f_\gamma$.
\end{lemma}
\begin{proof}
Appendix \ref{sec: appendix} is dedicated to the proof of this Lemma.
\end{proof}
\subsection{Conjugacy classes in reductive groups} We record some facts on conjugacy classes in reductive groups, which we will need later in the case of $\GSp_{2g}$. Let $G$ be a reductive group over a field $k$, and let $T$ be a maximal torus of $G$. The Weyl group $W=N(T)/T$ acts on $T$, and let $T/W\coloneqq \Spec(k[T]^W)$. 
\begin{proposition}\label{prop: steinberg conjugacy}
There exists a map $$\chi_k\colon \{\text{semisimple conjugacy classes of }G(k)\}\rightarrow (T/W)(k)$$
with the following properties.
\begin{enumerate}[(\alph*)]
\item (Steinberg \cite{Steinberg_1974}) If $k=\bar k$, then $\chi_k$ is bijective.
\item (Kottwitz \cite{Kottwitz_conj_classes}) If $G$ is a quasi-split group with $G^{der}$ simply connected, then $\chi_k$ is surjective.
\end{enumerate}
\end{proposition}
On the issue of injectivity, suppose that $\gamma_0,\gamma\in G(k)$ are semisimple elements where $\chi_k(\gamma_0)=\chi_k(\gamma)$, which is equivalent to saying that $\gamma_0,\gamma$ are conjugate in $G(\bar k)$. Then, there exists $g\in G(\bar k)$ such that $g\gamma_0 g^{-1}=\gamma$, so for every $\sigma \in \Gal(\bar k/k)$ we have $\sigma(g)\gamma_0 \sigma(g)^{-1}=\gamma$ and thus $g^{-1}\sigma(g)\in G_{\gamma_0}(\bar k)$. This defines a 1-cocycle $\sigma\mapsto g^{-1}\sigma(g)$, which gives a class $\inv(\gamma_0,\gamma)\in H^1(k,G_{\gamma_0})$ which is trivial in $H^1(k,G)$. The following proposition shows this class characterizes the semisimple conjugacy class.
\begin{proposition} (Kottwitz \cite{Kottwitz_conj_classes})\label{prop: preimage of steinberg}
Suppose $G^{der}$ is simply connected and $\gamma_0\in G(k)$ is semisimple. Then, the map $\gamma\mapsto \inv(\gamma_0,\gamma)$ gives a bijection from the set of semisimple conjugacy classes of $G(k)$ which are $G(\bar k)$ conjugate to $\gamma_0$ (i.e. $\chi_k^{-1}(\chi_k(\gamma_0))$) to $\ker(H^1(k,G_{\gamma_0})\rightarrow H^1(k,G))$.
\end{proposition}
Now, we specialize to the case of $\GSp_{2g}$, a connected, split (hence quasi-split), reductive group. It has derived group $\Sp_{2g}$ which is simply connected. Thus, the conditions in both Proposition \ref{prop: steinberg conjugacy} and \ref{prop: preimage of steinberg} hold. It is easy to see that in our case, $(T/W)(k)$ is in bijection with the monic symplectic polynomials $f(t)\in k[t]$ of degree $2g$ with some multiplier $c\in k^\times$, and correspondingly $\chi_k$ sends a matrix to its characteristic polynomial. Hence, by applying part (b) of Proposition \ref{prop: steinberg conjugacy} for $k=\Q$ and part (a) for $k=\bar \Q$, we get the following.
\begin{corollary}\label{cor: char poly to semisimple elt}
Let $f(x)$ be an ordinary Weil $p$-polynomial. Then, there is a semisimple element $\gamma\in \GSp_{2g}(\Q)$ which has characteristic polynomial $f(x)$ which is unique up to $\GSp_{2g}(\bar \Q)$ conjugacy.
\end{corollary}
We define $\gamma_A$ to be a representative of the conjugacy class corresponding to $f_A(x)$.
\subsection{Kottwitz triples and invariants}\label{sec: kottwitz triples}
By Theorem \ref{thm: orbital integral} and \ref{thm: AAGG}, we know the number of PPAVs in a PAV-isogeny class. However, because of the parametrization in Section \ref{sec: parametrization}, we would like to know the number of PPAVs in an AV-isogeny class (i.e. those that have some characteristic polynomial). The notion of Kottwitz triples allows us to bridge between these notions. Note that for our setting of prime fields $\F_p$, the place $p$ behaves very similarly to places $l\neq p$, but this is not true for finite fields $\F_q$, which requires the notion of twisted conjugacy. Since we are only dealing with prime fields, we ignore twisted conjugacy and simplify all definitions for the specific case of prime fields.
\begin{definition}[{\cite[206-207]{Kottwitz_perspective}}]
A \textit{Kottwitz triple} $(\gamma_0,\gamma,\delta)$ is a triple of elements where 
\begin{enumerate}[(\alph*)]
    \item $\gamma_0\in \GSp_{2g}(\Q)$ is a semisimple element defined up to $\GSp_{2g}(\bar\Q)$ conjugacy, where $\gamma_0=\gamma_A$ for some $g$-dimensional abelian variety $A$,
    \item $\gamma\in \GSp_{2g}(\A^p_f)$ up to $\GSp_{2g}(\A^p_f)$ conjugacy, such that for every $l\neq p$, the $l$-component $\gamma_l$ is conjugate to $\gamma_0$ in $\GSp_{2g}(\bar \Q_l)$, 
    \item $\delta\in \GSp_{2g}(\Q_p)$ up to conjugacy, where $\delta$ is conjugate to $\gamma_0$ in $G(\bar \Q_p)$.
\end{enumerate} 
\end{definition}
As we saw previously, from a PAV-isogeny class $[A,\lambda]$ we can get the elements $\gamma_A,\gamma_{[A,\lambda]}^p$ and $\gamma_{[A,\lambda],p}$, one can check these satisfy the conditions required for a Kottwitz triple because they share the same characteristic polynomial. Hence, we get a map 
$$\Phi\colon \{\text{PAV-isogeny classes}\} \rightarrow \{\text{Kottwitz triples}\}.$$

We restrict our attention to a fixed AV-isogeny class of ordinary abelian varieties (i.e. fixed characteristic polynomial), and let $[A,\lambda]$ be a polarized abelian variety such that $A$ is in this isogeny class. Restricting the map $\Phi$ to this AV-isogeny class in turn restricts the Kottwitz triples to those satisfying $\gamma_0=\gamma_A$, which gives the following
\begin{equation}\label{eqn: Phi_A}
\begin{tikzcd}
	{\Phi_A\colon\{\text{PAV-isogeny classes with char poly }f_A \}} & {\{\text{Kottwitz triples }(\gamma_A,\gamma,\delta)\}} \\
	{\ker(H^1(\Q,T_{[A,\lambda]})\rightarrow H^1(\R,T_{[A,\lambda]}))} & {\prod_l H^1(\Q_l,T_{[A,\lambda]})}
	\arrow[from=1-1, to=1-2]
	\arrow["\cong"', tail reversed, from=1-1, to=2-1]
	\arrow[hook, from=1-2, to=2-2]
	\arrow[from=2-1, to=2-2]
\end{tikzcd}
\end{equation}
where the bottom arrow is the natural map. From this diagram, we see that every non-empty preimage has cardinality $|\Sha^1(T_{[A,\lambda]})|$, see also \cite[207]{Kottwitz_perspective}.

We explain the vertical maps on both sides. On the left, Kottwitz showed that $[A',\lambda']$ is $\bar\Q$-isogenous to $[A,\lambda]$ if and only if the underlying abelian varieties $A$ and $A'$ are $\Q$-isogenous, and that the set of $\Q$-isogeny classes of $[A',\lambda']$ which are $\bar \Q$ isogenous to $[A,\lambda]$ is in canonical bijection with $\ker(H^1(\Q,T_{[A,\lambda]})\rightarrow H^1(\R,T_{[A,\lambda]}))$ \cite[205-206]{Kottwitz_perspective}. On the right, by Proposition \ref{prop: preimage of steinberg}, for each $l\neq p$, the possible choices of $\GSp_{2g}(\Q_l)$ conjugacy classes for the $l$-component $\gamma_l$ which is $\GSp_{2g}(\bar \Q_l)$ conjugate to $\gamma_{[A,\lambda],l}$ is in canonical bijection to $\ker(H^1(\Q_l,T_{[A,\lambda]})\rightarrow H^1(\Q_l,\GSp_{2g}))\subseteq H^1(\Q_l,T_{[A,\lambda]})$. Here, we note that by Tate's theorem \cite{Tate_End}, $T_{[A,\lambda]}$ is isomorphic to the centralizer of $\gamma_{[A,\lambda],l}$. The same applies to $\delta$ over the prime $p$. 

To describe the image of $\Phi_A$, Kottwitz defined an invariant $\alpha(\gamma_A,\gamma,\delta)\in \mathfrak K(T_{[A,\lambda]})^D$ which vanishes if and only if the triple $(\gamma_A,\gamma,\delta)$ is in the image of $\Phi_A$ \cite[Theorem 12.1]{Kottwitz_perspective}. Here, $D$ denotes the Pontryagin dual. We omit the construction of this invariant and refer the reader to \cite[166-168]{Kottwitz_perspective}. The only fact we will need from this construction is that $\mathfrak K(T_{[A,\lambda]})$ is a subgroup of $\pi_0((Z(\hat T_{[A,\lambda]})/Z(\hat\GSp_{2g}))^\Gamma)$, where $\Gamma=\Gal(\bar \Q/\Q)$ is the $L$-action on the dual groups.

We prove the following proposition, which is absolutely crucial for our proof. 
\begin{proposition}\label{prop: trivial kottwitz invariant}
Suppose that $A$ is a simple, ordinary abelian variety whose Galois group is the generic group $(\Z/2\Z)^g\rtimes S_g$. Then, the group $\pi_0((Z(\hat T_{[A,\lambda]})/Z(\hat \GSp_{2g}))^\Gamma)$ is trivial, so all the invariants $\alpha(\gamma_A,\gamma,\delta)$ are also trivial and $\Phi_A$ is surjective.
\end{proposition}
\begin{proof}
The key ingredient required is the large Galois group, which forces the group of invariants to be small. Recall that $T_{[A,\lambda]}=\ker(\G_m\times R_{K/\Q}(\G_m)\rightarrow R_{K^+/\Q}(\G_m))$ is a maximal torus in $\GSp_{2g}$. Over the algebraic closure, we have $T_{[A,\lambda]}(\bar \Q) = \ker\left(\bar \Q^\times \times (\bar\Q^\times)^{2g}\rightarrow (\bar\Q^\times)^g\right)$, where the $2g$ copies of $\bar\Q^\times$ correspond to the embeddings of $K\injects \bar \Q$, and likewise the $g$ copies correspond to the embeddings of $K^+ \injects \bar \Q$. Index the $2g$ copies of $\bar\Q^\times$ such that every two consecutive copies correspond to complex conjugate embeddings. One can check that the map above sends $(a,b_i)_{i\in [2g]}\mapsto (a^{-1}b_{2j-1}b_{2j})_{j\in [g]}$, and that the Galois action is given by first acting on the scalars in each entry, and then permuting the entries corresponding to the Galois action on the embeddings. 

Let $\sigma\in \Gal(\bar \Q/\Q)$ act on the $2g$ embeddings by the permutation $(1\ 2)^{\epsilon_1}\cdots (2g-1\ 2g)^{\epsilon_g}f(\rho)$ for some $\epsilon_i\in \{0,1\}$ and $\rho\in S_g$, where $f(\rho)$ sends the elements $2i-1$ and $2i$ to $2\rho(i)-1$ and $2\rho(i)$. The cocharacter lattice $X_*(T_{[A,\lambda]})$ has elements $t\mapsto (t^n,t^{m_1},t^{-n-m_1},\cdots t^{m_g},t^{-n-m_g})$, so it is isomorphic to $\Z^{g+1}$ with coordinates $(n,m_1,\ldots, m_g)$. Then, $\sigma$ sends $(n,m_1,\ldots, m_g)\mapsto (n,m_1^\sigma,\ldots, m_g^\sigma)$ where $m_i^\sigma=m_{\rho^{-1}(i)}$ if $\epsilon_i=0$ and $m_i^\sigma=n-m_{\rho^{-1}(i)}$ if $\epsilon_i=1$. 

Using this, we obtain the dual group $\hat T_{[A,\lambda]}(\bar\Q)=\Hom(X_*(T_{[A,\lambda]}),\bar\Q^\times)$, whose elements are given by $(n,m_1,\ldots, m_g)\mapsto a^nb_1^{m_1}\cdots b_g^{m_g}$. The $L$-action by $\sigma$ is given by precomposition, so we get $(n,m_1,\ldots,m_g)\mapsto (n,m_1^\sigma,\ldots,m_g^\sigma)\mapsto a^nb_1^{m_1^\sigma}\cdots b_g^{m_g^\sigma}=(ab_1^{\epsilon_1}\cdots b_g^{\epsilon_g})^n(b_{\rho(1)}^{(-1)^{\epsilon_{\rho(1)}}})^{m_1}\cdots (b_{\rho(g)}^{(-1)^{\epsilon_{\rho(g)}}})^{m_g}$. Hence, $\sigma$ acts on $\hat T_{[A,\lambda]}(\bar \Q)\cong (\bar \Q^\times)^{g+1}$ by sending $a\mapsto ab_1^{\epsilon_1}\cdots b_g^{\epsilon_g}$ and $b_i\mapsto b_{\rho(i)}^{(-1)^{\epsilon_{\rho(i)}}}$.

Let us first find the $\Gamma$-fixed points of $\hat T_{[A,\lambda]}$. Because $A$ has generic Galois group $(\Z/2\Z)^g\rtimes S_g$, for any choice of $\epsilon_i,\rho$ there will exist some $\sigma \in \Gal(\bar\Q,\Q)$ acting by the corresponding permutation on the embeddings. By picking $\epsilon_i=0$ and $\rho$ to be a cyclic permutation, we get $b_1=\cdots =b_g$. Then, by picking $\epsilon_1=1$, $\epsilon_2=\cdots=\epsilon_g=0$ and $\rho=\id$, we have $b_1=1$. Hence, the fixed points $\hat T_{[A,\lambda]}^\Gamma$ are $\{(a,1,\ldots,1)\mid a\in \bar \Q^\times\}$. Note that $Z(\hat T_{[A,\lambda]})=\hat T_{[A,\lambda]}$ because $\hat T_{[A,\lambda]}$ is a torus. 

Lastly, $Z(\hat\GSp_{2g})=Z(\text{GSpin}_{2g+1})$ is one-dimensional and lies inside $Z(\hat T_{[A,\lambda]})$ as $\{(a,1,\ldots,1)\mid a\in \bar \Q^\times\}$, so we have that $\pi_0((Z(\hat T_{[A,\lambda]})/Z(\hat \GSp_{2g}))^\Gamma)$ is trivial.
\end{proof}
We remark that similar computations were done in \cite{Rüd_2022_explicit} and \cite[Appendix]{Achter_2023}, although not in the setting of the generic Galois group. One could also loosen the condition on the Galois group. More importantly, this computation also shows us that the numerator in the Tamagawa number is $1$.
\begin{corollary}\label{cor: tamagawa numerator}
Suppose that $A$ is a simple, ordinary abelian variety with generic Galois group $(\Z/2\Z)^g\rtimes S_g$. Then, $H^1(\Q,X^*(T_{[A,\lambda]}))$ is trivial.     
\end{corollary}
\begin{proof}
Kottwitz's isomorphism gives $H^1(\Q,X^*(T_{[A,\lambda]}))\cong \pi_0(Z(\hat T_{[A,\lambda]})^\Gamma)$ \cite[Equation (2.4.1), Section 2.4.3]{Kottwitz_cuspidal}, and by the proof of Proposition \ref{prop: trivial kottwitz invariant} we have that $Z(\hat T_{[A,\lambda]})^\Gamma=\{(a,1,\ldots,1)\mid a\in \bar \Q^\times\}$ is connected.
\end{proof}
\subsection{Proof of Theorem \ref{thm: PPAV counts for char poly}} We end this section by proving the theorem on PPAV counts for a characteristic polynomial. The fact that there are a nonzero number of PPAVs follows from the generic Galois group, Proposition \ref{prop: PPAV exist condition} and Lemma \ref{lem: K/K+ ramified condition}. Let $[A,\lambda]$ be a polarization of $A$. By Theorem \ref{thm: AAGG} and Corollary \ref{cor: tamagawa numerator}, 
$$\# I_{[A,\lambda]} = \frac{p^{\dim(\mathcal A_g)/2}}{|\Sha^1(T_{[A,\lambda]})|}v_{\infty}(f_A)\prod_l v_l (\gamma_{[A,\lambda],l}).$$
We wish to sum up all $\# I_{[A',\lambda']}$ over all PAV-isogeny classes $[A',\lambda']$ which have characteristic polynomial $f_A$. By Proposition \ref{prop: trivial kottwitz invariant}, the map $$\Phi_A\colon\{\text{PAV-isogeny classes with char poly }f_A \} \rightarrow \{\text{Kottwitz triples }(\gamma_A,\gamma,\delta)\}$$
is surjective and $|\Sha^1(T_{[A,\lambda]})|$-to-$1$. Hence, applying Lemma \ref{lem: rational orbits to stable orbits} we have  
\begin{equation*}
\begin{split}
\sum_{[A',\lambda']} \# I_{[A',\lambda']} &= |\Sha^1(T_{[A,\lambda]})|\sum_{(\gamma_l)} \frac{p^{\dim(\mathcal A_g)/2}}{|\Sha^1(T_{[A,\lambda]})|}v_{\infty}(f_A)\prod_l v_l (\gamma_{l})\\
&= p^{\dim(\mathcal A_g)/2} v_{\infty}(f_A)\prod_l v_l(f_A)
\end{split}
\end{equation*}
as desired. Here, the tuple $(\gamma_l)$ includes the place $p$ where $\gamma_p=\delta$, and the products over $l$ include $p$ too. One could worry about commuting infinite products and sums, but this is not a problem here as there are only a finite number of PAV-isogeny classes which contribute a nonzero count.

\section{Properties of local factors} \label{sec: local factors}
In this section we prove some technical lemmas about the local factors $v_\infty(f)$, $v_l(f)$ and $v_l(t)$, which we will need in the later sections. Let $A$ be a simple ordinary abelian variety with characteristic polynomial $f_A$, which we denote as $f$ for convenience.
\subsection{$v_\infty(f)$} Recall that we defined the Steinberg map $\pi \colon \USp_{2g}\rightarrow \mathcal R_g\subseteq \R^g$ to send a matrix to the coefficients of the characteristic polynomial in Lemma \ref{lem: R_g Usp}. Let $\mu_{\mathcal R_g}=\pi_* \mu_{\USp_{2g}}$ be the pushforward of the Haar measure on $\USp_{2g}$. Note that the trace map $\tr\colon\USp_{2g}\rightarrow \R$ is obtained by projecting on the first entry $p_1\circ \pi$, so we can write the Sato-Tate measure as $\mu_{\ST_g}=(p_1)_*\mu_{\mathcal R_g}$. The following proposition gives an important interpretation of archimedean local factor $v_\infty(f)=\sqrt{|D(f)|}/(2\pi)^g$ as the density function of the measure $\mu_{\mathcal R_g}$, hence relating it to the Sato-Tate measure.
\begin{proposition}\label{prop: vinf and measure}
For $\mathbf b \in \mathcal R_g$, we have $$\mu_{\mathcal R_g}= \frac{\sqrt{\left|D\left(f_{[p](\mathbf b)}\right)\right|}}{(2\pi)^g} db_1\cdots db_g,$$
where $f_{[p](\mathbf b)}=x^{2g}+b_1p^{1/2}x^{2g-1}+\cdots +p^g$ was defined in Definition \ref{def: R_g}.
\end{proposition}
\begin{proof}
By the Weyl integration formula \cite[Section 5.0.4]{KS99},  the Haar measure of the part of $\USp_{2g}$ that corresponds to the angles being inside $(\theta_1,\theta_1+d\theta_1),\ldots,(\theta_g,\theta_g+d\theta_g)$ up to permutation is $\prod_i (\frac 2 \pi\sin^2(\theta_i))\prod_{i<j}(2\cos(\theta_i)-2\cos(\theta_j))^2$. By Lemma \ref{lem: R_g Usp} any element in $\USp_{2g}$ has a characteristic polynomial $g_{\mathbf b}(x)$ with $\mathbf b\in \mathcal R_g$ with roots $e^{\pm i\theta_1},\ldots,e^{\pm i\theta_g}$, so when we translate the Haar measure to the coefficients $b_i$, we need to divide by $|\det((\frac{db_j}{d\theta_k})_{j,k})|$, and it suffices to compute this.

Recall that $g^+_{\mathbf b}(x)=x^g+c_1x^{g-1}+\cdots+c_g$ has roots $e^{i\theta_1}+e^{-i\theta_1},\ldots,e^{i\theta_g}+e^{-i\theta_g}$, and from the proof of Lemma \ref{lem: explicit R_g} the coefficients $c_i$ are given by $b_i$ plus a linear combination of $b_{i-2},b_{i-4},\ldots$, which implies that $\det((\frac{db_j}{d\theta_k})_{j,k})=\det((\frac{dc_j}{d\theta_k})_{j,k})$. Using Vieta's formulas, we obtain
$$\frac{dc_j}{d\theta_k}=\frac{d}{d\theta_k}\left(\sum_{\substack{S\subseteq [g]\\|S|=j}}\prod_{s\in S}(e^{i\theta_s}+e^{-i\theta_{s}})\right)=- \sin(\theta_k)2^{j} \sum_{\substack{S\subseteq [g]\setminus \{k\}\\|S|=j-1}}\prod_{s\in S} \cos(\theta_s).$$
After factoring both $-\sin(\theta_k)$ and $2^j$ out of the matrix we have
$$\det\left(\left(\frac{dc_j}{d\theta_k}\right)_{j,k}\right)=\left(\prod_{i=1}^g- 2^i\sin(\theta_i)\right)\cdot\det\begin{pmatrix}
1 & x_2+x_3+\cdots+x_g & \cdots & x_2x_3\cdots x_g \\
1 & x_1+x_3+\cdots +x_g & \cdots & x_1x_3\cdots x_g \\
\vdots & \vdots & \ddots & \vdots \\
1 & x_1+x_2+\cdots +x_{g-1} & \cdots & x_1x_2\cdots x_{g-1}
\end{pmatrix},$$
where we write $x_i=\cos(\theta_i)$. To evaluate the determinant of the matrix on the right, we subtract the first row from the other rows to get
$$\begin{pmatrix}
1 & x_2+x_3+\cdots+x_g & \cdots & x_2x_3\cdots x_g \\
0 & x_1-x_2 & \cdots & (x_1-x_2)x_3\cdots x_g \\
\vdots & \vdots & \ddots & \vdots \\
0 & x_1-x_g & \cdots & (x_1-x_g)x_2\cdots x_{g-1}
\end{pmatrix}.$$
Factoring out the $x_1-x_i$ factor for rows $2\le i\le g$, we realize that the lower right $(g-1) \times (g-1)$ submatrix is of the same form as the first matrix with the variables $x_2,\ldots, x_g$ instead. Using this, we inductively compute the determinant to be $\prod_{i<j}(x_i-x_j)$. Putting everything together and recalling that $D(f_{[p](\mathbf b)})=\prod_{1\le i < j \le g} (2\cos(\theta_i)-2\cos(\theta_j))^2 \prod_{1\le i\le g} 4 \sin^2(\theta_j)$, we have that the density function of $\mu_{\mathcal R_g}$ is $\frac{\sqrt{|D(f_{[p](\mathbf b)})|}}{(2\pi)^g}$.
\end{proof}
\subsection{$v_l(t)$} \label{sec: vl(t)}
Before we look at $v_l(f)$, we first tackle the easier case of $v_l(t)$. Recall that for this paper we assumed $g\ge 2$, which is a necessary condition for the results here. Since we already defined $v_p(t)=1$, for the rest of this subsection we can restrict our attention to $l\neq p$, for which
$v_l(t)=\lim_{k\rightarrow \infty}v_{l,k}(t)$ where we let $$v_{l,k}(t)=\frac{\# \{\gamma\in \GSp_{2g}(\Z/l^k\Z)\mid \tr(\gamma)=t, \text{mult}(\gamma)=p\}}{\#\GSp_{2g}(\Z/l^k\Z)/(l^{k}\phi(l^k))}.$$

Let $\GSp_{2g,p,t}$ be the subscheme of $\GSp_{2g}$ cut out by trace $t$ and multiplier $p$, so the numerator above is simply $\#\GSp_{2g,p,t}(\Z/l^k\Z)$. The main goal here is to bound $v_l(t)=1+O(l^{-2})$, and give a quantitative statement about the convergence of $v_{l,k}(t)$ to $v_l(t)$. Later we will see that analogous statements for $v_l(f)$ are much harder to prove. This is because in the $v_l(t)$ case the singularities of $\GSp_{2g,p,t}$ are more tractable compared to the $v_l(f)$ case. 

Note that \cite[Section 4]{BLV} proved the $1+O(l^{-2})$ bound for their projection-style ratios but these are defined differently as we explained in Remark \ref{rmk: AAGG problem}. For their definition of $v_l(t)$, they had access to the machinery of the Serre-Oestrelé measure developed in \cite{Serre_1981}, \cite{Oesterlé1982}, which in particular gives an upper bound on the contribution of singular points to projection-style ratios, but unfortunately does not apply in our case. Nevertheless, the relevant scheme $\GSp_{2g,p,t}$ is the same, and we will take their description of the singularities as a starting point. We recall this in the next paragraph.

Recall that $\Sp_{2g}$ is a smooth scheme of dimension $2g^2+g$ and the tangent space $L$ at the identity are the block matrices $\begin{pmatrix}
A & B\\
C & D
\end{pmatrix}$ satisfying $B^t=B,C^t=C,D^t=-A$. Let $\GSp_{2g,p}$ be the subscheme of $\GSp_{2g}$ with multiplier $p$, this is smooth because multiplication by $\gamma =\begin{pmatrix} a & b \\
c & d \end{pmatrix}\in \GSp_{2g,p}$ gives an isomorphism of schemes from $\Sp_{2g}$ to $\GSp_{2g,p}$. Furthermore, multiplication by $\gamma$ also gives an isomorphism from the tangent space $L$ at the identity of $\Sp_{2g}$ to the tangent space at $\gamma$ in $\GSp_{2g,p}$, which we hereby call $\gamma L$ and it has elements $\begin{pmatrix}
a A+b C& a B + b D\\
c A + d C & c B + d D
\end{pmatrix}$. Then, the singular points $\gamma$ in $\GSp_{2g,p,t}(\F_l)$ are those points where $\gamma L\subseteq \{\Tr=0\}$, or in other words, $\Tr(a A + b C +c B + d D)=0$ for all $A,B,C,D$ given the constraints above. The fact that $B$ and $C$ are symmetric implies that $b=-b^t$ and $c=-c^t$ are both antisymmetric, and $D^t=-A$ implies that $a=d^t$. Thus, the singularities are contained within a $2g^2-g$ linear space and hence has codimension at least $2g-1$.

For our purposes we need more than just the singularities over $\F_l$. We extend the definition over $\Z/l^k\Z$ as follows. 
\begin{definition}\label{definition: singularity of order m}
A matrix $\gamma=\begin{pmatrix} a & b \\
c & d \end{pmatrix}\in \GSp_{2g,p,t}(\Z/l^k\Z)$ is a singularity of order $m$ if $m\le k$ is the largest integer such that $l^m\mid \Tr(\gamma v)$ for all $v\in L(\Z/l^k\Z)$. Just like before, such a matrix satisfies $b\equiv -b^t, c\equiv -c^t, a\equiv d^t \pmod{l^m}$.
\end{definition}

To bound $\#\GSp_{2g,p,t}(\Z/l^k\Z)$, we individually bound the number of matrices with singularity of each order $m$ as follows. Let $\gamma=\begin{pmatrix} a & b \\
c & d \end{pmatrix}\in \GSp_{2g,p,t}(\Z/l^k\Z)$ be a singularity of order $m$. The main term comes from the non-singular matrices when $m=0$. It is not hard to show that $\#\GSp_{2g,p,t}(\F_l)=\#\GSp_{2g}(\F_l)\frac{1+O(l^{-2})}{l(l-1)}$, this was shown in \cite[Lemma 4.5]{BLV} with the help of zeta functions, but we remark that this can also be shown by taking the average of terms in Lemma \ref{lem: good prime ratio of zeta functions}. Since the singular locus is of codimension $2g-1\ge 3$, we conclude that the non-singular locus $U$ also has $\#U(\F_l)=\#\GSp_{2g}(\F_l)\frac{1+O(l^{-2})}{l(l-1)}$ non-singular points. Since $U$ is smooth of dimension $2g^2+g-1$, these lift to $\#U(\F_l)l^{(2g^2+g-1)(k-1)}=\Theta(l^{(2g^2+g-1)k})$ points over $\Z/l^k\Z$.

For $m\ge \lceil k/2\rceil$, $\bar\gamma =\pi_m(\gamma)$ must lie in the linear subspace of dimension $2g^2-g$ described above, so there are at most $l^{(2g^2-g)m}$ possibilities for $\bar\gamma$. We actually need a small strengthening of this to $O(l^{(2g^2-g)m-1})$, which we prove as follows. Note that $\GSp_{2g,p}(\F_l)$ does not contain the linear space $L(\F_l)$ (for example, consider the origin), and thus the intersection is at least one dimension smaller and the Lang-Weil bound tells us there are are $O(l^{2g^2-g-1})$ points in $\GSp_{2g,p}(\F_l)\cap L(\F_l)$. Each of these points can lift to at most $l^{(2g^2-g)(m-1)}$ points in order to remain in $L(\Z/l^m\Z)$ which gives the result.

Since $\GSp_{2g,p}$ is smooth of dimension $2g^2+g$, $\gamma$, $\pi_m(\gamma)$ lifts to exactly $l^{(2g^2+g)(k-m)}$ points, so multiplying there are at most $O(l^{(2g^2+g-1)k+k-2gm-1})$ possibilities for $\gamma$. Note that here we are lifting along $\GSp_{2g,p}$ and not along $\GSp_{2g,p,t}$, and also note that the correction term in the exponent is $k-2gm-1\le k(1-g)-1\le -k-1$.

The case for $1\le m<\lceil k/2\rceil$ is slightly more complicated. Like the previous paragraph, there are at most $O(l^{(2g^2-g)m-1})$ possibilities for $\bar\gamma=\pi_m(\gamma)$. In this context, we prove the following lemma on the distribution of traces for the lifts of $\bar\gamma$ along $\GSp_{2g,p}$.
\begin{lemma}\label{lem: trace distributions for lift}
For $a\ge 2m+1$, $\bar\gamma$ has $l^{(2g^2+g)(a-m)}$ lifts to $\GSp_{2g,p}(\Z/l^a\Z)$. Let $N(t')$ be the number of such lifts with trace $t'\pmod{l^a}$. Then, $N(t')\neq 0$ only if $t'\equiv t\pmod{l^{2m}}$, and $N(t_1')=N(t_2')$ if $t_1\equiv t_2 \pmod{l^{2m+1}}$.
\end{lemma}
\begin{proof}
For the first statement, we show that any lift of $\bar\gamma$ modulo $l^{2m}$ must have trace $t$. Indeed, the lifts to $\GSp_{2g,p}(\Z/l^{2m}\Z)$ are given exactly by $\gamma(1+l^mv)$ for $v\in L(\Z/l^m\Z)$ because the second order terms vanish as we are taking modulo $l^{2m}$. Then, it is clear that $\Tr(\gamma(1+l^mv))=t+l^m\Tr(\gamma v)\equiv t\pmod{l^{2m}}$ because $\gamma$ is a singularity of order $m$.

We prove the second part by induction. The base case of $a=2m+1$ is vacuously true, and for the induction step we prove the statement for $a>2m+1$ assuming the statement for smaller values of $a$. By the induction hypothesis, it suffices to show that $N(t_1')=N(t_2')$ when $t_1'\equiv t_2'\pmod{l^{a-1}}$. For any $\gamma' \in \GSp_{2g,p}(\Z/l^a\Z)$ with trace $t_1'$, consider $\gamma'(1+l^{\lceil \frac a 2\rceil}v)$ for $v\in L(\Z/l^{\lfloor \frac a 2 \rfloor}\Z)$ which are exactly the lifts of $\pi_{\lceil \frac a 2\rceil}(\gamma')$ to $\GSp_{2g,p}(\Z/l^a\Z)$. We have $\Tr(\gamma'(1+l^{\lceil \frac a 2\rceil}v))=t_1'+l^{\lceil \frac a 2\rceil}\Tr(\gamma'v)$. As $\gamma'$ is a singularity of order $m$, we have $l^m\mid Tr(\gamma'v)$, and in fact by a simple linear algebra argument it takes every multiple of $l^m$ with equal chance as we vary $v$. Since $\lfloor \frac a 2\rfloor+m\le a-1$, there are the same number of lifts of $\pi_{\lceil \frac a 2\rceil}(\gamma')$ with trace $t_1'$ and $t_2'$, which proves that $N(t_1')=N(t_2')$.
\end{proof}

Applying the lemma when $a=k$, we have $N(t)=N(t+l^{2m+1}t')$ which implies that at most a $1/l^{k-(2m+1)}$ proportion of the $l^{(2g^2+g)(k-m)}$ lifts of each $\bar\gamma$ have trace $t$. Hence, there are at most $O(l^{(2g^2+g-1)k+(2-2g)m})$ singular points of order $m$, and the correction factor is $(2-2g)m\le -2m\le -2$. Note that the power savings when we looked at $\GSp_{2g,p}(\F_l)\cap L(\F_l)$ was crucial here.

Combining the above, we immediately have the following, noting that the $l=p$ case is obvious.
\begin{proposition}\label{prop: vl(t) = 1 + O(l^-2)}
For $g\ge 2$, we have $v_l(t)=1+O(l^{-2})$.
\end{proposition}
In fact, we can say more about the rate of convergence of $v_{l,k}(t)$ to $v_l(t)$.
\begin{proposition}\label{prop: vl,k(t) to vl(t)}
For $g\ge 2$ and $l\neq p$, we have $v_{l,k}(t)=v_l(t)+O(l^{-(2g\lceil k/2\rceil-k+1)})$.
\end{proposition}
\begin{proof}
From the second part of Lemma \ref{lem: trace distributions for lift}, for each $m\ge 1$ there exists $c_m$ such that for $k\ge 2m+1$, $\#\GSp_{2g,p,t}(\Z/l^k\Z)=c_ml^{(2g^2+g-1)k+(2-2g)m}$. Furthermore, from the discussion above, we see that $c_m=O(1)$ is uniformly bounded from above. Hence, we have 
\begin{equation*}
\begin{split}
\frac{\#\GSp_{2g,p,t}(\Z/l^k\Z)}{l^{(2g^2+g-1)k}} &= \#U(\F_l)l^{-(2g^2+g-1)}+\sum_{m=1}^{\lfloor k/2\rfloor}c_ml^{-2m(g-1)} + \sum_{m=\lceil k/2\rceil}^{\infty}O\left(l^{-(2mg-k+1)}\right)\\
&=\#U(\F_l)l^{-(2g^2+g-1)}+\sum_{m=1}^{\infty}c_ml^{-2m(g-1)} + O\left(l^{-(2g\lceil k/2\rceil-k+1)}\right),
\end{split}
\end{equation*}
where the term on the second line is the limiting value as $k\rightarrow \infty$. Lastly, $\#\GSp_{2g}(\Z/l^k\Z)/(l^k\phi(l^k))=Cl^{(2g^2+g-1)k}$ for some constant $C$ because $\GSp_{2g}$ is smooth of dimension $2g^2+g+1$, so we get our desired result.
\end{proof}
\subsection{$v_l(f)$}\label{sec: vl(fA)} Let $f$ be an ordinary Weil $p$-polynomial, so we have $f=f_A$ for an ordinary PPAV $A/\F_p$. Recall that we defined $v_l(f)=\lim_{k\rightarrow \infty} v_{l,k}(f)$, and from now on we let $$v_{l,k}(f)=\frac{\#\{\gamma \in M(\Z/l^k\Z) \mid f_\gamma\equiv f \}}{\#\GSp_{2g}(\Z/l^k\Z)/(l^{gk}\phi(l^k))}.$$
Also recall that when $l\neq p$ we have $M=\GSp_{2g}$. Let $\GSp_{2g,f}$ be the subscheme of $\GSp_{2g}$ cut out by matrices with characteristic polynomials $f$, so when $l\neq p$ the numerator above is simply $\#\GSp_{2g,f}(\Z/l^k\Z)$.

In the generic case, we will see that $v_l(f)$ is simply given as a ratio of two zeta functions, and we discuss this in Section \ref{sec: ratio of zeta functions}. However, we still need to handle the case where this does not happen. It happens to be rather difficult to prove bounds on $v_{l,k}(f)$ and $v_l(f)$ as this requires understanding the $\Z/l^k\Z$ points of $\GSp_{2g,f}$. The singularities of this scheme are much more complicated compared to $\GSp_{2g,p,t}$ in the $v_l(t)$ case as discussed previously. In Section \ref{sec: bounds, stab, disc} we discuss the results we are able to prove, as well as further conjectures on these ratios, and the proofs of these results will be the subject of Section \ref{sec: point counting}.
\subsubsection{Ratio of zeta functions}\label{sec: ratio of zeta functions}
First, we discuss the case for a good prime $l\nmid 2p\disc(f)$. In this case, we see that conjugacy of elements in the group in $\GSp_{2g}(\Z_l)$ is the same as both $\GSp_{2g}(\Q_l)$ conjugacy and $\GSp_{2g}(\bar\Q_l)$ conjugacy.
\begin{lemma}\label{lem: conjugacy for good l}
Suppose $l\nmid \disc(f)$ and $\gamma_1,\gamma_2\in \GSp_{2g}(\Z_l)$ both have characteristic polynomial $f$, i.e. they are $\GSp_{2g}(\bar \Q_l)$ conjugate. Then, $\gamma_1$ and $\gamma_2$ are conjugate in $\GSp_{2g}(\Z_l)$.
\end{lemma}
\begin{proof}
Follows from the proof of \cite[Lemma 3.1]{Achter_2023} and because the Weyl discriminant $D(f)$ divides $\disc(f)$.
\end{proof}
We can express the local factors explicitly as a ratio of zeta functions, which also shows that the infinite product in Theorem \ref{thm: PPAV counts for char poly} converges. This is also crucial for our method as this allows us to use an effective Chebotarev density theorem to truncate this product in Section \ref{sec: chebotarev}. Recall that $K\cong \Q[x]/f_A(x)$ and $K^+\cong \Q[x]/f_A^+(x)$, and that the local zeta function is defined to be $\zeta_{L,l}(s)=\prod_{\lambda \mid l}(1-N_{L/\Q}(\lambda)^{-s})^{-1}$ for primes $\lambda$ lying over $l$.
\begin{lemma}\label{lem: good prime ratio of zeta functions}
For $l\nmid 2p\disc(f_A)$ and $k\ge 1$, we have
$$v_l(f_A)=v_l(\gamma_A)=\frac{\zeta_{K,l}(1)}{\zeta_{K^+,l}(1)}.$$
\end{lemma}
\begin{proof}
The first equality follows from Lemma \ref{lem: rational orbits to stable orbits} and \ref{lem: conjugacy for good l} as $\GSp_{2g}(\Q_l)$ and $\GSp_{2g}(\bar \Q_l)$ conjugacy classes are the same. The second equality follows from \cite[Lemma 6.2]{Achter_2023}.
\end{proof}
For the case $l=p$, \cite[Corollary 6.6]{Achter_2023} tells us that $v_p(f)$ is also a ratio of local zeta functions given a condition on the $p$-adic valuation of the discriminant, and we state this as follows.
\begin{lemma}\label{lem: vp ratio zeta fns}
Recall $A$ is an ordinary abelian variety, and suppose $\ord_p(\disc(f_A))=g(g-1)$, then
$$v_p(f_A)=\frac{\zeta_{K,p}(1)}{\zeta_{K^+,p}(1)}.$$
\end{lemma}
In order for these to be useful, we need to show that the conditions are sufficiently generic in $f$.
\begin{lemma}\label{lem: proportion ordp > g(g-1)}
\begin{enumerate}[(\alph*)]
\item For any prime $l\neq p$, the proportion of ordinary Weil $p$-polynomials $f$ with $l\mid \disc (f)$ is $O(l^{-1})$.
\item For any ordinary Weil $p$-polynomial $f$, we have $\ord_p(\disc(f))\ge g(g-1)$. Furthermore, the proportion of ordinary Weil $p$-polynomials $f$ with $\ord(\disc(f))> g(g-1)$ is $O_g(p^{-1})$. 
\end{enumerate}
\end{lemma}
\begin{proof}
First, note that the discriminant of the polynomial is a polynomial in the coefficients of $f$ with $\bb{Z}$-coefficients, say $P(a_{2g-1},\ldots,a_g)$. Let $N$ be the gcd of the coefficients of $P$. Then by a simple computation, we see that 
$$\disc(x^{2g} + p^g) = (-1)^{g+1} (2g)^{2g} p^{g(2g-1)}.$$
Hence for $l \geq g$, we have $l \nmid N$ and so when reduced mod $l$, our polynomial is non-zero. We can then apply the Lang--Weil bound to deduce that the proportion of tuples with $P(a_1,\ldots,a_g) = 0$ mod $l$ is $O_g(l^{-1})$. 
\par 
For $l = p$, by Newton polygons we see that $g$ of the roots of $x^{2g} + a_{2g-1}x^{2g-1} + \cdots + p^g = 0$ have valuation $1$ and the other $g$ have valuation $0$ over $\bb{Q}_p$. It follows from the expression of the discriminant that the valuation of $\prod_{i < j} (\alpha_i - \alpha_j)^2$ is at least $g(g-1)$ and so $\ord_p(\disc(f)) \geq g(g-1)$. On the other hand, we have 
$$\disc(x^{2g} + x^g + p^g) = g^{2g} p^{g(g-1)} (1-4p^g)^g$$
and it follows that $p^{g(g-1)} \mid \mid N$. Factoring it out, we obtain a non-zero polynomial mod $p$ and then the Lang--Weil bound applies to give us that the proportion of tuples with $\ord_p(\disc(f)) > g(g-1)$ is $O_g(p^{-1})$
\end{proof}
\subsubsection{Bounds, stabilization and discriminants}\label{sec: bounds, stab, disc}
We restrict to the case $l\neq p$, and now we tackle the case where $v_l(f)$ may not necessarily be a ratio of zeta functions. We will defer the proofs of the key propositions to Section \ref{sec: point counting}. 

First, we discuss bounds on $v_l(f)$. A ratio of local zeta functions is bounded by $1+O(l^{-1})$, and we may expect that this is the case in general as well. This happens to be true in the elliptic curve case when $g=1$, here $v_l(f)$ can be written very explicitly as in \cite[Corollary 4.6]{Gekeler2003FrobeniusDO}. Thus, we make the following conjecture.
\begin{conjecture}\label{conj: v_l = 1+O(l^-1)}
We have that $$v_{l,k}(f)=\frac{\# \GSp_{2g,f}(\Z/l^k\Z)}{\#\GSp_{2g}(\Z/l^k\Z)/(l^{gk}\phi(l^k))}=1+O(l^{-1}),$$
and in particular this also holds for the limit $v_l(f)=1+O(l^{-1})$.
\end{conjecture}
When $k=1$, we are able to deduce the $\F_l$-points of $\GSp_{2g,f}$ from the Lang-Weil bounds. The difficulty comes when $k>1$, where we need to count the number of lifts of each $\F_l$-point. However, this can behave unexpectedly depending on the singularities of $\GSp_{2g,f}$. For example, it is possible that a singular point has many more lifts than expected. On the other hand, smooth $\F_l$-points lift nicely, and by looking at the non-singular locus of $\GSp_{2g,f}$ we are able to prove the following.
\begin{proposition}\label{prop: vl lower bound}
For $l\neq p$ and $k\ge 1$, we have $v_{l,k}(f) \ge 1-O(l^{-1})$, which further implies that $v_l(f)\ge 1-O(l^{-1})$.
\end{proposition}
Because of the nature of $v_l(f)$ as a limit, we can only understand it through the successive approximations $v_{l,k}(f)$, and for any useful application we would need to bound the error between them. From the explicit expression of $v_l(f)$ in the elliptic curve case we conjecture the following.
\begin{conjecture}
There is a constant $\delta>0$ where $v_l(f)=v_{l,k}(f)+O(l^{-k\delta})$.
\end{conjecture}
Unfortunately, we are unable to prove anything of this sort, again due to the difficulties of counting $\GSp_{2g,f}(\Z/l^k\Z)$. However, we are able to show that $v_{l,k}(f)$ stabilizes for sufficiently large $k$ depending on $f$.
\begin{proposition}\label{prop: stabilization} There is a constant $C>0$ where for $l\neq p$ and $k \geq C \ord_l(\disc(f))$ we have $v_l(f)=v_{l,k}(f)$.
\end{proposition}
However, for every choice of $l$ and $k$, there would be some polynomials $f$ for which $v_l(f)\neq v_{l,k}(f)$. To deal with this, we need show that the number of matrices belonging to such $v_l(f)$ is small, so that over $\Z/l^k\Z$, we retain most of the weight from the local factors $v_l(f)$ when we look at the truncation $v_{l,k}(f)$. We state this as the following proposition.
\begin{proposition}\label{prop: l^k not dividing disc lower bound}
There is a constant $\delta>0$ such that for $l\neq p$ and $k\ge 1$ we have
$$\frac{\#\{\gamma \in \GSp_{2g,p}(\Z/l^k\Z)  \mid l^k\nmid\disc(\gamma)\}}{\#\GSp_{2g,p}(\Z/l^k\Z) }\ge  1-O(l^{-k\delta}).$$
\end{proposition}
We prove Propositions \ref{prop: vl lower bound}, \ref{prop: stabilization} and \ref{prop: l^k not dividing disc lower bound} in the next section. As we can see, the strength of these statements are certainly not ideal when compared to the corresponding conjectures. Nevertheless, they are sufficient to complete the proof of Theorem \ref{thm: A_g}, although this leads to more complications in Section \ref{sec: final proof}.
\section{Point Counting modulo $l^k$}\label{sec: point counting}
The aim of this section is to establish some basic bounds on the number of points a variety $X/\bb{Z}_l$ has over the ring $\bb{Z}/l^k\Z$, which we will then use to prove Proposition \ref{prop: vl lower bound}, \ref{prop: stabilization} and \ref{prop: l^k not dividing disc lower bound}. When $k = 1$, we are interested in counting $\bb{F}_l$-points in which case we can apply Lang--Weil bounds which gives us satisfactory results. However when $k \geq 2$, the situation is much more complicated. When $X$ is smooth over $\bb{Z}_l$, it is not hard to show that every $\bb{F}_l$-point has exactly $l^{(k-1) \dim X}$ lifts to $\bb{Z}/l^k\Z$. But this fails when $X$ is singular. For example, if we let $X = \bb{Z}[x]/x^2$, then there are $l^k$ number of solutions mod $l^{2k}$, even though $\dim X = 0$. 
\par 
Let's first quickly handle the case of $k = 1$. The Lang--Weil bound implies that for a geometrically irreducible subvariety $V \subseteq \bb{P}^n$ that is defined over $\bb{F}_l$, we have 
$$|V(\bb{F}_l) - l^{\dim V}| \leq O_{\deg V, \dim V, n} (l^{\dim V - \frac{1}{2}})$$
where our big $O$ constant depends on $\deg V$ and $n$. If $V$ is not geometrically irreducible, one can still deduce an upper bound 
$$|V(\bb{F}_l)| \leq O_{\deg V, \dim V, n}(l^{\dim V}).$$

In our situations, we are normally interested in the case of an affine variety $V \subseteq \bb{A}^n$ that is cut out by some equations $f_1,\ldots,f_n$ for which we can bound the degree. We embed $\bb{A}^n$ into $\bb{P}^n$ and homogenize the $f_i$'s. Using Bezout, the degree of the intersection of the homogenized $f_i$'s is bounded by the product of $\deg(f_i)$. Let $W$ be this intersection. Then $V \subseteq W$ and $W \cap \bb{A}^n = V$. This implies that if we write $W = \cup_{i=1}^{n} Z_i$ and if $Z_1,\ldots,Z_m$ are the components that are not contained in the hyperplane $\{x_n = 0\}$, we have that $\cup_{i=1}^{m} Z_i$ is the Zariski closure of $V$ inside $\bb{P}^n$. Thus the degree of $\bar{V}$ may be bounded by the product of $\deg(f_i)$ and we can apply Lang--Weil to $\bar{V}$, noting that the additional points from $\bb{P}^n$ are contained in $\bar{V} \cap \{x_n = 0\}$ which is a variety one dimension lower and of degree equal to that of $\bar{V}$. Hence we can still apply the Lang--Weil bound even if $V$ is an affine subvariety of $\bb{A}^n$. 
\par 
We now move onto the case where $k \geq 2$. If $X$ is a hypersurface in $\bb{A}^n$ given by $\{f = 0 \}$, then the Igusa local zeta function for $f$ captures the number of solutions $f(x_1,\ldots,x_n) = 0$ has mod $l^k$ and it should be possible to apply the theory of Igusa zeta functions to obtain some corresponding bounds. For example the Igusa zeta function is known to be rational and some properties of its poles are known. However we are interested in certain subvarieties of $\GSp_{2g}$ and so we are forced to rely on more ad-hoc statements. 
\par 
We first recall the non-archimedean implicit function theorem which relies on the multivariate Hensel's lemma. 
\begin{theorem}[{\cite[Theorem 4]{Fis97}}] \label{Implicit1}
Let $A = \bb{Z}_l[x_1,\ldots,x_n]$ and let $f_1,\ldots,f_m \in \bb{Z}_l\{x_1,\ldots,x_n\}$ be convergent power series with $m \leq n$. Let $\mathbf{x} =(x_1,\ldots,x_n)$ be a tuple and consider the $m \times n$ matrix $M$ with entries 
$$m_{i,j} = \frac{\partial f_i}{\partial x_j}(\mathbf{x}).$$
Let $f$ denote the exponent of $l$ dividing the determinant of the first maximal $m \times m$ minor and let $e \geq f$ be a positive integer. Then if $f_i(\mathbf{x}) \equiv 0 \pmod {l^{e+f+1}}$ for all $i$, there exist power series $\phi_1,\ldots,\phi_m \in \bb{Z}_l[x_{m+1},\ldots,x_n]$ such that 
$$f_i(y_1 + l^e \phi_1(t_{m+1},\ldots,t_n),\ldots,y_m + l^e \phi_m(t_{m+1},\ldots,t_n), y_{m+1} + l^{2e} t_{m+1},\ldots, y_n + l^{2e} t_n) = 0 $$
for all $i$ and for all $t_{m+1},\ldots,t_n \in l \bb{Z}_l$. 
\end{theorem}
In particular, for a given $\bb{F}_l$-point with $f = 0$, we can choose $e=0$ so we obtain at least $l^{(k-1)(n-m)}$ lifts of this point to $\Z/l^k\Z$. This will be the key theorem that we will use to obtain lower bounds on $X(\bb{Z}/l^k\Z)$. 

\subsection{Lower bound for ratios}
Our aim in this subsection is to prove Proposition \ref{prop: vl lower bound} by applying Theorem \ref{Implicit1} to get a lower bound for $\GSp_{2g,f}(\bb{Z}/l^k\Z)$. First, fix an ordinary Weil $p$-polynomial $f$. We first recall results of \cite{Ste65}. Fix a field $K$ and consider the map $\pi: \Sp_{2g} \to \bb{A}^g$ that sends a matrix $M$ to $\mathbf a$ where $f_M=f_{\mathbf a}$, and we rewrite the map as $\pi=(p_1,\ldots,p_g)$. In \cite{Ste65}, the fundamental characters $\chi_1,\ldots,\chi_g$ were used instead but there exists an upper triangular invertible matrix $A$ with $\bb{Z}$-coefficients such that $(p_1,\ldots,p_g)^T = A(\chi_1,\ldots,\chi_g)^T$, so the results there apply to our map too.

Since $\Sp_{2g}$ is simply connected, \cite[Theorem 6.11]{Ste65} implies that the fiber $\pi^{-1}(x)$, viewed as a variety i.e. taking the reduced scheme structure,  for $x \in \bb{A}^g(K)$ is geometrically irreducible. On the other hand it is clear that if $x \in \bb{A}^g(K)$, then $\pi^{-1}(x)$ is defined over $K$ just by writing down the equations that determine the coefficients $\mathbf a$. Thus the reduced fiber $\pi^{-1}(x)$ is a geometrically irreducible variety that is defined over $K$. 

Now let $K'$ be an extension so that $\sqrt{p} \in K'$. Then there is an isomorphism between $\Sp_{2g}$ and $\GSp_{2g,p}$ that is equivariant with respect to the Steinberg morphism given by multiplication by $\sqrt{p} I_{2g}$. Hence after base changing to $K'$, any fiber $\pi^{-1}(x)$ with its reduced subscheme structure is geometrically irreducible. Thus the same holds over $K$ too.  
\par 
Next, we look at the differentials $(dp_1,\ldots,dp_g)$. For a fixed $M \in \GSp_{2g,p}$, this is a map $d \pi: T_M \GSp_{2g,p}\to K^g$ and hence each $dp_i$ can be viewed as an element of $K^{\dim \Sp_{2g}}$. Now, there exist $m=4g^2 - \dim \Sp_{2g}$ equations $f_1,\ldots,f_m$ that cut out $\GSp_{2g,p}$ inside $M_{2g \times 2g} \simeq \bb{A}^{4g^2}$ by looking at the equation $MJM^T = pJ$. We let $f_{m+1},\ldots,f_{m+g}$ be the additional $g$ equations that cut out $f_M=f$.

\begin{proposition} \label{TangentSpace1}
Let $f_M=f$. Then, the condition that $(dp_1,\ldots,dp_g)$ are linearly independent at $M \in \GSp_{2g,p}(K)$ is equivalent to the Jacobian matrix for $(f_1,\ldots,f_{m+g})$ having maximal rank at $M$.
\end{proposition}

\begin{proof}
Set $n = 4g^2$. Letting $V_i$ be the kernel of $(\frac{\partial f_i}{\partial x_1},\ldots, \frac{\partial f_i}{\partial x_n})$, the condition of the Jacobian matrix having maximal rank is equivalent to saying $\bigcap_{i=1}^{m+g} V_i$ has codimension $m+g$. By definition, $T_M \GSp_{2g,p}$ is given by $\bigcap_{i=1}^{m} V_i$ and is of codimension $m$ as $\Sp_{2g}$ is smooth. Since $(dp_1,\ldots,dp_g)$ are linearly independent, this implies that $T_M \Sp_{2g} \cap V_{m+1} \cap \cdots \cap V_{m+g}$ is of codimension $g$ in $T_M \Sp_{2g}$ which gives us codimension $m+g$ inside $k^n$ as desired.
\end{proof}

It immediately follows that the locus $U$ of $M \in \GSp_{2g,p}(\bar{K})$ for which $(dp_1,\ldots,dp_g)$ are linearly independent is a Zariski open set whose complement $Z$ has degree bounded in terms of $g$. Here, we use the notion of degree of its Zariski closure under the natural embedding $\bb{A}^{4g^2} \xhookrightarrow{} \bb{P}^{4g^2}$. On the other hand, a fiber $\pi^{-1}(x)$ is a geometrically irreducible subvariety of codimension $g$ in $\Sp_{2g}$ and degree bounded in terms of $g$. Thus $Z \cap \pi^{-1}(x)$ is a closed subvariety of $\pi^{-1}(x)$ with degree bounded in terms of $g$. Hence we obtain the following.

\begin{proposition} \label{LowerBoundPointCount1}
For each $x \in \bb{A}^g(\F_l)$, we have 
$$\#  (U \cap \pi^{-1}(x))(\F_l)   \geq l^{\dim \Sp_{2g} - g}\left(1 - O(l^{-1}) \right).$$
\end{proposition}

\begin{proof}
We first look at $\pi: \Sp_{2g} \to \F_l^g$. First, by \cite[Theorem 1.5]{Ste65}, the differentials $(dp_1,\ldots,dp_g)$ are linearly independent if $M \in \Sp_{2g}(\bar\F_l)$ is a regular element, i.e. when the centralizer has dimension $g$. On the other hand, there are regular elements in $\pi^{-1}(x)$ by \cite[Theorem 6.11]{Ste65}. Hence $U \cap \pi^{-1}(x)$ is non-empty where $U$ is the Zariski open in $\Sp_{2g}$ for which $(dp_1,\ldots,dp_g)$ is linearly independent. Now the isomorphism between $\Sp_{2g}$ and $\GSp_{2g,p}$ implies that $U \cap \pi^{-1}(x)$ is non-empty too for $\GSp_{2g,f}$ and our bound follows by Lang--Weil.     
\end{proof}
Now we prove the desired lower bound on $v_l(f)$.
\begin{proof}[Proof of Proposition \ref{prop: vl lower bound}]
It suffices to show that 
$$\# \GSp_{2g,f}(\bb{Z}/l^k\Z) \geq l^{k(\dim \Sp_{2g}-g)} \left(1 - O(l^{-1})\right)$$
because then we can divide by $\#\GSp_{2g}(\Z/l^k\Z)\le l^{k\dim \GSp_{2g}}(1+O(l^{-1}))$ which follows because $\GSp_{2g}$ is smooth and geometrically irreducible. 

For every $\F_l$-point of $U \cap \pi^{-1}(x)$, the Jacobian matrix given by $(f_1,\ldots,f_{m+g})$ has maximal rank and hence there is a maximal minor which is invertible. We may then apply Theorem \ref{Implicit1} to conclude that our $\F_l$-point has $l^{(k-1)(\dim \Sp_{2g}-g)}$ many lifts. By Proposition \ref{LowerBoundPointCount1}, we have at least $l^{\dim \Sp_{2g}-g} \left(1 - O_g(l^{-1}) \right)$ many such $\F_l$-points of $U \cap \pi^{-1}(x)$ which gives the desired result.
\end{proof}
\subsection{Stabilization of ratios} 
The next property we wish to analyze is the eventual stabilization of $v_{l,k}(f)$ as stated in Proposition \ref{prop: stabilization}. To show this we would need to compare $\# \GSp_{2g,f}(\bb{Z}/l^{k+1}\Z)$ with $\# \GSp_{2g,f}(\bb{Z}/l^k\Z)$. For $k$ sufficiently large, it seems reasonable to expect that the ratio of these two terms should converge to $l^{\dim \GSp_{2g,f}} = l^{\dim \Sp_{2g} - g}$. Our aim now is to prove that this is the case and also give an effective bound on how large $k$ has to be in terms of $\disc(f)$. 
\par 
Let's first work out the simpler case of $\GL_g$. Fix our characteristic polynomial $f$ with $\bb{Z}$-coefficients along with a matrix $M \in \GL_g(\bb{Z}_l)$ with characteristic polynomial $f$. Consider again the Steinberg map $\pi: \GL_g \to \bb{A}^g$ where we send $M$ to the coefficients of its characteristic polynomial. We denote the map by $(p_1,\ldots,p_g)$. Let $A$ be the $g \times g^2$ matrix representing the differential $(dp_1,\ldots,dp_g)$ at the matrix $M$. 

\begin{proposition} \label{DeterminantBound1}
Assume that $\disc(f) \not = 0$. Then there exists a constant $c > 0$, depending only on $g$, such that 
$$|\det(AA^T)| \geq   |\disc(f)|^c.$$
\end{proposition}
\begin{proof}
Since $\disc(f) \not = 0$, the matrix $M$ is diagonalizable and has $g$ distinct eigenvalues. Let $P$ be a matrix such that $PMP^{-1} = D$ is diagonal, and let $\lambda_1,\ldots,\lambda_g$ be the distinct eigenvalues in $\bar\Z_l$. Then, the matrix 
$$E_{j} = \prod_{i \not = j} \frac{M - \lambda_iI}{\lambda_j - \lambda_i} \in M_{g \times g}(\ovl{\bb{Q}}_l)$$
projects vectors to the $\lambda_j$-eigenspace. Let $v_1,\ldots,v_g$ be vectors in $\bar\Z_l$ which form an eigenbasis for $M$, chosen such that in our coordinates, we have $||v_i||_l=1$ for all $i$, where the $l$-adic valuation of a vector is defined to be the maximum of the $l$-adic valuation of its entries. The linear operator $E_j$ is a rank one matrix such that $E_j(v_i) = \delta_{i,j} v_i$. 
\par 
Let $P$ be the change of basis matrix for the $v_i$'s. Let the row vectors of $P^{-1}$ be $w_1^T,\ldots,w_g^T$. Then $v_j w_j^T$ is a rank one $g \times g$ matrix. In fact, we claim that this is exactly $E_j$. Indeed, we have $w_j^T v_i = \delta_{j,i}$ too and so it follows that $v_j w_j^T v_i = \delta_{i,j} v_i$. As $||v_j||_l = 1$, this then implies that the coefficients of $w_j^T$ are bounded above by that of $E_j$. As $||M - \lambda_i|| \leq 1$, it follows that 
$$||E_j|| \leq \prod_{i \not = j} |\lambda_i - \lambda_j|^{-1} \leq |\disc(f)|^{-1}$$
and so $||w_j^T|| \leq |\disc(f)|^{-1}$ and so is $||P^{-1}||$. This implies that the largest singular value of $P^{-1}$ is bounded above by $|\disc(f)|^{-1}$ and hence the smallest singular value of $P$ is bounded from below by $|\disc(f)|$. 
\par 
We now calculate our Jacobian matrix $A'$ for $PMP^{-1}$. The conjugation action $X \mapsto PXP^{-1}$ induces an endomorphism on $M_{g \times g}(\bb{Q}_l)$ and we let $B$ denote the resulting $g^2 \times g^2$ matrix transforming $A$ to $A'$ so that $A' = AB$ and furthermore, 
$$||B^{-1}|| \leq \max \{||P||,||P^{-1}||\} \leq |\disc(f)|^{-1} \implies |\det(A A^T)| \geq |\det((AB)(B^TA^T))| |\disc(f)|^{g}$$
as the singular values of $B^{-1}$ are at most $|\disc(f)|^{-1}$ in size. We hence reduce to the case where $M$ is a diagonal matrix.
\par 
Finally when $M$ is diagonal matrix, it is clear that $A$ only involves the diagonal entries. Thus it reduces to a $g \times g$ matrix where
$$a_{i,j} = (j-1)^{th} \text{ symmetric sum of } m_1,\ldots,m_{i-1},m_{i+1},\ldots,m_g$$
where $m_i$'s are the $i^{th}$ diagonal entry of $M$. Now $\det(A)$ expands to an alternating polynomial in the $m_i$'s with the same degree as $\disc(f)$ and thus must be a scalar multiple of $\disc(f)$. We thus get $|\det(AA^T)| = |\disc(f)|^2$, which implies that for our original $A$, we have the inequality
$$|\det(AA^T)| \geq |\disc(f)|^{g+2}$$
as desired.
\end{proof}
We now prove an analogue for $\GSp_{2g,p}(\bb{Z}_l)$ matrices, where now we fix $f$ to be an ordinary Weil $p$-polynomial. Let $M \in \GSp_{2g,p}(\bb{Z}_l)$. Recall from the discussion in Section \ref{sec: vl(t)} the explicit basis of the tangent space of $\GSp_{2g,p}$ at $M$ given by $M X_{i,j}^{(A)}, M X_{i,j}^{(B)}, M X_{i,j}^{(C)}$, where 
$$X^{(A)}_{i,j} = \left( \begin{matrix} E_{i,j} & 0 \\ 0 & -E_{i,j} \end{matrix} \right), X^{(B)}_{i,j} = \left( \begin{matrix} 0 & E_{i,j} + E_{j,i} \\ 0 & 0 \end{matrix} \right), X^{(C)}_{i,j} = \left( \begin{matrix} 0 & 0 \\ E_{i,j} + E_{j,i} & 0  \end{matrix} \right)$$
is a basis of the tangent space of $\Sp_{2g}$. Here, we take $1 \leq i,j \leq g$, and when $i = j$, we will modify $X^{(B)}_{i,i}$ and $X^{(C)}_{i,i}$ to only have one copy of $E_{i,i}$ instead of $E_{i,i} + E_{i,i}$. 

Let $\pi: \GSp_{2g,p} \to \bb{A}^g$ be our Steinberg map, which we denote by $(p_1,\ldots,p_g)$. We then let $A$ be the $g \times (\dim \Sp_{2g})$ matrix representing the differential $(dp_1,\ldots,dp_g)$ at the matrix $M$ with respect to our basis of $T_M \Sp_{2g}$. We prove the analogue of Proposition \ref{DeterminantBound1} for $\GSp_{2g,p}$.
\begin{proposition} \label{DeterminantBound2}
Assume that $\disc(f) \not = 0$. Then there is a constant $e > 0$ depending only on $g$ such that given $M \in \GSp_{2g,p}(\bb{Z}_l)$ with $\Char(M) = f$ and $p \not = l $ a prime, we have 
$$|\det(AA^T)| \geq |\disc(f)|^{e}.$$
\end{proposition}

\begin{proof}
As in the $\GL_g$ case, since $\disc(f) \not = 0$, we may let $P$ be a matrix such that $PMP^{-1}$ is diagonal. Let $v_1,\ldots,v_g,w_1,\ldots,w_g$ be the corresponding eigenbasis where the eigenvalue for $v_i$ is $p$ times the reciprocal of that of $w_i$. We normalize $v_i,w_i$ so that $||v_i|| = ||w_i|| = 1$. Let $\alpha_i = (v_i, w_i )$ where $( \, , \, )$ is the standard symplectic form. Note that $|\alpha_i| \leq 1$. We aim to bound $|\alpha_i|$ from below.
\par 
Let $N$ be the $\bar{\bb{Z}}_l$-span of $v_1,\ldots,v_g,w_1,\ldots,w_g$. Clearly $N \subseteq (\bar{\bb{Z}}_l)^{2g}$. Let $\lambda_1,\ldots,\lambda_{2g}$ be the eigenvalues of the $v_i$'s and $w_i$'s with $\lambda_{g+i} = \lambda_i^{-1}$. Since the operator $E_j$ sends  $(\bar{\bb{Z}}_l)^{2g}$ to $\frac{1}{|\disc(f)|} v_j$ and $E_{\lambda_{g+j}}$ sends $(\bar{\bb{Z}}_l)^{2g}$ to $\frac{1}{|\disc(f)|} w_j$ for $1 \leq j \leq g$, it follows from the fact that $\sum_{j=1}^{2g} E_j = \id$ that $\frac{1}{|\disc(f)|} (\bar{\bb{Z}}_l)^{2g} \subseteq N$. Hence 
$$|\det(N)| \geq |\disc(f)|^{2g} \implies \prod_{i=1}^{g} |\alpha_i|^2 \geq |\disc(f)|^{2g} \implies |\alpha_i| \geq |\disc(f)|^g.$$
In particular, if we scale the $v_i$ by $\alpha_i^{-1}$ so that $P \in \Sp_{2g}(\bar{\bb{Q}}_l)$, it follows that $1 \leq ||v_i|| \leq |\disc(f)|^g$. 
\par 
We can now apply the argument in the $\GL_g$ case to conclude that $||P||,||P^{-1}|| \leq |\disc(f)|^{-(g+1)}$. Then if $A'$ denotes the Jacobian matrix for $(dp_1,\ldots,dp_g)$ at $PMP^{-1}$ which uses the basis $PMP^{-1} X^{(-)}_{i,j}$, the matrix $A' = AB$ where $B$ is the $\dim \Sp_{2g} \times \dim \Sp_{2g}$ matrix corresponding to conjugation by $P$ for the basis $X^{(-)}_{i,j}$ at the identity. Since $||P||, ||P^{-1}|| \leq |\disc(f)|^{-(g+1)}$, it follows that the singular values of $B$ and $B^{-1}$ are bounded above by $|\disc(f)|^{-(g+1)}$. Thus 
$$|\det(A A^T)| \geq |\det((A')(A')^T| |\disc(f)|^{g(g+1)}$$
and we are now reduced to understanding the case where $M$ is a diagonal matrix. 
\par 
In the diagonal case, it is easy to check that again, only the diagonal matrices $X_{i,i}^{(-)}$ matter. We thus again obtain $|\det(A A^T)| = |\disc(f)|^2$ and hence we obtain our inequality as desired.
\end{proof}

Note that by the Cauchy--Binet formula, $\det(AA^T)$ is the sum of squares of the determinants of the maximal minors of $A$. In particular, there exists a maximal minor of $A$ whose determinant is $\geq |\disc(f)|^{e/2}$. 
\par 
Let $c$ be the exponent of $l$ in $\disc(f)^e$, where $e$ is in Proposition \ref{DeterminantBound2}. Fix a characteristic polynomial $f$ with $\bb{Z}$-coefficients.  We now turn our attention to studying lifts of $M \in \GSp_{2g,p}(\bb{Z}/l^{k-c})$, where $\Char(M) = f$ mod $l^{k-c}$, to $\GSp_{2g,p}(\bb{Z}/l^k)$ with characteristic polynomial $f$ mod $l^k$, where $k-c \geq 1$. 

\begin{proposition} \label{CharacteristicLift1}
Assume that $k > 2c$ and that $M$ lifts to a matrix $\tilde{M} \in \Sp_{2g}(\bb{Z}/l^k)$ with $\Char(\tilde{M}) = f$ mod $l^k$. Then $M$ lifts to a matrix $\tilde{M}_0 \in \Sp_{2g}(\bb{Z}_l)$ with $\Char(\tilde{M}_0) = f$. 
\end{proposition}

\begin{proof} 
We first fix some lift $M_0$ of $\tilde{M}$ to $\GSp_{2g,p}(\bb{Z}_l)$ with characteristic polynomial $f_0$ which need not be equal $f$. This exists because $\GSp_{2g,p}$ is smooth over $\bb{Z}_l$. Then every lift of $M$ to $\GSp_{2g,p}(\bb{Z}/l^{2(k-c)})$ is given by 
$$M_0 + M_0 l^{k-c} \left(a_{i,j} X^{(A)}_{i,j} + b_{i,j} X^{(B)}_{i,j} + c_{i,j} X^{(C)}_{i,j} \right)$$
where $a_{i,j},b_{i,j},c_{i,j} \in \bb{Z}_l$. Hence to obtain a lift $\tilde{M}_0$ of $M$ to mod $l^{2(k-c)}$ with $\Char(\tilde{M}_0) = f$ is equivalent to solving $f_k(a_{i,j}, b_{i,j}, c_{i,j}) = 0$ for $1 \leq k \leq g$, where each $f_k$ is the equation corresponding to the coefficient of $x^{2g-i}$ being equal to that of $f$. 
\par 
Now observe that the Jacobian matrix for the tuple $(f_1,\ldots,f_g)$ agrees with $l^{k-c} A$, where $A$ is the Jacobian of the differential $(dp_1,\ldots,dp_g)$ at $T_{M_0} \Sp_{2g}$ in our standard basis, after taking mod $l^{2(k-c)}$ as any degree two terms involving $a_{i,j}, b_{i,j}, c_{i,j}$ will be $0$. Since $2(k-c) > k$,  we obtain that there is a maximal minor whose determinant has order $\leq k$. Since $\tilde{M}$ gives us a lift whose characteristic polynomial agrees with that of $f$ mod $l^{k}$, we can apply Proposition \ref{Implicit1} with the starting tuple $a_{i,j} = b_{i,j} = c_{i,j} = 0$ along with $e = 0$. This gives us a lift $\tilde{M}_0 \in \GSp_{2g,p}(\bb{Z}/l^{2(k-c)})$ with $a_{i,j} = b_{i,j} = c_{i,j} = 0$ mod $l$ and hence a lift of $\tilde{M}$ to $\GSp_{2g,p}(\bb{Z}/l^{2(k-c)})$ with characteristic polynomial $f$. We can now repeat the process with $2(k-c)$ replacing $k$ and using $\tilde{M}_0$ as our new $\tilde{M}$ and $\tilde{M}_0$ mod $l^{2(k-c)-c}$ as our new $M$. Then as $2(k-c) > k$, iterating this process will provide us a lift to $\bb{Z}_l$ as desired. 
\end{proof}

Now let $k$ be large enough so that $k-c \geq \frac{k+1}{2}$. Fix a congruence class $M$ mod $l^{k-c}$ with $\Char(M) = f$ mod $l^{k-c}$ with $M \in \GSp_{2g,p}(\bb{Z}/l^{k-c})$. We want to show that the number of lifts of $M$ to $\Z/l^k\Z$ with characteristic polynomial being $f$ is proportional to the number of lifts to $\Z/l^{k+1}\Z$. 

\begin{proposition} \label{CharacteristicLift2} 
If $k-c \geq \frac{k+1}{2}$, then the number of lifts of $M$ to mod $l^{k+1}$ with characteristic polynomial $f$ is exactly $l^{\dim \Sp_{2g}-g}$ times the number of lifts of $M$ to mod $l^k$ with characteristic polynomial $f$.
\end{proposition}

\begin{proof}
First, we may assume there exists one such lift to mod $l^k$. Then by Proposition \ref{CharacteristicLift1}, there exists a lift $M_0$ in $\GSp_{2g,p}(\bb{Z}_l)$ such that $\Char(M_0) = f$. We may then describe every lift of $M$ to $\Sp_{2g}(\bb{Z}/l^k)$ as 
$$M = M_0 + M_0 l^{k-c} \left(a_{i,j} X^{(A)}_{i,j} + b_{i,j} X^{(B)}_{i,j} + c_{i,j} X^{(C)}_{i,j} \right).$$
We want to solve for $\Char(M) = f$ mod $l^k$. Since $k-c \geq \frac{k}{2}$, the expression for each coefficient of $\Char(M)$ is linear in $a_{i,j}, b_{i,j}, c_{i,j}$. In fact, as $\Char(M_0) = f$ mod $l^k$, it is equivalent to solving the linear system of equations $Ax = 0$ mod $l^c$ where $x = (a_{i,j}, b_{i,j}, c_{i,j})^T$. Similarly if we want to lift to $l^{k+1}$, it is equivalent to solving $Ax = 0$ mod $l^{c+1}$. 
\par 
We now write $A$ in Smith normal form, i.e. write $A = UDV$ where $U,V$ are invertible matrices over $\bb{Z}_l$ and $D$ is a $g \times \dim \Sp_{2g}$ matrix with entries only on the diagonal. Then the linear system is equivalent to $Dx = 0$ mod $l^c$ and mod $l^{c+1}$. If the diagonal has entries $d_1,\ldots,d_g$, then note that the order of $l$ dividing $d_i$ is $\leq c$. It is then straightforward to deduce that the number of solutions mod $l^{c+1}$ is exactly $l^{\dim \Sp_{2g} - g}$ times that of the number of solutions mod $l^c$ as desired.
\end{proof}
Proposition \ref{prop: stabilization} directly follows from the discussion above.
\subsection{Bounding matrices with zero discriminant}
Our next proposition will give an upper bound on the number of $\bb{Z}/l^k\Z$ points of the solutions for a polynomial $f(x_1,\ldots,x_n)$. When $k = 1$, then Lang--Weil bound implies that we can obtain a bound of the form $O_{\deg f}(l^{n-1})$. However as we see from the example of $f(x) = x^2$, we obtain $l^{k/2}$ many solutions and so we cannot expect a bound of the form $O(l^{k(n-1)})$. We will show that we can still get a power savings that scales linearly with $k$. For a polynomial $P$, define the $l$-adic valuation $||P||_l$ as the maximum of the $l$-adic valuations of the coefficients.
\begin{proposition} \label{UpperBoundPoints1}
Let $P(x_1,\ldots,x_d)$ be a degree $e$ polynomial with integer coefficients with $||P||_l = 1$. Then there exist constants $C, \delta > 0$ depending only on $d$, such that for any prime number $l$ and a natural number $k$, there are at most $Cel^{k(d - \frac{\delta}{C e})}$ solutions to $P(x_1,\ldots,x_d) = 0$ in $(\bb{Z}/l^k\Z)^d$. 
\end{proposition}
\begin{proof}
We first handle the one variable case first. By \cite[Lemma 7]{Bombieri_Schmidt_1987}, since $||P||_l = 1$, for any integer $\mu > 0$ there exists at most $\deg P = e$ many $\bb{Z}_l$-elements $(A_i,B_i)$ for which $A_i \bb{Z}_l + B_i$ are the solutions to $\ord_l(P(x))\ge \mu$ and we may assume that $A_i = l^{a_i}$ so we have $l^{\mu} \mid f(l^{a_i} x + B_i)$. Expanding out for the pair $(A_m,B_m)$ and looking at the coefficient of $x^j$, we obtain
$$l^{\mu} \mid l^{a_m j} \sum_{i = j}^{e} c_i \binom{i}{j} (B_m)^{i-j}$$
for every $j$ where $P(x) = \sum_{i=1}^{e} c_i x^i$. Let $c_j$ be the largest $j$ for which $\ord_l(c_j)= 0$. Then we have 
$$\ord_l \left(\sum_{i=j}^{e} c_i \binom{i}{j} (B_m)^{i-j} \right) = \ord_l(c_j) = 0.$$ Thus we must have 
$$l^{\mu} \mid l^{a_m j} \implies a_m \geq \frac{\mu}{e}.$$
In particular taking $\mu = k$, we can have at most $e l^{k - \frac{k}{e}}$ many such solutions as desired. 
\par 
We now move onto the general case. We write $P = Q_e(x_1,\ldots,x_{d-1})x_d^e + \cdots + Q_0(x_1,\ldots,x_{d-1})$ with $Q_e \not = 0$. Since $||P||_l = 1$, we must have $||Q_j||_l = 1$ for some $j$. Given $k$, we first remove all the tuples $x_1,\ldots,x_d$ mod $l^k$ such that $Q_j(x_1,\ldots,x_{d-1}) \equiv 0$ mod $l^{k/2}$. By induction on the number of variables, this removes 
$$C_{d-1} e l^{k \left(d - \frac{\delta_{d-1}}{2C_{d-1}e} \right)}$$
many tuples. For each tuple $\mathbf{x} = (x_1,\ldots,x_{d-1})$ for which $Q_j(\mathbf{x})$ is not divisible by $l^{k/2}$, it follows that we obtain a polynomial $P_{\mathbf{x}}$ in $x_d$ for which $||P_{\mathbf{x}}||_l \ge l^{-k/2}$. Factoring this common power of $l$ out, the remaining $x_d$ must satisfy $P_{\mathbf{x}}(x_d) \equiv 0$ mod $l^{k/2}$ where now $||P_{\mathbf{x}}||_l = 1$. It follows from the one variable case that there are at most 
$$C_1 e l^{k \left(d - \frac{\delta_1}{2 C_1 e} \right)}$$
solutions. Adding these two up gives us $C_d$ and $\delta_d$ as desired.
\end{proof}

We would like to apply this to the discriminant polynomial on $\GSp_{2g,p}$. We will exploit the fact that $\Sp_{2g}$ is locally étale over $\bb{A}^{\dim \Sp_{2g}}$. We first show that if $f: X \to Y$ is étale, then the number of $\bb{Z}/l^k\Z$-points in a fiber of $f$ can be bounded.

\begin{lemma} \label{lem: étaleBound1}
Let $f: X \to Y$ be a finite étale morphism of noetherian schemes over $\bb{Z}$. Then there exists $d > 0$ such that for every prime power $l^k$, the size of a fiber of $f:X(\bb{Z}/l^k\Z) \to Y(\bb{Z}/l^k\Z)$ is bounded by $d$. 
\end{lemma}

\begin{proof}
We may cover $X$ with an open cover $\{U_i\}$ and $Y$ with $\{V_i\}$ such that $f(U_i) \subseteq V_i$ and $\pi_i U_i \to V_i$ is standard étale, i.e. of the form $A \to A[x]_g/(h)$ where $g,h \in A[x]$ and $h$ monic with $h'$ invertible. Then a $\bb{Z}/l^k\Z$ point $z$ of $A$ corresponds to a morphism $\vphi: A \to \bb{Z}/l^k\Z$ and a $\bb{Z}/l^k \bb{Z}$ point in $\pi_i^{-1}(z)$ corresponds to an extension $\tilde{\vphi}: A[x]_g/(h) \to \bb{Z}/l^k\Z$. 
\par 
Let $\bar g,\bar{h}$ be the image of $g,h$ in $(\bb{Z}/l^k\Z)[x]$ under $\vphi$. Then an extension corresponds exactly to a solution $\alpha$ of $\bar{h}(\alpha) = 0$ mod $l^k$ along with $\bar{g}(\alpha)$ being invertible. But $h'$ being invertible in $A[x]_g/h$ implies that $h'(\alpha) \not = 0$ for each root $\alpha$ of $\bar{h}(x)$ mod $l$. Hensel's lemma then implies each root of $\bar{h}(x)$ mod $l$ lifts uniquely to a root of $\bar{h}(x)$ mod $l^k$ and in particular, we have $|\pi_i^{-1}(x)| \leq \deg h$. As there are finitely many $U_i$'s, there is a uniform upper bound on $|\pi_i^{-1}(x)|$ as desired and this proves our upper bound.    
\end{proof}

As $\Sp_{2g}$ is smooth over $\bb{Z}$, we may cover $\Sp_{2g}$ locally with open sets $U_i$ such that there is a morphism $f_i: U_i \to \bb{A}^{m}$, where $m = \dim \Sp_{2g}$, which is étale. The discriminant of the characteristic polynomial cuts out a codimension one subscheme $X$ of $U_i$, and we may take its scheme-theoretic image $Y$ which is a codimension one subscheme of $\bb{A}^m$. We can then apply Proposition \ref{UpperBoundPoints1} to bound $Y(\bb{Z}/l^k\Z)$ and then apply Lemma \ref{lem: étaleBound1} to bound $X(\bb{Z}/l^k\Z)$.
\par 
However, we would like to obtain a bound for the discriminant polynomial of $\GSp_{2g,p}$ that is uniform in $p$. For each fixed $l$, we have an isomorphism between $\GSp_{2g,p}$ and $\GSp_{2g,p'}$ over $\bb{Z}_l$, if $p/p'$ is a nonzero quadratic residue mod $l$, given by sending the matrix $M$ to $cM$ where $c^2 = p/p'$ in $\bb{Z}_l$, which multiplies the discriminant by a power of $c$. Hence, the number of matrices $M \in \GSp_{2g,p'}(\bb{Z}/l^k)$ with $l^k \mid \disc(f_M)$ is exactly the number of matrices $M' \in \GSp_{2g,p}(\bb{Z}/l^k)$ with $l^k \mid \disc(f_{M'})$. Thus, we effectively have two cases depending on whether $p$ is a quadratic residue or non-residue.
 
Now, we consider a different isomorphism between $\Sp_{2g}$ and $\GSp_{2g,p}$ given by multiplication by $\diag(p,\ldots,p,1,\ldots,1)$. This isomorphism however does not respect the discriminant and instead we obtain a polynomial $P(M,p)=P(x_1,\ldots,x_{4g^2},p)$ on $\Sp_{2g}$ corresponding to $\disc(f_M)$ on $\GSp_{2g,p}$. Here, $x_1,\ldots,x_{4g^2}$ are the entries of the matrix in $\Sp_{2g}$.

In each chart $U_i \to \bb{A}^m$, let $U_i = \Spec A_i$. Then we have a polynomial expression $P(x_1,\ldots,x_{4g^2},p)$ where each $x_i \in A_i$. Let $B = \bb{Z}[x_1,\ldots,x_m]$ so that $A_i$ is an étale $B$-algebra. By shrinking $U_i$, we may assume that $A_i$ is a standard étale algebra, i.e. $A_i \simeq B[x]_g/(h)$ where $g,h \in B[x]$ and $h$ is monic with $h'$ invertible in $B[x]_g/(h)$. We now show that there is a polynomial $R(x) \in B[x]$ such that $R(p)\in B$ is a multiple of $P(M,p)$ in $A_i$. Since the subscheme $X$ is cut out by $P(M,p)$, this means that $R(p)$ vanishes in the scheme-theoretic image $Y$.

\begin{lemma} \label{lem: SchemeImage1}
Let $B$ be a domain and let $A = B[x]_g/(h)$ be a standard étale algebra. Let $x_1,\ldots,x_n \in A$ and $y \in B$. Consider the element $P(x_1,\ldots,x_n,y)$ where $P$ is a polynomial in $n+1$ variables with $\bb{Z}$-coefficients. There is then a polynomial $R(x) \in B[x]$, depending only on $P$ and $x_i$'s, such that $R(y)$ is a multiple of $P(x_1,\ldots,x_n,y)$ in $A$.  
\end{lemma}

\begin{proof}
Note that $B[x]/h$ is integral over $B$ since $h$ is monic. Let $A'$ be the integral closure of $B$ in $\Frac(B[x]/h)$ so that $A \subseteq A'$. Let $g'$ be the norm of $g$ with respect to the finite extension $\Frac(A)$ over $\Frac(B)$. Then for some sufficiently large $c$, we have that $(g')^c x_i$ is integral over $B$. Now consider the product
$$Q(y) = \prod_{(x_1',\ldots,x_n')} P(x_1',\ldots,x_n',y)$$
where we vary $x_i'$ over all conjugates of $x_i$ over $B$, of which there are only finitely many since $\Frac(A)$ is a finite extension of $\Frac(B)$. This expands out to a symmetric polynomial in each of the $x_i'$'s and since $(g')^c x_i$ is integral, it follows that $Q(y) = \frac{R(y)}{(g')^{c'}}$ for some polynomial $R \in B[x]$ that depends only on $P$ and the $x_i$'s. Hence we have 
$$R(y) = (g')^{c'}\prod_{(x_1',\ldots,x_n')} P(x_1',\ldots,x_n',y).$$
Let $P'$ be the product of $(g')^{c'}$ along with the terms $ P(x_1',\ldots,x_n'y)$, where we vary all possible $(x_1',\ldots,x_n')$ except our original tuple $(x_1,\ldots,x_n)$. Then $P' \in \Frac(B[x]_g/(h))$ and since $P'$ is clearly integral over $B$, it must lie in $A'$.
\par 
Now note that there is an element $f \in B$ so that $fA' \subseteq A$. Indeed, we know that $A'$ is finitely generated as a $A$-module, say by $a_1,\ldots,a_r$ and for each $a_i$, we can find a suitable $f_i \in A$. Each $f_i$ has a multiple $b_i \in B$ and then we can take  $f = \prod_{i=1}^{r} b_i$. Thus $f P' \in A$ with $f \in B$ and it follows that 
$$f \cdot R(y) = P(x_1,\ldots,x_n,y) \cdot f \cdot P'$$
with $fP' \in A \implies f R(y)$ is a multiple of $P(x_1,\ldots,x_n,y)$ as desired.
\end{proof}
Now we put everything together to show that most matrices in $\GSp_{2g,p}(\Z/l^k\Z)$ have nonzero discriminant modulo $l^k$.
\begin{proof}[Proof of Proposition \ref{prop: l^k not dividing disc lower bound}]
Recall our setup above where we cover $\Sp_{2g}$ with finitely many opens so that we have a morphism $f_i: U_i \to \bb{A}^m$ that is standard étale with $U_i = \Spec A_i$. On each chart $U_i$, there is a polynomial $P_i(x_1,\ldots,x_{4g^2},p)$ where $x_i\in A_i$, such that under the isomorphism $\Sp_{2g} \simeq \GSp_{2g,p}$ over $\bb{Z}$ it is equal to $\disc(f_M)$. Also recall that for each $l$, we have two cases depending on whether $p$ is a quadratic or non-quadratic residue, and it suffices to give a bound for a specific choice of $p$ in each case.

Let $X_{i,p} = \Spec A_i/(P_i(x_1,\ldots,x_{4g^2},p))\subseteq U_i$. It suffices to give an upper bound on $\# X_{i,p}(\bb{Z}/l^k\Z)$, and so by Lemma \ref{lem: étaleBound1} it suffices to bound $\# Y_{i,p}(\bb{Z}/l^k\Z)$ where $Y_{i,p}$ is the scheme theoretic image of $X_{i,p}$ under $f_i$. By Lemma \ref{lem: SchemeImage1}, we have a polynomial $R_i(x) \in B[x]$ with $B = \bb{Z}[x_1,\ldots,x_m]$ such that $R_i(p)$ is a multiple of $P_i(x_1,\ldots,x_{4g^2},p)$ in $A_i$, which implies that $Y_{i,p}$ is a subscheme of $V(R_i(p))$. Hence, it suffices to bound the number of solutions to $R_i(p) = 0 $ mod $l^k$ where $R_i(p)$ is viewed as a polynomial in $\Z[x_1,\ldots,x_m]$.

Proposition \ref{UpperBoundPoints1} tells us that we can get this by giving a bound for $||R_i(p)||_l$. For the case of quadratic residues we can simply choose $p=1$ which gives a constant bound. For non-quadratic residues, we need a little more work. The coefficient of $x_1^{i_1} \cdots x_m^{i_m}$ for $R_i(p)$ is a polynomial $R_{i_1,\ldots,i_m}(p)$ in $p$. Pick one such polynomial which is not identically zero, then since this can have at most a finite number of zeroes, for sufficiently large $l$ we can choose a non-quadratic residue $p$ such that $l$ doesn't divide $R_{i_1,\ldots,i_m}(p)$. There are only a finite set of small $l$ so $l^{\ord_l(R_{i_1,\ldots,i_m}(p))}\le C$ is uniformly bounded. Thus, by possibly dividing $R_i(p)$ by a small power of $l$ and applying Proposition \ref{UpperBoundPoints1}, we have for some uniform constant $\delta > 0$
$$\# \{ R_i(x_1,\ldots,x_m,p) = 0 \text{ mod } l^k \} \leq l^{k(m - \delta)}.$$

As we only have finitely many $i$'s, we may find a $p$ that works for every $i$. Thus, taking the union across $U_i$ we obtain
$$\# \left\{M \in \GSp_{2g,p}(\bb{Z}/l^k\Z) \ \middle\vert \ l^k \mid \disc(f_M)\right\} \leq l^{ k(\dim \Sp_{2g} - \delta)}.$$
We get the desired result after dividing by $\#\GSp_{2g,p}(\Z/l^k\Z)=\Theta(l^{k\dim\Sp_{2g}})$.
\end{proof}
\section{Effective Chebotarev density theorem in families}\label{sec: chebotarev}
As mentioned in the proof overview, our goal in this section is to use an effective Chebotarev density theorem in families of number fields to truncate the infinite product $\prod_l v_l(f_A)$ appearing in Theorem \ref{thm: PPAV counts for char poly}.

We state the effective Chebotarev density theorem given in \cite{jessethorner}, which extends the work of \cite{PTW20} for more general Galois groups and does not rely on the strong Artin conjecture. Let $G$ be a nontrivial finite group and $L/\Q$ be a Galois extension of number fields with $\Gal(L/\Q)\cong G$. A rational prime $l$ is unramified when $l\nmid D_{L}$, and in this case define the Artin symbol $\left[\frac{L/\Q}{l}\right]$ to be the conjugacy class of $\Frob_l$ in $G$. We want an effective bound on how $\left[\frac{L/\Q}{l} \right]$ is equidistributed, so for a conjugacy class $C\subseteq G$ we define
$$\pi_C(x,L/\Q)\coloneqq \# \left\{l \text{ prime }\middle\vert \ l\le x,\ l\nmid D_K,\ \left[\frac{L/\Q}{l} \right]=C\right\}.$$
Then, the effective Chebotarev theorem bounds how far this is from the expected amount.
\begin{theorem}
[{\cite[Theorem 2.1]{jessethorner}}]\label{thm: eff chebotarev}
Fix a nontrivial finite group $G$, and constants $\epsilon >0$ and $Q\ge 3$. Let $\mathfrak F_G$ be a family of number fields $L$ that are Galois over $\Q$ with $\Gal(L/\Q) \cong G$. Define $\mathfrak F_G(Q)=\{L\in \mathfrak F_G \colon D_L\le Q\}$, and 

$$\mathfrak m_{\mathfrak F_G(Q)}\coloneqq \max_{L'\in \mathfrak F_G(Q)}\#\{L\in \mathfrak F_G(Q)\colon L\cap L'\neq \Q\}.$$

For all except $O(\mathfrak m_{\mathfrak F_G(Q)}Q^\epsilon)$ number fields $L\in \mathfrak F_G(Q)$, the following statement is true. If $C\subseteq G$ is a conjugacy class, then
\begin{equation}\label{eqn: eff chebotarev}
\left|\pi_C(x,L/\Q)- \frac{|C|}{|G|}\Li(x)\right| =O\left( x \exp\left(-\frac{1}{29}\left(\frac{\log(x)}{|G|}\right)^{1/2}\right)\right) \text{ for }x\ge (\log D_L)^{10^9|G|^3/\epsilon}.
\end{equation}
\end{theorem}
Call $L$ an $\epsilon$-good number field if Equation \eqref{eqn: eff chebotarev} is true for the constant $\epsilon>0$. In Section \ref{sec: truncation}, we use Equation \eqref{eqn: eff chebotarev} in the case when $\widetilde K$ is $\epsilon$-good to deduce that the tail end $\prod_{l>l_0}v_l(f_A)$ is close to $1$, allowing us to truncate the product as desired. In Section \ref{sec: bounding number of exceptional fields}, we bound the number of exceptional fields which are not $\epsilon$-good, showing that the truncation works generically.
\subsection{Prime splitting and truncation}\label{sec: truncation}
Let $A$ be a simple, ordinary abelian variety with generic Galois group $G=(\Z/2\Z)^g\rtimes S_g$, and $l$ be a good rational prime satisfying $l\nmid 2p\disc(f_A)$, so we have $v_l(f_A)=\zeta_{K,l}(1)/\zeta_{K^+,l}(1)$ by Lemma \ref{lem: good prime ratio of zeta functions}. Furthermore, $l$ is unramified in $K\cong \Q[x]/f_A(x)$ and thus also unramified in both $K^+$ and the Galois closure $\widetilde K$. By definition, we have that 
$$\zeta_{K,l}(1) = \prod_{\text{prime } \mathfrak l\subseteq \mathcal O_K, \, \mathfrak l\mid l} \frac{1}{1-N(\mathfrak l)^{-1}} = 1 + \frac{\#\{\text{prime } \mathfrak l\subseteq \mathcal O_K \mid N(\mathfrak l)=l \}}{l} + O(l^{-2}).$$
The analogous statement applies for the totally real field $K^+$, so combining them we get
$$\frac{\zeta_{K,l}(1)}{\zeta_{K^+,l}(1)} = 1 + \frac{a_{K,l}}{l} + O(l^{-2})$$
where we define
$$a_{K,l} \coloneqq \#\{\text{prime } \mathfrak l\subseteq \mathcal O_K \mid N(\mathfrak l)=l \}-\#\{\text{prime } \mathfrak l\subseteq \mathcal O_{K^+} \mid N(\mathfrak l)=l \}.$$ 

To determine $a_{K,l}$, we need to know how $l$ splits in both $K$ and $K^+$, and the following Proposition tells us that this information is determined by $\left[\frac{\widetilde K/\Q}{l}\right]$. 
\begin{proposition}[{\cite[Proposition 2.8, Chapter III]{Janusz_1996}}] Let $H$ be the subgroup of $G$ fixing $K$, and suppose
$\left[\frac{\widetilde K/\Q}{l}\right]$ has cycles of length $t_1,\ldots, t_s$ when acting on the cosets of $H$ in $G$. Then, the rational prime $l$ splits into primes $\mathfrak l_1,\ldots, \mathfrak l_s$ in $K$ where $N(\mathfrak l_i)=l^{t_i}$. The same statement holds when we replace $K$ with $K^+$ and $H$ with $H^+$.
\end{proposition}
Hence, for $\sigma \in G$ we define $a_{\sigma}=|(G/H)^\sigma|-|(G/H^+)^\sigma|$ to be the difference in the number of cosets of $H$ and $H^+$ which are fixed by $\sigma$. This is the same up to conjugacy, so for a conjugacy class $C\subseteq G$ we define $a_C = a_\sigma$ for $\sigma\in C$. Hence, the above proposition tells that $a_{K,l}=a_{\left[\frac{\widetilde K/\Q}{l}\right]}$. 
\begin{lemma}\label{lem: average a_sigma}
The average of $a_\sigma$ among all $\sigma \in G$ is $0$.
\end{lemma}
\begin{proof}
We show that the average size of $(G/H)^\sigma$ is $1$. Indeed, any coset $gH$ is a fixed point of $\sigma$ if and only if $\sigma gH = gH$, i.e. $g^{-1}\sigma g\in H$ which happens with $|H|/|G|$ probability. There are $|G|/|H|$ possible cosets, so by linearity of expectation, we expect an average of $1$ coset of $H$ in $G$ to be fixed. The same argument shows that the average size of $(G/H^+)^\sigma$ is also $1$, so by taking the difference we have the desired result. In principle, this proof works because both $\zeta_K$ and $\zeta_{K^+}$ have simple poles at $s=1$. 
\end{proof}
Heuristically, Chebotarev's density theorem tells us that $\left[\frac{\widetilde K/\Q}{l}\right]$ is equidistributed, so by the lemma above, $\zeta_{K,l}(1)/\zeta_{K^+,l}(1)$ should have an average of $1+O(l^{-2})$ across all primes $l$. The factors greater than $1$ should cancel out with those that are smaller than $1$, making the tail end of the infinite product equals to $1$ plus a small error. We make this rigorous starting from Equation \eqref{eqn: eff chebotarev}, in a similar way to \cite[Section 4]{Ma_2024}.
\begin{proposition}\label{prop: truncation}
Suppose $\widetilde K$ is an $\epsilon$-good number field, i.e. it satisfies Equation \eqref{eqn: eff chebotarev} with the constant $\epsilon>0$. Then, there exists constants $C_1, C_2>0$ such that for $l_0 \ge (\log D_L)^{C_1/\epsilon}$, we have $$\prod_{l>l_0} \frac{\zeta_{K,l}(1)}{\zeta_{K^+,l}(1)} = 1 + O\left(\exp\left(-C_2\log(l_0)^{1/2}\right)\right).$$ 
\end{proposition}
\begin{proof}
First, we note that the contribution from ramified primes are insignificant as each term is $1+O(l_0^{-1})$ and there are at most $\log(D_L)$ many ramified primes. From now on we only consider the product over unramified primes, and by taking the logarithm, we have
$$\sum_{l> l_0} \log\left(\frac{\zeta_{K,l}(1)}{\zeta_{K^+,l}(1)}\right) = \sum_{l> l_0} \left(\frac{a_{K,l}}{l} + O(l^{-2})\right) =\sum_{l>l_0} \frac{a_{K,l}}{l} + O(l_0^{-1}).$$
By partial summation,
$$\sum_{l>l_0}\frac{a_{K,l}}{l}=\int_{l_0}^\infty \frac{s(x)}{x^2}$$
where 
\begin{equation*}
\begin{split}
s(x) &\coloneqq \sum_{l_0< l \le x} a_{K,l}=\sum_{C} a_C \left(\pi_C(x,L/\Q)-\pi_C(l_0,L/\Q)\right)\\
&=\sum_C a_C \frac{|C|}{|G|}(\Li(x)-\Li(l_0)) + O\left(x \exp\left(-C_3\log(x)^{1/2}\right)\right)
\end{split}
\end{equation*}
where the latter sum is over conjugacy classes $C\subseteq G$ and $C_3>0$ is some constant. By Lemma \ref{lem: average a_sigma}, the sum in the second line evaluates to zero and we are left with the error term. Substituting this back, we have $$\sum_{l> l_0} \log\left(\frac{\zeta_{K,l}(1)}{\zeta_{K^+,l}(1)}\right) = O\left(\int_{l_0}^\infty \frac{\exp(-C_3\log(x)^{1/2})}{x}+ l_0^{-1}\right)=O\left(\exp\left(-C_3\log(l_0)^{1/2}\right)\right)$$
which gives the desired result after exponentiating.
\end{proof}
\subsection{Hilbert's irreducibility theorem} In order for Proposition \ref{prop: truncation} to be useful, we need to show that most of the $\widetilde K$ are $\epsilon$-good. The main input for this is an effective Hilbert irreducibility theorem. Most versions in the literature either do not give explicit dependence on the coefficients on the polynomial in question, or give a polynomial dependence, which is not strong enough for our purposes. We state the version in \cite{effectivehilbertsirreducibilitytheorem} which only has a $\log$ dependence on the coefficients of the polynomial which this was in turn proven by means of a Bombieri-Pila type of bound. Here, we simplify the statement in the case when the polynomial has coefficients in $\Z$. In \cite{effectivehilbertsirreducibilitytheorem}, only the statement bounding the number of specializations such that $f(t_1,\ldots,t_s,Y)$ has a root is given but it is straightforward to deduce the corresponding statement regarding irreducibility following \cite[Lemma 4.2]{effectivehilbertsirreducibilitytheorem}. Here, we denote $[B]=\Z\cap[-B,B]$, and the height $H(f)$ of an integer polynomial can be defined as the maximum absolute value of its coefficients divided by the gcd of them.
\begin{theorem}[{\cite[Theorem 5.5]{effectivehilbertsirreducibilitytheorem}}] \label{thm: EffectiveHilbertIrreducibility1} 
Fix $s\ge 1$ and let $f(T_1,\ldots, T_s,Y)\in \Z [T_1,\ldots, T_s,Y]$ be irreducible of degree $d_Y$ in the variable $Y$ and degree $d_T$ in the $T_i's$. Then there are positive constants $\mu,\nu$ depending on $s,d_T,d_Y$ such that
$$\# \{(t_1,\ldots, t_s)\in [B]^s \mid f(t_1,\ldots,t_s,Y) \text{ has a solution } y \in \bb{Z}\} = O\left( (\log H(f) + 1)^{\mu} B^{s - \frac{1}{2}} (\log (B))^{\nu}\right).$$
\end{theorem}

\begin{corollary} \label{cor: EffectiveHilbertIrreducibility2}
Fix a positive integer $s$ and let $f(T_1,\ldots,T_s,Y) \in \Z [T_1,\ldots,T_s,Y]$ be irreducible of degree $d_Y$ in the variable $Y$ and $d_T$ in the $T_i$'s. Then there are positive constants $\mu,\nu$ depending on $s,d_T,d_Y$ such that 
$$\# \{(t_1,\ldots,t_s) \in [B]^s \mid f(t_1,\ldots,t_s,Y) \text{ is reducible over } \bb{Z}\} = O\left((\log H(f)+1)^{\mu} B^{s - \frac 1 2}( \log B)^{\nu}\right).$$
\end{corollary}

\begin{proof}
We follow the strategy of \cite[Proposition 4.1]{effectivehilbertsirreducibilitytheorem}, adapting it for multiple variables. We first assume that $f$ is monic in $Y$. Let $\bar{\bb{Q}(T_1,\ldots,T_s)}$ be the algebraic closure of $\bb{Q}(T_1,\ldots,T_s)$ and let $f(\mathbf{T},Y) = \prod_{i=1}^{d_Y} (y - y_i)$ where each $y_i \in \bar{\Q(T_1,\ldots,T_s)}$. Each $y_i$ is integral over $\bb{Z}[T_1,\ldots,T_s]$ and so for any $\omega \subseteq \{1,\ldots,d_Y\}$ and non-negative integer $j$ with $j \leq |\omega|$, the element $\tau_{\omega,j}$ which is the $j^{th}$ symmetric sum of the $y_i$'s in $\omega$ is integral over $\bb{Z}[T_1,\ldots,T_s]$. 
Let $P_{\omega,j}(\mathbf{T},Y)$ be the minimal polynomial of $\tau_{\omega,j}$. Then it is monic and has coefficients in $\bb{Z}[T_1,\ldots,T_s]$. 
\par 
Now suppose we have a specialization $t_1,\ldots,t_s \in \bb{Z}$ such that $f(t_1,\ldots,t_s,Y)$ is reducible over $\Z$, so we obtain a factorization 
$$f(t_1,\ldots,t_s,Y) = \prod_{i \in \omega} (Y - y_i(\mathbf t)) \prod_{i \not \in \omega} (Y - y_i(\mathbf t))$$
for some $\omega \subseteq d_Y$ with $1 \leq |\omega| \leq d_Y - 1$. Take $R(Y) = \prod_{i \in \omega} (Y - y_i(\mathbf t))$. Then the coefficients of $R(Y)$ are $\tau_{\omega,j}(t_1,\ldots,t_s)$ up to a sign. At least one of the $\tau_{\omega,j}$ does not lie in $\Q(T_1,\ldots,T_s)$ else $f(\mathbf{T},Y)$ would be reducible over $\Q(T_1,\ldots,T_s)$. If we let $P_{\omega,j}$ be the minimal polynomial for this particular $\tau_{\omega,j}$, then $2 \leq \deg P_{\omega,j} \leq 2^{d_Y}$ and $P_{\omega,j}(t_1,\ldots,t_s,Y)$ has $\tau_{\omega,j}(t_1,\ldots,t_s)$ as a root which lies in $\bb{Z}$. 
\par 
It now suffices to note that the number of such $P_{\omega,j}$'s are bounded in terms of $d_Y,d_T$, their degree is bounded in terms of $d_Y$, and their log-height is bounded in terms of a multiple of $(\log H(P) + 1)$ that depends only on $d_Y$ and $d_T$, much like in \cite[Lemma 4.2]{effectivehilbertsirreducibilitytheorem}. Thus by applying Theorem \ref{thm: EffectiveHilbertIrreducibility1} on each of the $P_{\omega,j}$'s, we obtain our upper bound as desired when $f(\mathbf{T},Y)$ is monic in $Y$.
\par 
In general, we write 
$$f(\mathbf{T},Y) = a_0(\mathbf{T}) Y^{d_Y} + \cdots + a_{d_Y}(\mathbf{T})$$
and consider
$$g(\mathbf{T},Y) = a_0(T)^{d_Y - 1} f\left(\mathbf{T}, \frac{Y}{a_0(T)} \right) = Y^{d_Y} + a_1(T) Y^{d_Y - 1} + \cdots + (a_0(T))^{d_Y - 1} a_{d_Y}(T).$$
This expands to a monic polynomial in $Y$ with degree $d_Y$ in $Y$ and degree bounded in terms of $d_T$ in $T$. It is clear that $\log (H(g) + 1)$ is bounded by a constant, depending on $d_T, d_Y$ and $s$, times $\log (H(f)+1)$ and that $g$ is irreducible since $f$ is irreducible. Under a specialization such that $a_0(t_1,\ldots,t_s) \not = 0$, we have $g(t_1,\ldots,t_s,Y)$ is reducible if and only if $f(t_1,\ldots,t_s,Y)$ is. If we had a specialization such that $a_0(t_1,\ldots,t_s) = 0$, we are in a case where the degree drops and we can handle that by induction. Hence the general case follows from the monic case as desired. 
\end{proof}
\subsection{Bounding exceptional fields}\label{sec: bounding number of exceptional fields} 
Let $G=(\Z/2\Z)^g \rtimes S_g$ and $\mathfrak F_G$ be the family of number fields $\widetilde K$ for $K\cong \Q[x]/f_A(x)$ where $A$ ranges across all simple ordinary abelian varieties $A/\F_p$ of dimension $2g$ with generic Galois group $\Gal(\widetilde K/\Q)\cong G$. Recall from Proposition \ref{prop: ordinary weil p-poly} that such a polynomial $f_A(x)$ can be written as $f_{\mathbf a}(x)$ for some $\mathbf a\in [p](\mathcal R_g)\cap \Z^g$, so we denote the corresponding fields to be $K_{\mathbf a}$ and $\widetilde K_{\mathbf a}$. Furthermore, note that $\disc(f_{\mathbf a})$ is upper bounded by a polynomial in $p$, and because $f_{\mathbf a}$ is defined to be monic, we see that $\disc(K_{\mathbf a})$ and thus also $\disc(\widetilde K_{\mathbf a})$ are also bounded polynomially in $p$. Hence, we have $\mathfrak F_G=\mathfrak F_G(Q)$ where $Q=O(p^A)$ for some constant $A$.

Recall that our goal is to bound the proportion of $\mathbf a$ for which $\widetilde K_{\mathbf a}$ is not $\epsilon$-good, and we wish to do this by upper bounding $\mathfrak m_{\mathfrak F_G(Q)}= \max_{L'\in \mathfrak F_G(Q)}\#\{L\in \mathfrak F_G(Q)\colon L\cap L'\neq \Q\}$. Hence, for each $\widetilde K_{\mathbf{a'}}$, we want to upper bound the number of fields $\widetilde K_{\mathbf{a}}$ that intersect it nontrivially. Since $\widetilde K_{\mathbf{a'}}$ only has a bounded number of subfields $L\neq \Q$, it suffices to show that for any such subfield that the proportion of $\widetilde K_{\mathbf{a}}$ containing $L$ is small.

Our first step is to show that there is no field $L$ which is contained in every $\widetilde K_{\mathbf{a}}$. 
\begin{lemma}\label{lem: field example}
For any number field $L\neq \Q$, there exists $\mathbf a\in \Z^g$ not necessarily in $[p](\mathcal R_g)$ where $\widetilde K_{\mathbf a}$ has generic Galois group and does not contain $L$.
\end{lemma}
\begin{proof}
We use a discriminant argument. Suppose that there exists such a field $L$ such that every $\widetilde K_{\mathbf a}$ with generic Galois group contains it. Since every nontrivial extension of $\Q$ is ramified, $L$ has to be ramified at some prime $l$, so $\widetilde K_{\mathbf a}$ must also be ramified at $l$. This implies that $K_{\mathbf a}$ is ramified at $l$ because $\widetilde K_{\mathbf a}$ is the compositum of the Galois conjugates of $K_{\mathbf a}$. Thus, for each prime $l$ it suffices to find some $K_{\mathbf a}$ with generic Galois group where $l$ does not ramify.

We first deal with the case $l\neq p$. Consider the polynomials $x^{2g}+p^g$ and $x^{2g}+px^{2g-2}+p^2x^{2g-4}+\cdots+p^g$ and it is easy to show using a root of unity argument that these have discriminants $(-1)^g(2g)^{2g}p^{g(2g-1)}$ and $(-1)^g2^{2g}(g+1)^{2g-2}p^{g(2g-1)}$ respectively. Furthermore, note that $x^{2g}+x^{2g-1}+\cdots +x^g+px^{g-1}+p^2x^{g-2}+\cdots +p^g$ is equal to $x^{2g}+x^{2g-1}+\cdots +1$ modulo $2$ which has odd discriminant. Noting that $g$ and $g+1$ are coprime, we see that for any prime $l\neq p$ at least one of the above polynomials of the form $f_{\mathbf a}$ have discriminant coprime to $l$. Furthermore, this remains true under translate $\mathbf a$ by multiples of $l$. The proof of Proposition \ref{prop: generic galois} shows that a generic $f_{\mathbf a}$ has Galois group $G$ (here, we need to take larger boxes depending on $l$, so $\mathbf a$ may not be in $[p](\mathcal R_g)$), which implies that there is some translate $\mathbf a$ for which $\widetilde K_{\mathbf a}$ has generic Galois group and $l$ does not ramify.

The case of $l=p$ is more difficult because by Lemma \ref{lem: proportion ordp > g(g-1)} we have that $p^{g(g-1)}\mid \disc(f_{\mathbf a})$ for every $\mathbf a$, so the above strategy does not work naively. We choose $\mathbf a \in \bb{Z}^g$ with $\ord_p(a_i) = 0$ for all $i$ such that $\widetilde{K}_{\mathbf a}$ has generic Galois group and also satisfying $p^{g(g-1)} \mid \mid \disc(f_{\mathbf{a}})$ by Lemma \ref{lem: proportion ordp > g(g-1)}. Now let $\alpha$ be a root of $f_{\mathbf a}(x)$ and consider an embedding $\bar{\bb{Q}} \xhookrightarrow{} \bar{\bb{Q}_p}$. It suffices to show that under any such embedding, we must have $\bb{Q}_p(\alpha)$ being unramified over $\bb{Q}_p$.
\par 
First by Newton polygons, it follows that $g$ of the roots of $f_{\mathbf a}$ satisfy $|x|_p = 1$ and the other $g$ satisfy $|x|_p = p^{-1}$. Let $b_1,\ldots,b_g$ be the roots satisfying $|b_i|_p = 1$ and let $c_1,\ldots,c_g$ be the other roots. Since the discriminant satisfy $|\disc(f_{\mathbf a})|_p = p^{-g(g-1)}$, it follows that $|b_i - b_j|_p = 1$ and $|c_i - c_j|_p = 0$. 
\par 
We now have two cases. The first case is when $\alpha$ is among the $b_i$'s, say $\alpha = b_1$. We consider the conjugates of $\alpha$ over $\bb{Q}_p$. Since conjugates have the same absolute value, they must be among the $b_i$'s, say $b_1,\ldots,b_k$. We may then form a basis of integral elements of $\bb{Q}_p(\alpha)$ by $1,\frac{\alpha}{p},\ldots,\frac{\alpha^{k-1}}{p^{k-1}}$ as $|\alpha|_p = 1$.  Since $\frac{b_i}{p}$ are the conjugates of $\frac{\alpha}{p}$ over $\bb{Q}_p$ for $1 \leq i \leq k$, the discriminant of this basis has absolute value equivalent to 
$$\prod_{1 \leq i \not = j \leq k} \left|\frac{b_i}{p} - \frac{b_j}{p} \right|_p= 1$$
as $|\frac{b_i}{p} - \frac{b_j}{p}|_p = 1$ for each $i,j$. It follows that  $\bb{Q}_p(\alpha)$ is unramified over $\bb{Q}_p$. Similarly if $\alpha$ is among the $c_i$'s, the same argument goes through without dividing by $p$ as we already have $|c_i - c_j|_p = 1$. Hence $\bb{Q}_p(\alpha)$ is unramified over $\bb{Q}_p$ for any embedding $\bar{\bb{Q}} \xhookrightarrow{} \bar{\bb{Q}_p}$ which implies that $p$ is unramified in $K_{\mathbf a}=\bb{Q}(\alpha)$ as desired. 
\end{proof}

Now, let $L$ be a minimal subfield (has no nontrivial subfields) of $\widetilde K_{\mathbf a'}$ for some $\mathbf a'\in [p](\mathcal R_g)$. We use the following method to detect whether $L$ and $\widetilde K_{\mathbf a}$ intersect. Let $\alpha$ and $\beta$ be primitive roots of $L$ and $\widetilde K_{\mathbf a}$ respectively and consider $\alpha +\beta$. In the case where the two fields intersect, $L$ must be contained in $\widetilde K_{\mathbf a}$ so $\alpha + \beta$ is also contained inside $\widetilde K_{\mathbf a}$. 

Otherwise, $\alpha+\beta$ must generate the compositum $L\widetilde K_{\mathbf a}$. To see this, first note that the compositum has degree $[L: \Q][\widetilde K_{\mathbf a}:\Q]$ by \cite[Corollary 3.19]{FT_Milne_2022} as $\widetilde K_{\mathbf a}/\Q$ is Galois. Then, suppose otherwise that $\alpha+\beta$ generates a proper subfield of $L\widetilde K_{\mathbf a}$, so there exists different embeddings $\sigma_1,\sigma_2\colon L\widetilde K_{\mathbf a} \hookrightarrow \C$ for which $\sigma_1(\alpha+\beta)=\sigma_2(\alpha+\beta)$. This implies that $\sigma_1(\alpha)-\sigma_2(\alpha)=\sigma_1(\beta)-\sigma_2(\beta)$, but this is contained in the intersection of the Galois closure of $L$ and $\widetilde K_{\mathbf a}$ which is $\Q$. In fact, this must be zero because if we let $c=\sigma_1(\beta)-\sigma_2(\beta)$, then $c=\beta-\sigma_1^{-1}\sigma_2(\beta)$, but $\sigma_1^{-1}\sigma_2$ has finite order, say $m$, in $\Gal(\widetilde K_{\mathbf a}/\Q)$, which means that $0=\beta-(\sigma_1^{-1}\sigma_2)^m(\beta)=mc$. This implies that $\sigma_1$ and $\sigma_2$ are the same embeddings for both $L$ and $\widetilde K_{\mathbf a}$, and hence also for $L\widetilde K_{\mathbf a}$, giving a contradiction.

Using the minimal polynomials of $\alpha$ and $\beta$, we can construct a polynomial with root $\alpha+\beta$ of the expected degree, so the above argument says that this polynomial is irreducible if and only if the fields $L$ and $\widetilde K_{\mathbf a}$ don't intersect. Then, this allows us to apply Hilbert's irreducibility theorem.

Keeping this setting, we construct this polynomial in the following two lemmas, starting with the minimal polynomial for $L$ with root $\alpha$.
\begin{lemma}\label{lem: L min poly construction}
There exists a monic polynomial $P_L(x)\in \Z[x]$ with coefficients at most polynomial in $p$ such that $L\cong \Q[x]/P_L(x)$.
\end{lemma}
\begin{proof}
Let $L$ be the fixed field of the subgroup $H$ in $\widetilde K_{\mathbf{a'}}$. By the proof of the primitive element theorem, we can construct a primitive element $\gamma=\sum_i c_i\gamma_i$ of $K_{\mathbf{a'}}$, where $c_i=O(1)$ are bounded integers and $\gamma_i$ are the roots of $f_{\mathbf{a'}}(x)$. Then, $\gamma$ is a root of the polynomial 
$$Q(x)=\prod_{\sigma_h\in H}(x-\sigma_h(\gamma)),$$ and also note that the coefficients are elementary symmetric polynomials $e_1,\ldots, e_{|H|}$ with variables $\sigma_h(\gamma)$. 

We claim that $L=\Q(e_1,\ldots, e_{|H|})$. Indeed, the latter field is $H$-invariant so $L$ contains it and $\widetilde K_{\mathbf {a'}}$ is at most a degree $|H|$ extension of the field by $Q(x)$ so we must have equality. Again, we can find a primitive element $\delta = \sum_i d_ie_i$ of $L$ where $d_i$ are bounded integers, and we let $$P_L(x)=\prod_{\sigma_g\in G/H} (x-\sigma_g(\delta))$$
which is well defined since $\delta$ is already $H$-invariant. The coefficients of $P_L(x)$ are $G$-invariant so they are in $\Q$, and furthermore they are algebraic integers and hence in $\Z$, so this is indeed a monic integer minimal polynomial of $L$. 

In fact, the coefficients of $P_L(x)$ are polynomials of the form $R(\gamma_i)$ which are invariant when the indices are permuted according to $G=(\Z/2)^g\rtimes S_g$. By pairing the roots into pairs $\{\gamma_{2i+1},\gamma_{2i+2}\}$, the polynomial $R(\gamma_i)$ is symmetric in each pair and thus can be expressed as elementary symmetric polynomials in each pair. However, recall that the pair $\{\gamma_{2i+1},\gamma_{2i+2}\}$ are the solutions to $x^2-\alpha_ix+p=0$ where $\alpha_i$ are the roots of $f^+_{\mathbf {a'}}$, so we can write $R(\gamma_i)=R'(\alpha_i,p)$. But now the expression in terms of $\alpha_i$ is $S_g$ invariant, so it can be expressed as the coefficients of $f^+_{\mathbf{a'}}$ which in turn can be expressed as the coefficients of $f_{\mathbf{a'}}$, which have size at most polynomial in $p$. Tracing back, because $c_i$ and $d_i$ were bounded, we have the desired result.
\end{proof}
Now, we construct the polynomial for $\alpha+\beta$ as discussed previously.
\begin{lemma}\label{lem: poly for alpha+beta}
There exists an polynomial $R(x,\mathbf a)\in\Z[x,a_1,\ldots, a_g]$ irreducible in the variable $x$ with coefficients bounded polynomially in $p$, such that for any $\mathbf a\in \Z^g$, if $R(x,\mathbf a)$ is irreducible as a polynomial in $x$ then $\widetilde K_{\mathbf a}$ has generic Galois group and does not intersect $L$.
\end{lemma}
\begin{proof}
For now, we fix $\mathbf a$ so that $\widetilde K_{\mathbf a}$ is the field obtained in Lemma \ref{lem: field example} and choose bounded $c_i$ such that $\beta=\sum_i c_i\gamma_i$ is a primitive root of $\widetilde K_{\mathbf a}$, where here $\gamma_i$ are the roots of $f_{\mathbf a}(x)$. Let $Q(x)=\prod_{\sigma\in G} (x-\sigma(\beta))$, this has degree $|G|$ and just like the proof of Lemma \ref{lem: L min poly construction} we can write each coefficient of $Q(x)$ as a polynomial in $\mathbf{a}$ with coefficients bounded polynomially in $p$, so we can rewrite this as $Q(x,\mathbf a)$.

The polynomials $P_L(x)$ and $Q(x)$ have roots $\alpha$ and $\beta$ which are primitive roots of $L$ and $\widetilde K_{\mathbf a}$ respectively. We want a polynomial $R(x)$ with root $\alpha+\beta$, so we consider the elements $(\alpha+\beta)^i$ for $0\le i\le |G|[L:\Q]$. Using the monic polynomials $P_L(x)$ as well as $Q(x)$ we can express each of these as a linear combination of monomials $\alpha^j\beta^k$ for $0\le j<[L:\Q]$ and $0\le k<|G|$. By a simple Gaussian elimination argument, there is a linear dependence between $(\alpha+\beta)^i$ which gives rise to a not necessarily monic polynomial $R(x)$ of degree $[L\colon\Q]|G|$ with $\alpha+\beta$ as a root. Furthermore, $R(x)$ is irreducible because its root generates $L\widetilde K_{\mathbf a}$ which is a number field of the same degree. Since the coefficients of $R(x)$ can be written as a polynomial in $\mathbf a$ with coefficients bounded polynomially in $p$, we write it in the form of $R(x,\mathbf a)$. The irreducibility of $R(x)$ then implies the irreducibility of $R(x,\mathbf a)$ in the variable $x$.

Now we allow $\mathbf a$ to vary, and suppose for some new choice of $\mathbf a\in \mathbb Z^g$ that $R(x,\mathbf a)$ is irreducible in $x$. Then, we can define the $\gamma_i$ to be the roots of the polynomial $f_{\mathbf a}$ and $\beta=\sum_i c_i\gamma_i$, which value is now dependent on $\mathbf a$. However, since our construction is completely algebraic, it is still true that the new $\alpha+\beta$ is still a root of $R(x,\mathbf a)$. Since it is irreducible, $\alpha+\beta$ generate a field of degree $[L\colon\Q]|G|$. At the same time, it lies in the compositum of $L\widetilde K_{\mathbf a}$, which forces this to be of degree $[L\colon\Q]|G|$. In particular, this implies that $\widetilde K_{\mathbf a}$ has generic Galois group and does not intersect $L$.
\end{proof}
Then, we can apply Hilbert's irreducibility theorem as stated in Corollary \ref{cor: EffectiveHilbertIrreducibility2} to bound the number of fields intersecting $L$.
\begin{lemma}\label{lem: tilde K does not contain L}
There exists some constant $C>0$ independent of $L$ such that there are at most $O(p^{\frac{g(g+1)-1}{4}}(\log p)^C)$ many $\mathbf a\in [p](\mathcal R_g)\cap \Z^g$ where $\widetilde K_{\mathbf a}$ has generic Galois group and also intersects $L$.
\end{lemma}
\begin{proof}
Cover $[p](\mathcal R_g)$ with hypercubes of length $p^{1/2}$ using Proposition \ref{prop: volume R_g}, and there are $O(p^{g(g-1)/4})$ such hypercubes in the cover. Apply Corollary \ref{cor: EffectiveHilbertIrreducibility2} to the polynomial $R(x,\mathbf a)$ in Lemma \ref{lem: poly for alpha+beta}, noting that we can shift the origin of the box $[B]$ while keeping the coefficients and thus the height $H(f)$ bounded by a polynomial in $p$. This shows that there are at most $O(p^{g/2-1/4}(\log p)^C)$ vectors $\mathbf a$ in each box where $R(x,\mathbf a)$ is reducible, and thus $O(p^{\frac{g(g+1)-1}{4}}(\log p)^C)$ such $\mathbf a$ in total.
\end{proof}
Finally, we bound the proportion of $\mathbf a$ for which the corresponding $\widetilde K_{\mathbf a}$ are not $\epsilon$-good. A naive argument would be to say that $\mathfrak m_{\mathfrak F_G(Q)}=O(p^{\frac{g(g+1)-1}{4}}(\log p)^C)$ by Lemma \ref{lem: tilde K does not contain L} and thus Theorem \ref{thm: eff chebotarev} implies that $O(m_{\mathfrak F_G(Q)}Q^\epsilon)$ many number fields that are not $\epsilon$-good. However, it is possible that there are many $\mathbf a$ such that $\widetilde K_{\mathbf a}$ are isomorphic to the same number field, and one can show that this is in fact the case. Nevertheless, we are able to prove it with a small trick.
\begin{proposition}\label{prop: bounding exceptional fields}
There exists a constant $A>0$ such that for any $\epsilon>0$, there are at most $O(p^{-1/4+A\epsilon})$ proportion of irreducible ordinary Weil $p$-polynomials with generic Galois group that are not $\epsilon$-good.
\end{proposition}
\begin{proof}
Enumerate the number fields up to isomorphism in $\mathfrak F_G$ as $K_1,K_2,\ldots$ where for each $K_i$ there are exactly $n_i$ vectors $\mathbf a$ where $K_i\cong \widetilde K_{\mathbf a}$. Note that $\sum n_i=\Theta(p^{g(g+1)/4})$ by Corollary \ref{cor: num ord weil poly}. Split $\mathfrak F_G$ into disjoint subsets indexed by $j$, where the $j$-th set $\mathfrak F_G^j$ contains the $K_i$ for which $n_i\in [2^j,2^{j+1})$, and note that there are $O(\log p)$ such subsets.

For each $j$, consider any $K_i\in \mathfrak F_G^j$ and let $\mathbf{a'}$ be such that $K_i\cong \widetilde K_{\mathbf a'}$. By looking at the minimal subfields and applying Lemma \ref{lem: tilde K does not contain L}, there are at most $O(p^{\frac{g(g+1)-1}{4}}(\log p)^C)$ many $\mathbf a$ where $\widetilde K_{\mathbf a}$ intersects $\widetilde K_{\mathbf a'}$. However, each field in $\mathfrak F_G^j$ is isomorphic to at least $2^j$ many $\widetilde K_{\mathbf a}$, so there are at most $O(p^{\frac{g(g+1)-1}{4}}(\log p)^C2^{-j})$ many $K\in \mathfrak F_G^j$ that intersect $K_i$, giving a bound for $m_{\mathfrak F_G^j(Q)}$. Then, applying Theorem \ref{thm: eff chebotarev} and recalling that $Q$ is polynomial in $p$, there are $O(p^{\frac{g(g+1)-1}{4}+A\epsilon}2^{-j})$ many number fields in $\mathfrak F^j_G$ that are not $\epsilon$-good. But for each one of these fields there are at most $2^{j+1}$ occurrences of it as some $\widetilde K_{\mathbf a}$, so there are at most $O(p^{\frac{g(g+1)-1}{4}+A\epsilon})$ many $\mathbf a$ such that $\widetilde K_{\mathbf a}$ is not $\epsilon$-good that come from $\mathfrak F_G^j$. Summing over all $j$ gives the desired answer.
\end{proof}
\section{Summing Euler products}\label{sec: summing euler products}
In this section, we prove a lemma which allows us to average a finite product of Euler factors in a box, and we combine this with our results in the previous sections to prove Theorem \ref{thm: A_g}.
\subsection{Averaging products of periodic factors}\label{sec: avg periodic factors} After truncating the sum using the effective Chebotarev theorem, we need to take a sum over $\mathbf a$ of $\prod_{l\le l_0} v_{l}(f_{\mathbf a})$. We present the following lemma which helps us average a finite product of Euler factors. This is a generalization of \cite[Proposition 6.4]{Ma_2024} and in a similar spirit to \cite{Kowalski_2011}, \cite[Section 4]{David_Koukoulopoulos_Smith_2016}.
\begin{proposition}\label{prop: averaging out}
Let $m\ge 1$ and for $0\le i\le m$, let the functions $h_i\colon \Z^g\rightarrow \R_{\ge 0}$ be $w_i$-periodic in every variable, have average value $z_i$, and also assume that the $w_i$ are pairwise coprime. Let $T=\max_{0\le i\le m} w_i$ be an upper bound for the periods. Let $m_0=\max_{\mathbf a}(h_0(\mathbf a))$ and $m_i=\max_{\mathbf a}(|\log(h_i(\mathbf a))|)$ for $1\le i\le m$. Let $C=\max_{0\le i\le m}m_i$ and choose $D\ge \sum_{i=1}^m m_i$. Then, the average of the product in a box of side length $N$ is approximately equal product of the averages:
$$\frac{1}{N^g}\sum_{\mathbf a\in [1,N]^g}\prod_{i=0}^m h_i(\mathbf a)= \prod_{i=0}^m z_i + O\left(\frac{D(mCT)^{6D+1}}{N}+C2^{-6D}\right).$$
Here, we denote $[1,N]=\{1,2,\ldots,N\}$.
\end{proposition}
\begin{proof}
By Taylor expansion, the term on the left is
\begin{equation*}
\begin{split}
\frac{1}{N^g}\sum_{\mathbf a\in [1,N]^g}\prod_{i=0}^m h_i(\mathbf a) &= \frac {1}{N^g}\sum_{\mathbf a\in [1,N]^g}h_0(\mathbf a)\exp\left(\sum_{i=1}^m\log(h_i(\mathbf a))\right)\\
&=\sum_{n=0}^\infty \frac{1}{n!N^g}\sum_{\mathbf a\in [1,N]^g}h_0(\mathbf a)\left(\sum_{i=1}^m\log(h_i(\mathbf a))\right)^n\\
&=\sum_{n=0}^\infty \frac{1}{n!}\sum_{\substack{b_1+\cdots +b_m=n, \\b_i \ge 0}}\frac{1}{N^g}\sum_{\mathbf a\in [1,N]^g}h_0(\mathbf a)\prod_{i=1}^m \log(h_i(\mathbf a))^{b_i}.
\end{split}
\end{equation*}

On the other hand, define the independent random variables $H_i$ with the distribution of $h_i(\mathbf a)$ where $\mathbf a$ is chosen randomly, this is well defined by the periodicity of $h_i$. Then, by independence and linearity of expectation applied to the Taylor expansion, the term on the right is
\begin{equation*}
\begin{split}
\prod_{i=0}^m z_i &= z_0\E\left[\prod_{i=1}^m  H_i \right] = z_0 \E\left[\sum_{n=0}^\infty \frac{1}{n!} \left(\sum_{i=1}^m\log(H_i)\right)^n\right]
=z_0\E\left[\sum_{n=0}^\infty\frac{1}{n!}\sum_{\substack{b_1+\cdots +b_m=n, \\b_i \ge 0}}\prod_{i=1}^m  \log(H_i)^{b_i}\right]\\
&=\sum_{n=0}^\infty\frac{1}{n!}\sum_{\substack{b_1+\cdots +b_m=n, \\b_i \ge 0}}z_0\prod_{i=1}^m  \E\left[\log(H_i)^{b_i}\right].
\end{split}
\end{equation*}

It suffices to bound the difference between the inner terms for each $n$. For $n> 6D$, using Stirling's formula we conclude that both terms $$\frac{1}{n!N^g}\sum_{\mathbf a\in [1,N]^g}h_0(\mathbf a)\left(\sum_{i=1}^m\log(h_i(\mathbf a))\right)^n \text{ and  } z_0 \E\left[\frac{1}{n!} \left(\sum_{i=1}^m\log(H_i)\right)^n\right]$$
are $O(CD^{n}/n!)=O(C2^{-n})$.
so summing over $n\ge 6D$ we get an error contribution of $O(C2^{-6D})$.

Next, suppose $n\le 6D$, then consider the term $h_0(\mathbf a)\prod_{i=1}^m \log(h_i(\mathbf a))^{b_i}$ where $\sum b_i=n$. This is periodic in each variable with period dividing $P=w_0\prod_{i \colon b_i>0}w_i\le T^{n+1}$ by condition (a). By Chinese Remainder Theorem, since $w_i$ are coprime, we have in a box $B$ of side length $P$ that
$$\frac{1}{P^g}\sum_{\mathbf a\in B}h_0(\mathbf a)\prod_{i=1}^m \log(h_i(\mathbf a))^{b_i}=z_0\prod_{i=1}^m \E\left[\log(H_i)^{b_i}\right].$$
By subdividing the $[1,N]^g$ into boxes of side length $P$, we have $O(N^{g-1}T^{n+1})$ terms leftover, and after accounting for these terms we have
$$\frac{1}{N^g}\sum_{\mathbf a\in [1,N]^g}h_0(\mathbf a)\prod_{i=1}^m \log(h_i(\mathbf a))^{b_i}=z_0\prod_{i=1}^m \E\left[\log(H_i)^{b_i}\right]+O\left(\frac{(CT)^{n+1}}{N}\right).$$
Summing this for $n\le 6D$ over the tuples $b_1+\cdots +b_m=n$ of which there are less than $m^n$ many, which gives a error contribution of $O\left(D(mCT)^{6D+1}/N\right)$.

\end{proof}

\subsection{Proof of Theorem \ref{thm: A_g}}\label{sec: final proof} Our overall strategy is to first give a lower bound $\# \mathcal A_g(\F_p,t)$ for each $t$, and then turn this into an $L^1$-bound by showing this lower bound accounts for most of the PPAVs in $\mathcal A_g(\F_p)$. We divide the proof into several parts for clarity.

\subsubsection{Part I: Applying Theorem \ref{thm: PPAV counts for char poly}}
Let $\mathcal F\subseteq \Z^g\cap [p](\mathcal R_g)$ be the subset of $\mathbf a$ satisfying the following conditions:
\begin{enumerate}[(\alph*)]
    \item $f_{\mathbf a}$ is an irreducible ordinary Weil $p$-polynomial (so $f_{\mathbf a}=f_A$ for a simple ordinary abelian variety $A$).
    \item $f_{\mathbf a}$ has generic Galois group $G=(\Z/2\Z)^g\rtimes S_g$.
    \item $\ord_p(\disc(f))=g(g-1)$.
    \item The corresponding $\widetilde K_{\mathbf a}$, which we recall to be the Galois closure of $\Q[x]/f_{\mathbf a}(x)$, is $\epsilon_1$-good, where we choose a sufficiently small constant $\epsilon_1>0$.
\end{enumerate}
By Lemma \ref{lem: generic irreducibility}, Proposition \ref{prop: generic galois}, Lemma \ref{lem: proportion ordp > g(g-1)} and Corollary \ref{prop: bounding exceptional fields}, there are at most $O(p^{-1/5})$ proportion of ordinary Weil $p$-polynomials which do not satisfy at least one of the conditions above. For each trace $t$, define $\mathcal F_t=\mathcal F\cap \{a_1=t\}$.

Applying Theorem \ref{thm: PPAV counts for char poly}, we have
\begin{equation*}
\# \mathcal A_g(\F_p,t)= \sum_{\mathbf a\in \mathcal \Z^g\cap [p](\mathcal R_g)} \#\mathcal A_g(\F_p,f_{\mathbf a}) \ge \sum_{\mathbf a\in \mathcal F_t} \#\mathcal A_g(\F_p,f_{\mathbf a}) = p^{\dim(\mathcal A_g)/2} \sum_{\mathbf a\in \mathcal F_t} v_\infty(f_{\mathbf a})\prod_l v_l(f_{\mathbf a}).
\end{equation*}
\subsubsection{Part II: Truncating to a finite product} 
By Lemma \ref{lem: good prime ratio of zeta functions} and \ref{lem: vp ratio zeta fns} we have that for $l\nmid 2p\disc(f_{\mathbf a})$ and $l=p$ that $v_l(f_{\mathbf a})=\zeta_{K,l}(1)/\zeta_{K^+,l}(1)$. There are at most $O(\log(p))$ remaining bad primes and Proposition \ref{prop: vl lower bound} tells us for $l\neq p$ that $v_l(f_{\mathbf a})\ge 1-C_1/l$ for some constant $C_1$. Choose $l_0=\exp(C_2\log\log(p)^2)$ for some sufficiently large constant $C_2$.
Then, combining these with Proposition \ref{prop: truncation}, we lower bound the tail by
$$\prod_{l>l_0} v_l(f_{\mathbf a})\ge 1-O\left(\log(p)^{-1+\epsilon}\right).$$
\subsubsection{Part III: Lower bounding by the auxiliary factor $v_l'(f)$} The point of introducing this auxiliary factor is two-fold: (1) to make the local factors periodic and (2) being able to upper bound the local factors. This is a complication that arises because of our imperfect results in Section \ref{sec: bounds, stab, disc}. For each prime $l\le l_0$, we choose an integer $k=k(l)$ depending on $l$ as follows. Let $l_1=\log(p)^{1-\epsilon}$, then for $l\le l_1$ define $k$ to be the largest integer such that $l^k\le l_0$ which implies that $l^k>\sqrt{l_0}$, and for $l_1<l\le l_0$ set $k=1$. Define the auxiliary factor
$$v_l'(f_{\mathbf a})=\begin{cases}
v_{l,k}(f_{\mathbf a}) & \text{if } \ord_l(\disc(f_{\mathbf a}))\le k/C_3\\
\max(1-C_1/l,0) & \text{otherwise}
\end{cases}$$ 
where we pick $C_3>1$ to be the constant in Proposition \ref{prop: stabilization} so that we have $v_l(f_{\mathbf a})\ge v_l'(f_{\mathbf a})$. Furthermore $v_l'(f_{\mathbf a})$ is $l^k$-periodic in each of the variables $\mathbf a$. Recall that by Proposition \ref{prop: vl lower bound} we have a lower bound $v_l'(f_{\mathbf a})\ge 1-C_1/l$, but now we also have an upper bound $v_l'(f_{\mathbf a})\le l^{k(g+1)}\le l_0^{g+1}$ simply by the definition of $v_{l,k}$. 

Furthermore, if $l>l_1$ we can show that $v_l'(f_{\mathbf a})=1+O(l^{-1})$. Indeed, in this case, we defined
$k=1$, so $\ord_l(\disc(f_{\mathbf a}))\le k/C_3$ implies that $l\nmid \disc(f_{\mathbf a})$ so $l$ is a good prime and $v_{l,k}(f_{\mathbf a})=1+O(l^{-1})$ as it is a ratio of local zeta functions by Lemma \ref{lem: good prime ratio of zeta functions}. 

Combining the above, we have
$$\# \mathcal A_g(\F_p,t)\ge \left(1-O\left(\log(p)^{-1+\epsilon}\right)\right)p^{\dim(\mathcal A_g)/2} \sum_{\mathbf a\in \mathcal F_t} v_\infty(f_{\mathbf a})\prod_{l\le l_0} v_l'(f_{\mathbf a}).$$
\subsubsection{Part IV: Making the region a disjoint union of boxes}
Consider a subdivision of $\R^g$ into boxes $B$ of side length $N=p^{1/4}$, and let $[p](\mathcal R_g^-)=\bigsqcup_{B\in \mathcal B}B$ be the union of those boxes which are contained inside $[p](\mathcal R_g)$. Define $\mathcal F'=\Z^g\cap [p](\mathcal R_g^-)=\bigsqcup_{B\in \mathcal B} B \cap \Z^g$ and $\mathcal F'_t = \mathcal F' \cap \{a_1 =t\}=\bigsqcup_{B_t\in \mathcal B_t} B_t\cap \Z^g$, where the family $\mathcal B_t =\{B_t=B\cap \{t\le a_1<t+1\}\mid B\in \mathcal B, B_t\neq \emptyset\}$ consists of the nonempty slices of boxes at trace $t$. Then, to convert the sum from being over $\mathcal F_t$ to $\mathcal F_t'$, note that for any $\mathbf a\in \mathcal F'$, we upper bound
\begin{equation}\label{eqn: upper bound for each f}
v_\infty(f_{\mathbf a})\prod_{l\le l_0} v_l'(f_{\mathbf a})\le O(1) \cdot\prod_{l_1<l\le l_0} (1+O(l^{-1})) \cdot\prod_{l\le l_1} l_0^{g+1} = O(p^{1/6})
\end{equation}
Hence, adding the terms with $\mathbf a\in \mathcal F'_t\setminus \mathcal F_t$ we obtain 
\begin{equation}\label{eqn: first lower bound}
\begin{split}
\# \mathcal A_g(\F_p,t)\ge & \left(1-O\left(\log(p)^{-1+\epsilon}\right)\right)p^{\dim(\mathcal A_g)/2} \left(\sum_{\mathbf a\in \mathcal F_t'} v_\infty(f_{\mathbf a})\prod_{l\le l_0} v_l'(f_{\mathbf a})-|\mathcal F'_t\setminus \mathcal F_t|O(p^{1/6})\right).
\end{split}
\end{equation}
\subsubsection{Part V: Averaging out the Euler product in boxes}
We wish to apply Proposition \ref{prop: averaging out} to average out the product in the boxes of side length $N$. In the notation of the proposition, take $h_0(\mathbf a)=\prod_{l\le 2C_1} v_l'(f_{\mathbf a})$ which has period $\le l_0^{2C_1}$ and $m_0=\max(h_0)=l_0^{2(g+1)C_1}$. Then, for each prime $2C_1<l\le l_0$ assign some $h_i(\mathbf a)=v_l'(f_{\mathbf a})$, which has period $\le l_0$. Then, we have $m_i=O(\log(l_0))$ where we crucially used that $1-C_1/l>1/2$ so that it is okay to take the log on the lower bound. In addition, recall that if $l>l_1$ then $v_l'(f_{\mathbf a})=1+O(l^{-1})$ so $m_i=O(l^{-1})$. All in all, this yields the constants $m,C,T\le l_0^{O(1)}$ and $\sum_{i=1}^m m_i=O(l_1\log(l_0))$ so we can choose $D=\Theta(\log(p)^{1-\frac\epsilon 2})$. Then for any $g$-dimensional box $B$ of side length $N$, Proposition \ref{prop: averaging out} gives
$$\frac{1}{N^g}\sum_{\mathbf a\in B\cap \Z^g}\prod_{l\le l_0} v_l'(f_{\mathbf a})=  \prod_{l\le l_0} z_l + O(\log (p)^{-1})$$
where $z_l$ is the average of $v_l'$. We can also apply this to a slice $B_t$ by fixing $a_1=t$ and let $a_2,\ldots, a_g$ vary in a $(g-1)$-dimensional box, to get
$$\frac{1}{N^{g-1}}\sum_{\mathbf a\in B_t\cap \Z^g}\prod_{l\le l_0} v_l'(f_{\mathbf a})=  \prod_{l\le l_0} z_{l,t} + O(\log (p)^{-1}),$$
where $z_{l,t}$ is the average of $v_l'$ subject to $a_1=t$. Combining the two tells us that
\begin{equation}\label{eqn: zl to zl,t}
\frac{1}{N}\sum_{t\in [t_0+1,t_0+N]}\prod_{l\le l_0} z_{l,t} = \prod_{l\le l_0}z_l +O(\log(p)^{-1}).
\end{equation}
\subsubsection{Part VI: Separating the archimedean and non-archimedean terms}
By Proposition \ref{prop: vinf and measure}, recall that $v_{\infty}(f_{\mathbf a})$ is the density function of $\mu_{\mathcal R_g}$ at $[p]^{-1}(\mathbf a)$. Since the density function is continuous on a compact set, it is uniformly continuous and thus within each slice $B_t$, $v_{\infty}(f_{\mathbf a})$ varies by at most $O(p^{-1/4})$. For all $B_t$ choose some vector $\mathbf a_{B_t}\in B_t$, then we decouple the archimedean and non-archimedean factors by applying the equation above:
\begin{equation}\label{eqn: sep arch and non-arch}
\begin{split}
\sum_{\mathbf a\in \mathcal F_t'} v_\infty(f_{\mathbf a})\prod_{l\le l_0} v_l'(f_{\mathbf a}) &= \sum_{B_t\in \mathcal B_t}(v_\infty(f_{\mathbf a_{B_t}})+O(p^{-1/4}))\sum_{\mathbf a\in B_t\cap \Z^g} \prod_{l\le l_0} v_l'(f_{\mathbf a})\\
&= \left(\sum_{B_t\in \mathcal B_t} N^{g-1}\left(v_\infty(f_{\mathbf a_{B_t}})+O(p^{-1/4})\right)\right)\left(\prod_{l\le l_0} z_{l,t} + O(\log(p)^{-1})\right).
\end{split}
\end{equation}
It is clear that we can approximate the first term using an integral
$$\int_{[p](\mathcal R_g)\cap\{t\le a_1<t+1\}}v_{\infty}(f_{\mathbf a})da_1\cdots da_g+O(p^{-1/4})\Vol([p](\mathcal R_g^+)),$$
and changing variables to $\mathbf b=[p]^{-1}(\mathbf a)$ and applying Proposition \ref{prop: vinf and measure} gives
$$p^{g(g+1)/2}\left(\int_{\mathcal R_g\cap \left\{\frac{t}{\sqrt p}\le a_1 < \frac{t+1}{\sqrt p}\right\}}\mu_{\mathcal R_g}+O(p^{-1/4})\right)=p^{(g(g+1)-1)/2}\left(\ST_g(t/\sqrt p)+O(p^{-1/4})\right).$$
\subsubsection{Part VII: Finding a common lower bound} In order to get to an $L^1$-bound between the two distributions $\#\mathcal A_g(\F_p,t)/\#\mathcal A_g(\F_p)$ and $\frac{1}{\sqrt p}\ST_g(t/\sqrt p)\prod_{l}v_l(t)$, we give a distribution that lower bounds both of them. Substituting Equation \eqref{eqn: sep arch and non-arch} into Equation \eqref{eqn: first lower bound} and dividing by $\#\mathcal A_g(\F_p)=(1+O(p^{-1/2}))p^{\dim(\mathcal A_g)}$, we get that $\#\mathcal A_g(\F_p,t)/\#\mathcal A_g(\F_p)\ge M(t)$ where $M(t)$ is given by
\begin{equation*}
\begin{split}
\frac{1-O(\log(p)^{-1+\epsilon})}{\sqrt p}  \left(\left(\ST_g(t/\sqrt p)-O(p^{-1/4})\right)\left(\prod_{l\le l_0} z_{l,t} - O(\log(p)^{-1})\right)-\frac{|\mathcal F'_t\setminus \mathcal F_t|}{p^{g(g+1)/2}}O(p^{1/6})\right).
\end{split}
\end{equation*}
We want to show that $M(t)$ also lower bounds $\frac{1}{\sqrt p}\ST_g(t/\sqrt p)\prod_{l}v_l(t)$, and for that we need to compare $z_{l,t}$ to $v_l(t)$. By definition of $v_l'$ and Proposition \ref{prop: vl lower bound}, we have $ v_l'(f_{\mathbf a})\le v_{l,k}(f_{\mathbf a})$, so $z_{l,t}\le v_{l,k}(t)$ as the latter is the average of $v_{l,k}(f_{\mathbf a})$ across $a_1=t$. Furthermore, by Proposition \ref{prop: vl(t) = 1 + O(l^-2)} and \ref{prop: vl,k(t) to vl(t)}, we have $v_l(t)=1+O(l^{-2})$ and $v_{l,k}(t)=v_l(t)(1+O(l^{-\max(2,k)}))$. Combining these, we have
\begin{equation*}
\begin{split}
\prod_{l\le l_0}z_{l,t}&\le \prod_{l}v_l(t)\cdot \prod_{l\le l_1}\left(1+O(l_0^{-1/2})\right)\prod_{l>l_1}(1+O(l^{-2}))=(1+O((\log p)^{-1+\epsilon}))\prod_{l}v_l(t).
\end{split}
\end{equation*}
Absorbing the error factor, we see that
$\frac{1}{\sqrt p}\ST_g(t/\sqrt p)\prod_{l}v_l(t)\ge M(t)$.
\subsubsection{Part VIII: Bounding losses from lower bound} Because $M(t)$ lower bounds $\#\mathcal A_g(\F_p,t)/\#\mathcal A_g(\F_p)$, the $L^1$-distance between these two distributions is simply the difference in both sums. It is clear that
$\sum_t \#\mathcal A_g(\F_p,t)/\#\mathcal A_g(\F_p) = 1$. Now we lower bound the sum over $M(t)$. First note that $$\sum_t \frac{|\mathcal F'_t\setminus \mathcal F_t|}{p^{g(g+1)/2}}=\frac{|\mathcal F'\setminus \mathcal F|}{p^{g(g+1)/2}}=O(p^{-1/5})$$
because from Part I of the proof there are at most $O(p^{-1/5})$ proportion of vectors which do not satisfy some given condition. Then, we repeat the method in Part IV, breaking the sum into intervals of length $N$, separating the archimedean and non-archimedean factors, changing the sum into an integral as in Part IV, and then applying Equation \eqref{eqn: zl to zl,t} to get
\begin{equation}\label{eqn: first averaging over t}
\begin{split}
&\sum_{|t|\le 2g\sqrt p}\left(\ST_g(t/\sqrt p)-O(p^{-1/4})\right)\left(\prod_{l\le l_0} z_{l,t} - O(\log(p)^{-1})\right) \\
&= \sqrt p\left(\int_{[-2g\sqrt p,2g\sqrt p]}\ST_g(t)dt-O(p^{-1/4})\right)\left(\prod_{l\le l_0} z_l+O(\log(p)^{-1})\right) .
\end{split}
\end{equation}
The integral is simply $1$ so we are left to lower bound $\prod_{l\le l_0}z_l$. The definition of $v_l'$ is such that all matrices $\gamma$ with $\ord_l(\disc(f_\gamma))\le k/C_3$ are included which is exactly the condition of Proposition \ref{prop: stabilization}, so
$$z_{l}\ge\frac{\#\{\gamma \in \GSp_{2g}(\Z/l^k\Z) \mid \mult(\gamma)=p, l^{k/C_3}\nmid\disc(\gamma)\}}{\#\{\gamma \in \GSp_{2g}(\Z/l^k\Z) \mid \mult(\gamma)=p\}}\ge 1-O(l^{-k\delta})$$
for some constant $\delta>0$, where we used Proposition \ref{prop: l^k not dividing disc lower bound} in the last inequality. 

We can also prove a uniform bound $z_l\ge 1-O(l^{-2})$ as follows. First, note that $l\neq p$ since our range is $l\le l_0 <p$, and we can assume $l\neq 2$ as we are only concerned with asymptotics. If $l\nmid \disc(f_{\mathbf a})$, then $l$ is a good prime for $f_{\mathbf a}$ and so $v_l'(f_{\mathbf a})=v_{l,k}(f_{\mathbf a})=1+O(l^{-1})$ by Lemma \ref{lem: good prime ratio of zeta functions}. Also note that if we just averaged $v_{l,k}(f_{\mathbf a})$ we would get $1$ by definition. Hence, the error comes when $l\mid \disc(f_{\mathbf a})$, and note that by Lemma \ref{lem: proportion ordp > g(g-1)}, there are $O(l^{-1})$ proportion of characteristic polynomials satisfying this. Furthermore, for each of them, by definition, $v_l'(f_{\mathbf a})\ge 1-C_1/l$, so it differs from $v_{l,k}(f_{\mathbf a})$ by at most $O(l^{-1})$, giving the desired result.

Hence, combining the two lower bounds for appropriate intervals, we have
$$\prod_{l\le l_0}z_l\ge \prod_{l\le l_1}(1-O(l_0^{-\delta/2}))\prod_{l_1<l\le l_0}(1-O(l^{-2}))=1-O((\log p)^{-1+\epsilon}),$$
and substituting our results back into the sum over $M(t)$ gives
$$\sum_{|t|\le 2g\sqrt p}M(t)\ge 1-O((\log p)^{-1+\epsilon}).$$
Hence, the $L^1$ distance between $M(t)$ and $\#\mathcal A_g(\F_p,t)/\#\mathcal A_g(\F_p)$ is $O((\log p)^{-1+\epsilon})$. 
\subsubsection{Part IX: Upper bounding sum of expected distribution} We are left with showing that the $L^1$ distance between $\frac{1}{\sqrt p}\ST_g(t/\sqrt p)\prod_{l}v_l(t)$ and $M(t)$ is small, so it suffices to upper bound the sum $\sum_t\frac{1}{\sqrt p}\ST_g(t/\sqrt p)\prod_{l}v_l(t)$. We basically repeat our method: truncate the infinite product, make each local factor periodic, apply Proposition \ref{prop: averaging out}, and split into archimedean and non-archimedean parts. However, this time it is a lot easier because of the concrete bounds given in Proposition \ref{prop: vl(t) = 1 + O(l^-2)} and \ref{prop: vl,k(t) to vl(t)}. It is pretty clear that because we start with better bounds for $v_l(t)$ then $v_l(f)$, the error term here would definitely be smaller than the error incurred in the previous part of the proof. Nevertheless, for completeness we give as best a bound as we can. Since this part is independent of the previous part of the proof, we choose to reuse some variables for notational simplicity. 

Take $l_0=p^{\epsilon}$ for a sufficiently small $\epsilon$, and let $k=k(l)$ be the largest integer such that $l^k\le p^{4\epsilon}$, which implies that $l^k>p^{2\epsilon}$. Applying Proposition \ref{prop: vl(t) = 1 + O(l^-2)} and \ref{prop: vl,k(t) to vl(t)}, we approximate the infinite product with a finite product of periodic local factors:
$$\prod_l v_l(t) = \prod_{l>l_0}(1+O(l^{-2}))\prod_{l\le l_0}(1+O(l^{-k}))v_{l,k}(t)=(1+O(p^{-\epsilon}))\prod_{l\le l_0}v_{l,k}(t).$$

Suppose $v_{l,k}(t)\ge 1/2$ when $l> C$ for some constant $C$, then apply Proposition \ref{prop: averaging out} with $h_0(t)=\prod_{l\le C}v_{l,k}(t)$, and assign $h_i(t)$ to each $v_{l,k}(t)$ with $C<l\le l_0$. We see that $m\le p^{\epsilon}$, $T=\max(l^k)\le p^{4\epsilon}$, $C=O(1)$, $D=O(1)+\sum_{l}l^{-2}=O(1)$. Picking $N=p^{1/2-\epsilon}$, if $\epsilon$ is sufficiently small then 
$$\frac{1}{N}\sum_{t\in [t_0+1,t_0+N]} \prod_{l\le l_0}v_{l,k}(t)=1+O(p^{-\epsilon})$$
where we used the fact that $v_{l,k}(t)$ average to $1$ by definition. 

Applying the method in Part VI with the above equation like in Equation \eqref{eqn: first averaging over t}, we have
$$\sum_t\frac{1}{\sqrt p}\ST_g(t/\sqrt p)\prod_{l\le l_0}v_{l,k}(t)=\left(\int_{[-2g\sqrt p,2g\sqrt p]} \ST_g(t)dt+O(p^{-\epsilon})\right)(1+O(p^{-\epsilon}))=1+O(p^{-\epsilon})$$
so combining this with the first equation simply gives 
$$\sum_t\frac{1}{\sqrt p}\ST_g(t/\sqrt p)\prod_{l}v_{l}(t)=1+O(p^{-\epsilon}).$$

In conclusion, summing both $L^1$ distances, we get an error of $O((\log p)^{-1+\epsilon})$. However, the computation here seems to imply that it could be strengthened to $O(p^{-\epsilon})$. The current barrier to improving our bound is that we are forced to pick $l_1=(\log p)^{1-\epsilon}$ because of Equation \eqref{eqn: upper bound for each f}. This in turn stems from having a loose upper bound of $v_l'(f_{\mathbf a})\le l_0^{g+1}$ because we were not able to prove an upper bound of $1+O(l^{-1})$ for $v_l(f)$ as in Conjecture \ref{conj: v_l = 1+O(l^-1)}.
\newpage
\appendix
\section{Stable orbital integrals as Gekeler ratios (by Jeff Achter and Julia Gordon)}\label{sec: appendix}

Let $G = \GSp_{2g}$, thought of as a smooth group scheme over
$\bb{Z}_l$, and suppose $\gamma_0 \in G(\bb{Z}_l)$ has generic
fiber which is regular and semisimple (as an element of $G(\rat_l)$). Below
we will recall the definition of a certain measure, the
geometric measure, on the stable orbit of $\gamma_0$, and use it to
define an orbital integral on the \emph{stable} orbit of $\gamma_0$ (it also defines the measure on the \emph{rational} orbit of $\gamma_0$, but this is not relevant for the present paper).

On one hand, as in \eqref{eqn: vl ppav definition}, define a renormalized orbital integral 
\[
\nu_l(\gamma_0) =
\frac{l^{d(\Sp_{2g})}}{\#\Sp_{2g}(\ff_l)}\mathcal O^{\geom}_{\gamma_0,l}(\mathbf{1}_{G(\bb{Z}_l)}),
\]
where for a smooth scheme $X \to \bb{Z}_l$, $d(X)$ denotes it
relative dimension.

On the other hand,  as in \eqref{eqn: vl t and f definition}, for a polynomial $f(T) \in \bb{Z}_l[T]$ with
nonzero discriminant, we define the stable Gekeler ratio as follows: 
\begin{definition} For a polynomial $f(T) \in \bb{Z}_l[T]$ with
nonzero discriminant, 
\begin{align*}
\nu_{l,n}(f) &=   \frac{\#\st{\gamma \in
    G(\bb{Z}_l/l^n) \mid f_\gamma \equiv f \bmod l^n}}{\#G(\bb{Z}_l/l^n) /
  (\#\aff^g(\bb{Z}_l/l^n) \#\mathbb G_m(\bb{Z}_l/l^n))} \\
\nu_l(f) &= \lim_{n\to\infty} \nu_{l,n}(f).
\end{align*}
\end{definition} 

Let $f_{\gamma_0}$ be the characteristic polynomial of $\gamma_0$.
The goal of this appendix is to give an essentially self-contained proof of Proposition \ref{L:apptarget}:

\begin{proposition}
  \label{L:apptarget}
  For $\gamma_0 \in G(\bb{Z}_l)$, the stable orbital integral with respect to the geometric measure 
  (the RHS of the equality below) can be expressed as a limit of Gekeler ratios: 
  \[
  \nu_l(f_{\gamma_0}) =   \sum_{\gamma \in S_{\gamma_0}} \nu_l(\gamma),
  \]
  where $\gamma$ ranges over a set $S_{\gamma_0}$ of representatives for the
  rational conjugacy classes inside the stable conjugacy class of
  $\gamma_0$.
\end{proposition}

We also prove a similar result (with a different test function) for $l=p$ in \S \ref{sub:l_is_p} below. 
\begin{remark}
  In \cite{Achter_2023}, the authors purported to give an explicit formula for
 each rational orbital integral $\nu_l(\gamma_{[A,\lambda]})$ in terms of matrix counts; it needs a modification. 
 A detailed
 explanation and erratum will appear elsewhere.
\end{remark}

\begin{remark} While we prove this lemma for the purposes of the present paper only for $G=\GSp_{2g}$, the argument works much more generally.  A similar statement in full generality will appear elsewhere. 
\end{remark}

\subsection{Measures}

\subsubsection{Counting measure}

Let $(R,\mathfrak m)$ be a discrete valuation ring with fraction field
$K$, with residue field $\kappa$ of finite cardinality $l$, and with normalized absolute value $\abs{\cdot} = l^{-\val_p(\cdot)}$. Let
$R_n = R/\mathfrak m^n$, and let $\pi_n: R \to R_n$ be the truncation
map.

For any $R$-scheme $X$, we denote by $\pi_n^X$ the corresponding map
$\pi_n^X: X(R) \to X(R_n)$ induced by $\pi_n$. If $X$ is flat (in
particular, if $X$ is smooth), we let $d(X)$ be the relative
dimension of $X$ over $R$.

Once and for all, fix the Haar measure $dx$ on $\A^1(K)$ such that the
volume of $\aff^1(R) = R$ is $1$. 
With this normalization, the fibres of the standard projection 
$\pi_n^{\A^d}:\A^d(R)\to \A^d(R_n)$ have volume $l^{-nd}$. 

Now suppose $X$ is smooth over $R$, with a non-vanishing top degree
differential form $\omega_X$.  Integrating $\omega_X$ defines a Borel
measure $\abs{d\omega_X}$ on $X(R)$ (well-defined since, if $\omega_X$
exists, it is unique up to multiplication by a unit), and
$$
\vol_{\abs{d\omega_X}}(X(R))=\frac{\#X_\kappa(\kappa)}{l^{d(X)}}
$$
\cite{weil:adeles}.
Indeed, if $X$ is smooth of relative
dimension $d$, then $X(K)$ is an analytic manifold, locally isomorphic
to $\aff^d(K)$; the Jacobian of any such isomorphism is a non-vanishing top degree differential, and thus has to agree with $\omega_X$ up to a unit. Therefore, by the Jacobian transformation rule, $\abs{d\omega_X}$  is
the pullback of the Haar measure defined by $\omega_{\aff^d}$ as above, under any such local isomorphism. 
\begin{lemma}
  \label{L:subofsmooth}
  If $X$ is smooth over $R$ and if $A\subset X(R_n)$, then
  \[
  \vol_{\abs{d\omega_X}} (\pi_n^X)\inv(A) = \frac{\#A}{l^{nd(X)}}.
\]
\end{lemma}

\begin{proof}
  It suffices to prove this in the case where $A$ consists of a single
  point.  Since (as reviewed above) $X$ is analytically locally isomorphic to
  $\aff^{d(X)}$, the result follows from the corresponding calculation on affine
  space.
\end{proof}

\subsubsection{Geometric measure}

Continue to let $G = \gsp_{2g}$. As in \cite[\S 2.3]{Achter_2023}, let $\aff_G = \aff^g
\times \gp_m$ be the space of characteristic
polynomials of elements of $G$, and let $\mathfrak c:G \to \aff_G$ be the morphism which
takes an element to the coefficients of its characteristic polynomial, except the last coordinate is the multiplier $\eta$ rather than the determinant $\det$. 
 For $\gamma_0 \in G(R)$, let
\[
Y_{\gamma_0}  
=\mathfrak c\inv(\mathfrak c(\gamma_0)).
\]
Then for any $R$-algebra $\alpha:R \to S$, we have
\begin{equation}
\label{E:ygammapoints}
Y_{\gamma_0}(S) = \st{ \gamma \in G(S) : \alpha(f_\gamma) = \alpha(f_{\gamma_0})}.
\end{equation}
  Now suppose $\gamma_0$ has regular semisimple generic fibre in $G(K)$.  Since the stable orbit of a regular semisimple element of $G(K)$ is determined by its characteristic polynomial, we have
  \[
Y_{\gamma_0}(K) = \st{ \gamma \in G(K) |  f_\gamma = f_{\gamma_0}}= \st{ \gamma \in G(K) \mid \exists g \in G(\bar K): g\inv \gamma g = \gamma_0}.
\]

On $\aff_G$, let $\omega_{\aff_G} = dx_1 \wedge \dots \wedge dx_g
\times \frac{d\eta}{\eta}$. Now, $\omega_G$ is a generator of the top exterior power of the
cotangent bundle of $G$.  
For each $a\in \A_G(K)$, there is a unique generator
$\omega^\geom_{a}$ of the top exterior power of
the cotangent bundle of the fibre $\fc^{-1}(a)$ such that
\[
\omega_G = \omega^\geom_{a} \wedge
\omega_{\aff_G}.
\]
Since the stable orbit of $\gamma_0$ coincides with the fibre $\fc^{-1}(\fc(\gamma_0))$, we obtain a measure
$\mu^\geom_{\gamma_0} = \mu^\geom_{\mathfrak c(\gamma_0)}$ associated to
$\omega^\geom_{\mathfrak c(\gamma_0)}$ on $Y_{\gamma_0}(K)$.
 
Then for any $\phi\in C_c^\infty(G(K))$,
$$
\int_{G(K)} \phi(g)\, \abs{d\omega_G}=\int_{\A_G(K)}\int_{\fc^{-1}(a)}\phi(g)\,d\mu^\geom_{a}
\,\abs{d\omega_{\A_G}(a)}.$$
This measure also appears in \cite{langlands-frenkel-ngo}.

The stable orbital integral of $\phi$ with respect to the geometric measure is then 
\begin{equation*}
\mathcal O^{\geom, \stab}_{\gamma_0}(\phi) 
=
\int_{Y_{\gamma_0}(K)} \phi(g) d\mu^{\geom}_{\gamma_0}.
\end{equation*}

Let $S_{\gamma_0}$ be a set of representatives for the conjugacy classes inside the stable conjugacy class of $\gamma_0$.  Then we may equivalently write
\[
\mathcal O^{\geom,\stab}_{\gamma_0}(\phi) = \sum_{\gamma \in S_{\gamma_0}} \int_{O(\gamma)} \phi(g\inv \gamma g )d\mu^{\geom}_{\gamma_0},
\]
where $O(\gamma)$ stands for the rational orbit of $\gamma$. The stable orbit of $\gamma_0$ is the disjoint union of $O(\gamma)$ as $\gamma$ runs over $S_{\gamma_0}$; and each orbit $O(\gamma)$ is open in the stable orbit;  and  $d\mu^{\geom}_{\gamma_0}$ restricts to an invariant measure on each $O(\gamma)$ for $\gamma\in S_{\gamma_0}$.

\subsection{Gekeler numbers and volumes, for $l$ not equal to $p$}\label{sub:gek_p_not_l}
We specialize to the case $R = \bb{Z}_l$.  Let $\gamma_0 \in G(\Z_l)$ be regular semisimple, with characteristic polynomial $f_{\gamma_0}(T)$.
For each positive integer $n$, consider the subset $V_n(\gamma_0)$ of $G(\Z_l)$ defined as 
\begin{align}\label{eq:subsets}
V_n= V_n(\gamma_0)&\coloneqq\st{\gamma\in G(\Z_l)\mid f_\gamma(T)\equiv f_{\gamma_0}(T) \bmod  l^n} \nonumber\\
                  &= \st{ \gamma \in G(\bb{Z}_l) \mid \pi_n^{\A_G}(\fc(\gamma)) = \pi_n^{\A_G}(\fc(\gamma_0))}\nonumber\\
  &= (\pi_n^G)\inv (Y_{\gamma_0}(R_n))\text{ (by \eqref{E:ygammapoints})}
\intertext{and set}
V(\gamma_0) &\coloneqq \cap_{ n \ge 1} V_n(\gamma_0).\nonumber
\end{align}

We define an auxiliary ratio:
\begin{equation}\label{eq:vn}
v_n(\gamma_0) \coloneqq \frac{\vol_{\abs{d\omega_G}}(V_n(\gamma_0))}{l^{-(g+1)n}}.
\end{equation}

Now we would like to  relate the limit of these 
ratios $v_n(\gamma_0)$ both to the limit of Gekeler  ratios $\nu_{l,n}(\gamma_0)$ 
and to a stable orbital integral.

Let $\phi_0 = \mathds{1}_{G(\bb{Z}_l)}$ be the characteristic function of the maximal compact subgroup
$G(\Z_l)$ in $G(\Q_l)$. 

\begin{proposition}\label{prop:st.oi}
We have   
\[
\lim_{n\to \infty} v_n(\gamma_0)= 
\mathcal O_{\gamma_0}^{\geom, \stab}(\phi_0).
\]
\end{proposition}

\begin{proof}
Because for regular semisimple elements equality of characteristic polynomials is equivalent to stable conjugacy in $G(\rat_l)$, $V(\gamma_0)$ is the intersection of $G(\bb{Z}_l)$ with the stable orbit $\cO^\stab(\gamma_0)$ 
of $\gamma_0$ in $G(\Q_l)$. 
Then the stable orbital integral
$\mathcal O_{\gamma_0}^{\geom, \stab}(\phi_0)$
is nothing but the volume of the set $V(\gamma_0)$, as a subset of $\cO^\stab(\gamma_0) = Y_{\gamma_0}(\rat_l)$, with respect to the 
measure $d\mu^{\geom}_{\gamma_0}$.

Let $a_0=\fc(\gamma_0)= (\underline a_0,q)\in (\A^g\times \Gm)(\Q_l)$, and let $U_n(a_0)$ be its neighbourhood of radius $l^{-n}$ in every coordinate.  Its volume (with respect to the usual measure on the affine space) is $\vol_{\mu_{\A_G}}(U_n(\gamma_0)) = 
l^{-(g+1)n}$. (Since $l\nmid q$, $\abs{\eta(\gamma_0)} = \abs{q} = 1$; and so if $\pi_1(\gamma) = \pi_1(\gamma_0)$, then $\abs{\eta(\gamma)} = 1$. Therefore, for such $\gamma$, the factor $\frac{1}{|\eta(\gamma)|}$ in the definition of $\mu_{\A_G}$ is just $1$.)

Moreover, by definition, $V_n(\gamma_0) = \fc\inv(U_n(\gamma_0)) \cap G(\bb{Z}_l)$.  Consequently,
\begin{equation}\label{eq:st.volume}
\lim_{n\to \infty}v_n(\gamma_0)
=\lim_{n\to \infty}\frac{\vol_{|d\omega_G|}(\fc^{-1}(U_n(\gamma_0)) \cap G(\Z_l))}{\vol_{|d\omega_{\aff_G}|}(U_n(\gamma_0))}
=\vol_{\mu^{\geom}_{\gamma_0}}(V(\gamma_0)),
\end{equation} 
by  definition of the geometric measure.  
\end{proof}

\begin{lemma}
We have the following relation between these auxiliary ratios and our stable Gekeler ratios:
\[
v_n(\gamma_0) =
\frac{\#\Sp_{2g}(\bb{Z}_l/l^n)}{l^{d(\Sp_{2g})}}
                 \nu_{l,n}(f_{\gamma_0}).
\]
\end{lemma}

\begin{proof}
Continue to let $Y_{\gamma_0} = \mathfrak c\inv(\mathfrak c(\gamma_0)) \subset G$ be the stable orbit of $\gamma_0$. Recall \eqref{eq:subsets} that $V_n(\gamma_0) = (\pi_n^G)^{-1} (Y_{\gamma_0}(\Z_l/l^n))$.  
Then
\begin{align*}
v_n(\gamma_0) &= \frac{\vol_{\abs{d\omega_G}}(V_n(\gamma_0))}{l^{-(g+1)n}} \\
              &= \frac{\#Y_{\gamma_0}(\bb{Z}_l/l^n)}{l^{d(G)n} \cdot l^{-(g+1)n} }\quad\text{(Lemma \ref{L:subofsmooth})}\\
              &= \frac{\#Y_{\gamma_0}(\bb{Z}_l/l^n)/(\#G(\bb{Z}_l/l^n)/\#\aff_G(\bb{Z}_l/l^n))}{l^{d(G)n} \cdot l^{-(g+1)n}/(\#G(\bb{Z}_l/l^n)/\#\aff_G(\bb{Z}_l/l^n)) }\\
              &=\frac{\#\Sp_{2g}(\bb{Z}_l/l^n)\#\gp_m(\integ/l^n)}{l^{(d(G)-1-g)n}\#\aff^g(\bb{Z}_l/l^n)\#\gp_m(\integ/l^n)}
                \frac{\#Y_{\gamma_0}(\bb{Z}_l/l^n)}{\#G(\bb{Z}_l/l^n)/\#\aff_G(\bb{Z}_l/l^n)}\\
  &=\frac{\#\Sp_{2g}(\bb{Z}_l/l^n)}{l^{d(\Sp_{2g})}}
                \frac{\#Y_{\gamma_0}(\bb{Z}_l/l^n)}{\#G(\bb{Z}_l/l^n)/\#\aff_G(\bb{Z}_l/l^n)}.
\end{align*}
\end{proof}

We finally come to the point of this appendix in the $l\neq p$ case.

\begin{proof}[Proof of Proposition \ref{L:apptarget}]
  We now compute (with $O(\gamma)$ denoting the rational orbit of $\gamma$, which is an open subset of $O^\stab(\gamma_0)$ and thus inherits the geometric measure): 
  \begin{align*}
    \sum_{\gamma \in S_{\gamma_0}} \nu_l(\gamma) &=
    \sum_{\gamma \in S_{\gamma_0}} \frac{l^{d(\Sp_{2g})}}{\#\Sp_{2g}(\ff_l)} \int_{O(\gamma)} \phi_0(g\inv \gamma g )d\mu^{\geom}_{\gamma_0} \\
    &= \frac{l^{d(\Sp_{2g})}}{\#\Sp_{2g}(\ff_l)} \mathcal{O}^{\geom,\stab}_{\gamma_0}(\phi_0) \\
    &= \frac{l^{d(\Sp_{2g})}}{\#\Sp_{2g}(\ff_l)} \lim_{n\to\infty} \nu_n(\gamma_0) \\
    &= \lim_{n\to\infty} \nu_{l,n}(f_{\gamma_0}).
  \end{align*}
\end{proof}

\subsection{The case $l=p$}\label{sub:l_is_p}
In all the arguments above, the support of the test function of interest was contained in $G(\Z_p)$, which made it possible to do all the calculations with the Haar measure on $G$. When $l=p$, we are dealing with an integral matrix that does not have an integral inverse; thus we need to create a geometric framework for such matrices. As far as we understand, there is no canonical $\Z_p$-scheme on which to do this calculation. 

We make one choice of such a scheme, suitable for the present purposes but likely not optimal, namely, we define 
a $\Z$-scheme $\cM$ such that  
 $\cM(\Z_l)= \gsp_{2g}(\Q_l)\cap \mat(\Z_l)$. (This construction works for all primes $l$, though it is needed here only for $l=p$.)  

Over $\integ$, we have the scheme $\mat_{2g}$ of $2g\times 2g$
matrices with coefficients in $\integ$.

Let $J \in \mat_{2g}(\integ)$ be the block diagonal matrix $\begin{pmatrix}0 & I_g \\ -I_g & 0\end{pmatrix}$.

Let $\cM \subset \mat_{2g} \times \aff^1$ be the closed subscheme
which, on $S$-points, is given by
\[
\cM(S) = \st{ (A,m) \in \mat_{2g}(S)\times \aff^1(S) :  A^T J A
  = m \cdot J}.
\]
(The point is that $\cM$ is defined by equations in $m$ and the
entries of $A$.) If $(A,m) \in \mathcal M(S)$, then $m$ is completely determined by $A$.  Consequently, we often abuse notation and view an element of $\mathcal M(S)$ as an element of $\mat_{2g}(S)$.

The open subscheme $\mathcal M^* \coloneqq \mathcal M \times_{\mat_{2g}\times \mathbb A^1} (\mat_{2g}\times \mathbb G_m)$ is isomorphic to $G$, and in particular is smooth. Therefore, the singular locus of $\cM$ is supported in the locus $m=0$. 

The mistake in the original paper was in requiring the existence of a lift of a solution to the congruence condition on characteristic polynomials $\mod p^n $ to a point in $\cM(\Z_p)$ congruent to some conjugate of $\gamma_0$ $\mod p^n$. That condition is too stringent; the only lifting condition we should require is that an approximate solution lift to a symplectic similitude matrix.  
Now we can define the (`stable') Gekeler ratio at $l=p$ as well. This definition extends the definition given 
above for $l\neq p$, since the condition of lifting to $\cM$ is automatically satisfied when $l$ does not divide the determinant of $\gamma$.

\begin{definition}\label{def:gekeler}
For each rational
prime $l$ , let
\begin{equation}\label{eqgekelerterms}
\nu_l(f_{\gamma_0}) = \lim_{n\ra\infty}
\frac{\#\st{\gamma \in
   \cM(\Z_l/l^n): f_\gamma \equiv f_{\gamma_0} \bmod l^n \text{\ and there exists \ } \tilde\gamma\in \cM(\Z_l): \tilde \gamma\equiv \gamma \mod l^n } }{\#\gsp_{2g}(\integ/l^n)/\#\A_G(\Z_l/l^n) }.
\end{equation}
\end{definition}

This definition initially seems very inconvenient for computation, since it refers to $l$-adic points.
Fortunately, the `lifting to a symplectic similitude' condition is satisfied automatically when $\gamma_0$ comes from an \emph{ordinary} $X$, and can be checked at a finite level for general $X$ (see Lemma \ref{L:liftability} below). 

Thus, for $X$ ordinary, \ref{def:gekeler} simplifies to the following formulation.
\begin{definition}\label{def:gekeler}
If $X$ is ordinary, for each rational
prime $l$ , let
\begin{equation}\label{eqgekelerterms}
\nu_l(f_{\gamma_0}) = \lim_{n\ra\infty}
\frac{\#\st{\gamma \in
   \cM(\Z_l/l^n): f_\gamma \equiv f_{\gamma_0} \bmod l^n }}{\#\gsp_{2g}(\integ/l^n)/\#\A_G(\Z_l/l^n) }.
\end{equation}
\end{definition}

We recall the setup of \cite{Achter_2023}.  We start with $q = p^e$ and fix a principally polarized abelian
variety $(X,\lambda)/\ff_q$ where $1/2$ is not a slope of the
$p$-divisible group $X[p^\infty]$. From this, we derived data
$\delta_0 \in G(\rat_q)$ and $\gamma_0 \in G(\rat_p)$.  Let $R = \bb{Z}_p$ and let $f(T) =
f_{\gamma_0}(T)\in R[T]$ be the characteristic polynomial of $\gamma_0$. Let $s = s_X$ be the largest slope of $X[p^\infty]$ which is less than $1/2$.  (Note that $X$ is ordinary if and only if $s_X = 0$.)  The following factorization exists precisely because $X[p^\infty]$ is assumed to have no component of slope $1/2$.

\begin{lemma}
  There is a factorization $f(T) = h(T) h^+(T)$ over $R[T]$, where $h$
  is a polynomial of degree $g$ all of whose roots $\alpha_i$ have $p$-adic
  valuation $e_i := \ord_p(\alpha_i)$ satisfying $0 \le e_i \le s_X e< \half e$, and where $h^+(T)= T^g h(q/T)$ (and thus all
  roots of $h^+$ have $p$-adic valuation $\half e < (1-s_X)e \le \ord_p(q/\alpha_i)
  \le e$).
\end{lemma}

\begin{proof}
  Because $\gamma_0 = N_{\rat_q/\rat_p}(\delta_0)$, the Newton polygon
  of the polynomial $f_{\gamma_0}(T)$ is equal to the Newton polygon
  of the $p$-divisible group $X[p^\infty]$ scaled by $e$. Let
  $\alpha_1, \dots, \alpha_g$ be the roots of $f_{\gamma_0}(T)$ in
  $\bar\rat_p$ with $p$-adic valuation at most $e s_X$.  Because
  $\gal(\rat_p)$ preserves the valuation on $\bar\rat_p$, $\st{\alpha_1,
    \dots, \alpha_g}$ is stable under $\gal(\rat_p)$, and thus $h(T)
  \coloneqq \prod_{1\le i \le g}(T-\alpha_i)$ and $h^*(T):=
  \prod(T-q/\alpha_i)$ lie in $\rat_p[T]$; by Gauss's lemma, they are
  elements of $R[T]$.
\end{proof}

In the following statement, note that $e^* \le \lfloor 2s_Xe\rfloor \le e$; and that if $X$ is ordinary,  then $e^*=0$ and the liftability hypothesis is vacuous.

\begin{lemma}
\label{L:liftability}
  Suppose $n>e$ and $\gamma \in \mathcal M(R_n)$ satisfies
  \[
  f_{\gamma}(T) = \pi_n(f_{\gamma_0}(T)). 
  \]
  Let $e^* = \max_{i,j}(\min(e,e_i+e_j))$.
 If $\gamma$ lifts to
  $\mathcal M(R_{n+e^*})$, then it lifts to $\mathcal M(R)$.
\end{lemma}

\begin{proof}

  Let $V_n = V\tensor R_n$.  The multiplier $\mu(\gamma)$ can be read
  off from $f_{\gamma}(T)$, and thus coincides with
  $\pi_n(\mu(\gamma_0))$.  Consequently, for all $x$ and $y$ in $V_n$
  we have\[
  \langle \gamma x, \gamma y \rangle = q\langle x,y \rangle.
  \]
  
  The $R$-factorization of $f_{\gamma_0}(T)$ induces a factorization
  \[
  f_{\gamma}(T) = h(T) h^+(T)
  \]
  in $R_n[T]$.

  Let $V_n = V\tensor R_n$, and let $U$ be the maximal saturated
  submodule of $\ker h(\gamma)$.
     Then $U$ is a $\gamma$-stable summand of $V_n$.  (One way to see this is by constructing $U$ as follows.  Let $\tilde{\gamma}$ be a lift of $\gamma$ to $\mat_{2g}(R)$ with characteristic polynomial $f_{\gamma_0}(T)$ -- this is possible because $\mat_{2g}$, unlike $\mathcal M$, is smooth.  Let $\tilde{U} = \ker h(\tilde{\gamma})$, and and then let $U= \tilde{U} \otimes_R R_n$.) 

  We verify that $U$ is isotropic.  It suffices to see this after an
  extension of scalars.  Consequently, possibly after replacing $R$
  with the integral closure of $\integ_p$ in the splitting field of
  $f_{\gamma_0}(T)$, we may and do assume that $h(T)$ has a
  (nonunique!) factorization $h(T)= \prod_{i=1}^g (T-\alpha_i)$, where
  $\ord_p(\alpha_i) \le e s_X$.  Use this to produce a basis for $U$
  in which $\gamma$ is upper-triangular.  Specifically, choose an 
  $R_n$-basis $x_1, \dots, x_g$ for $U$ such that
  \[
  \gamma(x_j) = \alpha_j x_j + \sum_{i < j}a_{ij}x_i.
  \]
On one hand, we have
\begin{align*}
  \langle \gamma x_1, \gamma x_2 \rangle & = q \langle x_1,x_2 \rangle = p^e \langle x_1,x_2 \rangle.
  \intertext{On the other hand, we have}
  \langle \gamma x_1, \gamma x_2 \rangle &= \langle \alpha_1 x_1, \alpha_2 x_2 +
    a_{12}x_1 \rangle \\
  &= \alpha_1\alpha_2 \langle x_1,x_2 \rangle
  \intertext{and so}
  p^{e_{12}^*} \langle x_1,x_2 \rangle &= 0\text{ in } R_n
  \intertext{ where $e_{12}^* = \min(e, e_1+e_2) \le e^*$, and we see that}
  \langle x_1,x_2 \rangle &\in p^{n-e_{12}^*} R_n.
\end{align*}
Suppose $\gamma$ admits a lift $\tilde{\gamma} \in \mathcal
M(R_{n+e^*})$; let $\tilde{x_i}$ be a compatible choice of lift of
$x_i$.  Then the same calculation shows that
\begin{align*}
  p^{e_{12}^*} \langle \tilde{x_1},\tilde{x_2} \rangle & =0 \text{ in } R_{n+e^*}
  \intertext{and thus}
  \langle \tilde{x_1}, \tilde{x_2} \rangle  &\in p^{n+e^*-e_{12}^*}R_{n+e^*};\\
  \pi_n(\langle \tilde{x_1}, \tilde{x_2} \rangle) &= \langle x_1,x_2 \rangle =0 \text{ in } R_n.
\end{align*}
Inductively, the same argument shows that $\langle x_i,x_j \rangle = 0$ for all
$1 \le i \le j \le g$.  We conclude that  $U$ is a maximal isotropic
summand of $V_n$.

Similarly, let $W$ be the maximal saturated submodule of
$\ker(h^*(\gamma))$.  It, too, is a $\gamma$-stable maximal isotropic
subspace.  Moreover, because $h$ and $h^*$ share no common roots in
$R_n$, $W$ is isomorphic, via $\langle \cdot,\cdot \rangle$, to the dual of $U$.

Choose an $R_n$-basis for $U$, and choose the dual basis for $W$. Then
$\gamma \vert_U$ is specified by a matrix $A \in \mat_g(R_n)$;
$\gamma\vert_{W}$ is specified by a matrix $D \in \mat_g(R_n)$; and
because $\gamma \in G(R_n)$, we have $A^T D = q I_g$.  Because the symplectic group acts
transitively on symplectic bases for $V$, $\gamma$ is
$G(R_n)$-conjugate to the block-diagonal matrix $\beta := \begin{pmatrix} A&0\\0&D\end{pmatrix}$.

Let $\tilde{U}$  and $\tilde{W}$ be lifts of $U$ and $W$, respectively, to
maximal isotropic summands of $V$.  Let $\tilde{A}$ be a lift of $A$ to
$\End(\tilde{U}) \iso \mat_g(R)$, and choose $\tilde{D}$ so that
$\tilde{A}^T \tilde{D} = qI_g$; then each of $\tilde{A}$ and $\tilde{D}$
lies in $\mat_g(R) \cap \gl_g(\rat_p)$.  Finally, let $\tilde{\beta}$ be the block-diagonal matrix $\begin{pmatrix}\tilde{A} & 0 \\ 0 & \tilde{D} \end{pmatrix}$; then $\tilde{\beta} \in \mathcal M(R)$ is $G(R)$-conjugate to a lift of $\gamma$.

\end{proof}

\subsubsection{Finishing the calculation at $l=p$}
With the new definitions of the Gekeler ratio, we can now complete the proof of the equality between the Gekeler ratio and the orbital integral in the case $l=p$, along the lines of \cite[Section 4.3]{Achter_2023}. 
Lemma 3.7 in \cite{Achter_2023} holds, and almost everything in Section 4.3 there holds without change as well: Lemmas 4.8, 4.9 do not mention the ratios at all, and Lemma 4.10 is a statement about geometric orbital integrals, which does not involve the erroneously defined ratios. Lemma 4.11 concerns the Serre-Oesterl\'e measure on $\gsp_{2g}$, which coincides with the measure $|d\omega_G|$,  and also is not affected by the error. Only Corollary 4.12 uses the incorrect definition of the ratio $\nu_p$, and thus needs to change. However, its proof is correct as long as we use the correct definition of $\nu_p$. We include the corrected version for completeness. 
\begin{proposition} 
 Suppose that either $X$ is ordinary or that $q = p$.
For $l =p$,
the Gekeler ratio \eqref{eqgekelerterms} is related to the geometric orbital integral by 
\[
\nu_p([X,\lambda]) = q^{-\frac{g(g+1)}2}\frac{p^{\dim(G^\der)}}{\#G^\der(\integ_p/p)}O^{\geom, \stab}_{\gamma_0}(\phi_{q,p}), \]
where $\phi_{q,p}$ is the characteristic function of  the same  double coset as in \cite[\S 4C]{Achter_2023}. 
\end{proposition}
\begin{proof} 
First observe that $\vol_{\abs{d\omega_A}}(U_n(\gamma_0))=q p^{-n\mathrm{rank}(G)}=q p^{-(g+1)n}$, since  we are using the invariant measure on the $\mathbb G_m$-factor of $\A_G=\A^{\mathrm{rank}(G)-1}\times \mathbb G_m$, and for $\gamma_0$ (and therefore, for all points in $U_n$), the $\mathbb G_m$-coordinate is the multiplier, with absolute value $q^{-1}$.  
To give a proof which works uniformly in the cases $X$ ordinary and $q=p$, we maintain the condition on lifting (and just remember that this condition is vacuous in the ordinary case). 
Thus, by Definition \ref{def:gekeler} 
we can rewrite  $\nu_p([X, \lambda])$ (with $G=\gsp_{2g}$) as  
\begin{equation}\label{eq:serre_at_p}
\begin{aligned}
&\nu_p([X, \lambda]) =
\lim_{n\to \infty}
\frac{\#\st{\gamma \in \cM(\Z_p/p^n): f_\gamma \equiv f_{\gamma_0} \bmod p^n, \exists \tilde\gamma\in \cM(\Z_p): \tilde \gamma\equiv \gamma \mod p^n }} {\#\gsp_{2g}(\integ/p^n)/\#\A_G(\Z_p/p^n)}\\
&= \frac{p^{\dim(G^\der)}}{\#G^\der(\integ_p/p)}\cdot \\
&\quad\quad
\lim_{n\to \infty}\frac{\#\st{\gamma \in \cM(\Z_p/p^n): f_\gamma \equiv f_{\gamma_0} \bmod p^n, \ \exists \tilde\gamma\in \cM(\Z_p): \tilde \gamma\equiv \gamma \mod p^n } p^{-n(\dim(G)-\rk(G))}}{\vol_{|d\omega_A|}(U_n(\gamma_0)) q^{-1}}, 
\end{aligned}
\end{equation}
whereas the geometric orbital integral is, by definition, 
\begin{equation}\label{eq:volume_ratio}
\vol_{\mu^{\geom}_{\gamma_0}}(V(\gamma_0)) = 
\lim_{n\to \infty}\frac{\vol_{|d\omega_G|}(\fc^{-1}(U_n(\gamma_0)) \cap \cM(\Z_p))}{\vol_{|d\omega_A|}(U_n(\gamma_0))}.
\end{equation} 

Thus all we need to do now is compare the numerators of (\ref{eq:serre_at_p}) and (\ref{eq:volume_ratio}). We present a somewhat ad hoc argument on point counts that follows the lines of the argument in \cite[Lemma 4.11]{Achter_2023}. 
We recall that the Cartan decomposition for $G$ gives us $\cM(\Z_p)=\sqcup_{\lambda\in A^+}D_{\lambda, p}$, where $A^+$ is the set of dominant co-characters, and $D_{\lambda,p}$ is the double coset $G(\Z_p) \lambda(p) G(\Z_p)$.  
Instead of Lemma 4.11 of that paper, we prove the following.

\begin{lemma}\label{lemma:main} 
For $n$ sufficiently large:
\begin{enumerate} 

\item The neighbourhood of the identity in $\cM(\Z_p)$ of valuative radius $n$ is contained in $G(\Z_p)$.

\item If $x\in \cM(\Z_p/p^n)$, then all lifts of $x$ to $\cM(\Z_p)$ lie in the same double coset of $G(\Q_p)$.

\item If $x, x'\in \cM(\Z_p/p^n)$ lift to elements of the same double coset of the Cartan decomposition for $G(\Q_p)$, then the fibres of the map $\pi_{\cM, n}$ over $x$ and $x'$ have the same volume with respect to the Haar measure $|d\omega_G|$ on $G(\Q_p)$.

\item 
For each $x\in \cM(\Z/p^n)\cap \pi_{\cM, n}(D_{e\mu_0,p})$, the fibre  $\pi_{\cM, n}^{-1}(x)$  has volume $q^{-\frac{g(g+1)}2+1}p^{-n\dim(G)}$ with respect to the measure $|d\omega_G|$.   
\end{enumerate} 
\end{lemma} 

\begin{proof} 
1. A matrix sufficiently close to the identity has determinant $1$; it is in $G(\Q_p)$ by definition of 
$\cM(\Z_p)$, and is integrally invertible since the determinant is a unit.

2.  Let $\gamma, \gamma'$ be arbitrary lifts of  $x$ to $\cM(\Z_p)$.  
We have $\gamma'-\gamma$ is of the form $\gamma'-\gamma = p^n X$, where $X\in \mat_{2g}(\Z_p)$.
We can write $\gamma^{-1}\in G(\Q_p)$ as $\gamma^{-1} = p^{-\val_p(\det \gamma)}A$, where $A\in \cM(\Z_p)$ is a matrix of minors of $\gamma$.  Indeed,  the entries of $A$ are minors of an integral matrix, thus integral; and since $\gamma^{-1}$ is a symplectic similitude, $A$ is also a symplectic similitude, and therefore $A\in \cM(\Z_p)$.  
Multiplying by $\gamma^{-1}$, we get: $\gamma^{-1} \gamma' -I = p^{n-\val_p(\det(\gamma))}AX$. The matrix $AX$ is integral; thus  for $n\gg \val_p(\det(\gamma))$, the matrix $\gamma^{-1} \gamma'$ is integral and congruent to the identity modulo at least $p^{n-val_p(\det \gamma)}$.  Therefore, $\gamma^{-1} \gamma' \in G(\Z_p)$ by the previous item.

3. 
First, by assumption,  the fibres of $\pi_{\cM, n}$ over $x$ and $x'$  are non-empty, 
and then both fibres are contained in the same double coset of the Cartan decomposition for $G(\Q_p)$ by the previous point. Let $\gamma \in \pi_{\cM, n}^{-1}(x)$,  $\gamma' \in \pi_{\cM, n}^{-1}(x')$. 
Then by assumption, there exist elements $k_1, k_2\in G(\Z_p)$ such that $\gamma'=k_1\gamma k_2$. 
Then if $\gamma''\equiv \gamma\mod p^n$, we have $k_1 \gamma'' k_2 \equiv \gamma' \mod p^n$, i.e. 
$$ k_1 \pi_{\cM, n}^{-1}(\gamma) k_2 \subset \pi_{\cM, n}^{-1}(\gamma').$$
Since the roles of $\gamma$ and $\gamma'$ in this argument could have been switched, we see that  
the fibre over $x'$ is a translate of the fibre over $x$ by $\gamma^{-1}\gamma'$, and hence has the same volume with respect to the Haar measure on $G$. 

4. This is essentially proved in Lemma 4.11 of \cite{Achter_2023}.  We modify the proof to fit the present statement.

Given the previous point, we just need to prove the statement for one element of the double coset 
 $D_{e\mu_0, p}$.  We take
$g_0=e\mu_0(p)=\diag(q, \dots, q, 1\dots, 1)\in D_{e\mu_0, p}$, and let $x_0 = \pi_n(\gamma_0)\in \cM(\Z_p/p^n)$ be the truncation of $g_0$ modulo $p^n$.
We know that each fibre of the map $\pi_{G, n}$ has volume $p^{-n\dim(G)}$. 
Consider the left translation by $g_0$ map $L_{g_0}: G(\Z_p) \to \cM(\Z_p)$, given explicitly by: 
 \begin{equation*}
  \xymatrix@R-2pc{
    L_{g_0}: G(\Z_p) \ar[r]& \mathcal M(\Z_p)\\
{\left[\begin{smallmatrix} A &  B \\ C & D\end{smallmatrix}\right]} \ar@{|->}[r]&
{\left[\begin{smallmatrix} qA &  qB \\ C & D\end{smallmatrix}\right].}
  }
  \end{equation*}

 It induces a map of truncated matrices 
 \begin{equation*}
  \xymatrix@R-2pc{
    L_{g_0}^n : G(\Z/p^n) \ar[r]& \mathcal M(\Z_p/p^n)\\
{\left[\begin{smallmatrix} A &  B \\ C & D\end{smallmatrix}\right]} \ar@{|->}[r]&
{\left[\begin{smallmatrix} qA &  qB \\ C & D\end{smallmatrix}\right];}
  }
  \end{equation*}
  unlike $L_{g_0}$, this map is not injective. 

We observe that $L_{g_0}$ takes each fibre of the truncation $\pi_{G, n}$ to a subset of a fibre of $\pi_{\cM, n}$: 
for any $x\in G(\Z_p/p^n)$, we have that $L_{g_0}(\pi_{G, n}^{-1}(x)) \subset \pi_{\cM}^{-1}(L_{g_0}^n(x))$ is an open subset. Since the measure $|d\omega_G|$ is invariant under $L_{g_0}$, we just need to find the index of $L_{g_0}(\pi_{G, n}^{-1}(x))$ in  $\pi_{\cM}^{-1}(L_{g_0}^n(x))$. 

Note that if $\gamma\in \cM(\Z_p)\equiv g_0 \mod p^n$ then $g_0^{-1}\gamma \equiv I \mod p^{n-e}$, where $e=\val_p(q)$; in particular, $\gamma$ is in the image of a neighbourhood of $I$ under the map $L_{g_0}$ (when $n$ is large). 

For simplicity, we move the calculation to the Lie algebra $\fg$. 
Let $n\gg e$. 
We can write the truncated exponential approximation: $g_0^{-1}\gamma=I+X+\frac12X^2+\dots$ for 
some $X\in \fg(\Z_p)$; in particular,  there exists  $X\in\fg(\Z_p)$ such that $g_0^{-1}\gamma \equiv I+X \mod p^{2(n-e)}$, and thus, since the exponential is bijective on a small neighbourhood of zero in $\fg(\Z_p)$, our set 
$\pi_{\cM}^{-1}(L_{g_0}^n(x))/ L_{g_0}(\pi_{G, n}^{-1}(x)) $
 is in bijection with the set 
 $$\left\{X\in \fg(\Z_p) \mid X \equiv 0 \mod p^n \right\} / \left \{\left[\begin{smallmatrix} qI_g&  0 \\ 0 & I_g\end{smallmatrix}\right] X :  X \equiv 0 \mod p^n\right\}.
 $$
This set is a finite-dimensional vector space over $\F_p$, so we just need to find its dimension, which equals the dimension of the kernel of the multiplication by 
$\left[\begin{smallmatrix} qI_g&  0 \\ 0 & I_g\end{smallmatrix}\right]$ on $\fg(\Z_p/p^n)$. 

We have $\fg =\mathfrak{sp}_{2g}\oplus \mathfrak{z}$, where $\mathfrak{z}$ is the $1$-dimensional Lie algebra of the centre. It will be convenient to decompose it further: let $\mathfrak h$ be the Cartan subalgebra of $\mathfrak{sp}_{2g}$ consisting of diagonal matrices, and let $V$ consist of matrices whose diagonal entries are all zero; then
$$\fg=(\mathfrak{z}\oplus \mathfrak h) \oplus V.$$
Consider the action of multiplication by $\left[\begin{smallmatrix} qI_g&  0 \\ 0 & I_g\end{smallmatrix}\right]$ on each term of this direct sum decomposition. 

On the term $\mathfrak{z}\oplus \mathfrak h$ it acts by 
$\diag(a_1, \dots, a_{2g})\mapsto \diag (q a_1, \dots, q a_g, a_{g+1}, \dots, a_{2g})$, which in the 
$\mathfrak{z}\oplus \mathfrak h$-coordinates can be written as (recalling that $a_i+a_{g+i}=z$ is independent of $i$): 
$$\begin{aligned}
&\frac{z}2\oplus \left(\frac{z}2-a_{g+1}, \dots \frac{z}2-a_{2g}, -\frac{z}2+a_{g+1}, \dots -\frac{z}2+a_{2g}\right) \\ &\mapsto \frac{qz}2\oplus
\left(\frac{qz}2-\frac{(q+1)a_{g+1}}2, \dots \frac{qz}2-\frac{(q+1)a_{2g}}2, -\frac{qz}2+\frac{(q+1)a_{g+1}}2, \dots -\frac{qz}2+\frac{(q+1)a_{2g}}2\right).
\end{aligned}
$$
The only points $(z, a_{g+1}, \dots, a_{2g})$  that are killed $(\bmod p^n)$ by this map are of the form 
$(z', 0, \dots, 0)$ with $qz'=0$; there are $q$ of them.   

Next consider an element $X=\left[\begin{smallmatrix} A &  B \\ C & D\end{smallmatrix}\right] \in V$. Then $A$ is determined by $D$, and $B$ is skew-symmetric (up to a permutation of rows and columns).  Multiplication by  
$\left[\begin{smallmatrix} qI_g&  0 \\ 0 & I_g\end{smallmatrix}\right]$
scales each entry of  $A$ and $B$ by a factor of $q$, and does not change $C$ and $D$.
Since $A$ is determined by $D$, the elements $X$ killed by this map are in bijection with symmetric matrices $B$ 
with entries in $\Z_p/p^n$ that are killed by multiplication by $q$. 
Since the space of such matrices is a $g(g+1)/2$-dimensional 
linear space, the number of such matrices $B$ is $q^{g(g+1)/2}$. 

Therefore, every fibre of $\pi_{\cM, n}$ over a point of $\cM(\Z_p/p^n)$ is a disjoint union of $q^{g(g+1)/2+1}$ translates of the fibres of the map $\pi_{G, n}$, and hence its volume is 
$q^{g(g+1)/2+1}p^{n\dim(G)}$. 
\end{proof}

Given the Lemma, it is easy to finish the proof of the Proposition. 
First, recall that 
\begin{enumerate}
\item The stable orbit of $\gamma_0$ coincides with the rational orbit of $\gamma_0$. 
\item The intersection of the orbit of $\gamma_0$ with $\cM(\Z_p)$ is contained in the single double coset $D_{e\mu_0, p}$ \cite[Lemma 4.8]{Achter_2023}.
\end{enumerate}

By definition, (in the case $X$ not ordinary) we only count the $x\in \cM(\Z_p/p^n)$ that have a lift $\gamma \in \cM(\Z_p)$. By the congruence of characteristic polynomials, such a lift would have to be in a $p^{n'}$-neighbourhood of the stable orbit of $\gamma_0$ (with possibly $n'<n$), which coincides with the rational orbit in this case; and then by the statement (2) mentioned above, $\gamma$ would have to be congruent modulo $p^{n'}$ to an element of, and therefore contained in, the double coset $D_{e\mu_0}$.
Then by the last point of Lemma \ref{lemma:main}, the fibre $\pi_{\cM, n}^{-1}(x)$ has volume  $q^{g(g+1)/2+1}p^{-n\dim(G)}$ with respect to $|d\omega_G|$.
Therefore, 
$$\vol_{|d\omega_G|}(\pi_{\cM, n}^{-1}(\fc(\gamma_0)))= q^{g(g+1)/2+1}p^{-n\dim(G)} \#\{\gamma \in \cM(\Z/p^n) \vert \exists \tilde\gamma\in \cM(\Z_p): \tilde \gamma\equiv \gamma \mod p^n \}. $$  
We obtain, from \eqref{eq:serre_at_p} and \eqref{eq:volume_ratio}: 
\begin{equation}
\begin{aligned} 
\nu_p([X, \lambda])  &= 
q^{-\frac{g(g+1)}2-1}\frac{p^{\dim(G^\der)}}{\#G^\der(\integ_p/p)} 
\lim_{n\to \infty}\frac{\vol_{|d\omega_G|}(\pi_{\cM, n}^{-1}(\fc(\gamma_0))p^{-n\dim(G)}}{\vol_{|d\omega_A|}(U_n(\gamma_0)) q^{-1}} \\
&= q^{-\frac{g(g+1)}2}\frac{p^{\dim(G^\der)}}{\#G^\der(\integ_p/p)} O^\geom(\gamma_0),
\end{aligned}
\end{equation} 
which completes the proof.
\end{proof}

\newpage
\printbibliography
\end{document}